%% file: main_new.tex
\documentclass[11pt]{amsart}

\usepackage{txfonts}
\usepackage[T1]{fontenc}
\usepackage{fancyhdr}

\usepackage[utf8]{inputenc} 
\usepackage{geometry} 
\usepackage{booktabs} 
\usepackage{array} 
\usepackage{paralist} 
\usepackage{verbatim} 
\usepackage{tikz}
\usetikzlibrary{arrows}
\usetikzlibrary{decorations.markings}
\usepackage{subfig}
\usepackage{tabularx}
\usepackage{enumitem}
\usepackage{amsmath,amsthm,amscd}
\usepackage{amssymb,amsfonts}
\usepackage{fancyhdr} 
\usepackage{pgf,tikz}
\usepackage{caption}
\usepackage{tikz-cd}
\usepackage{tikzscale}
\usetikzlibrary{patterns}
\usepackage{rotating}
\usepackage{xcolor}
\usepackage{lscape}
\usepackage{xspace}
\usepackage{xfrac}
\usepackage{blindtext}
\usepackage{bm}
\usepackage[ruled,vlined]{algorithm2e}

\usetikzlibrary{decorations.markings}
\usetikzlibrary{patterns}
\usetikzlibrary{calc}
\usepackage{soul}

\usepackage{tikz}
        \usetikzlibrary{patterns,hobby}
        \usepackage{pgfplots}
        \pgfplotsset{compat=1.6}

\usepackage{faktor} 
\usepackage{pgf,tikz}
\usetikzlibrary{arrows.meta}
\usetikzlibrary{decorations.markings}
\usetikzlibrary{patterns}

\usepackage{color}   
\usepackage{hyperref}
\hypersetup{
    colorlinks=true, 
    linktoc=all,     
    linkcolor=blue,  
    citecolor=blue,
    filecolor=blue,
    urlcolor=blue
}

\usepackage[T1]{fontenc}
\usepackage{hyperref}
\usepackage{nameref}

\newcommand{\C}{\mathbb{C}}
\newcommand{\lcm}{\operatorname{lcm}}

\newcommand\PP{{\mathbb P}}
\newcommand\RR{{\mathcal R}}

\newcommand\RH{{\mathcal H}}
\newcommand\RC{{\mathcal C}}
\newcommand\RP{{\mathcal P}}

\newcommand\R{{\mathbb R}}

\newcommand{\pslr}{{\mathrm{PSL}_2 (\mathbb{R})}}
\newcommand{\pslz}{{\mathrm{PSL}_2 (\mathbb{Z})}}

\newcommand{\glplus}{\mathrm{GL}^+(2,\mathbb{R})}
\newcommand{\slz}{{\mathrm{SL}_2 (\mathbb{Z})}}

\newtheorem{thm}{Theorem}[section]
\newtheorem{cor}[thm]{Corollary}
\newtheorem{prop}[thm]{Proposition}
\newtheorem{lem}[thm]{Lemma}
\theoremstyle{definition}
\newtheorem{defn}[thm]{Definition}
\theoremstyle{remark}
\newtheorem{rmk}[thm]{Remark}
\theoremstyle{definition}

\theoremstyle{definition}

\theoremstyle{definition}

\numberwithin{equation}{section}
\title{One-dimensional strata of residueless meromorphic differentials}

\author{Myeongjae Lee} 
\address[Myeongjae Lee]{Mathematics Department, Stony Brook University, Stony Brook, NY 11794-3651, USA}
\email{myeongjae.lee@stonybrook.edu}

\author{Guillaume Tahar}
\address[Guillaume Tahar]{Beijing Institute of Mathematical Sciences and Applications, Huairou District, Beijing, China}
\email{guillaume.tahar@bimsa.cn}

\date{\today}

\begin{document}
\begin{abstract}
In projectivized strata of meromorphic $1$-forms on elliptic curves with only one zero, the locus of residueless differentials is a complex curve endowed with a canonical complex projective structure. Drawing on the multi-scale compactification of strata, we provide formulas to compute the genus of these curves and the degree of their natural forgetful map to $\mathcal{M}_{1,1}$. Additionally, we distinguish two non-hyperelliptic components in the residueless locus of the exceptional stratum $\mathcal{H}_{1}(12,-3,-3,-3,-3)$ and hence complete the classification of the connected components of these loci.
\end{abstract}
\maketitle
\setcounter{tocdepth}{1}
\tableofcontents

\section{Introduction}

For integers $b_1,\dots,b_p\geq 2$ and $a=b_1+\dots+b_p$, denote $\mu \coloneqq(a,-b_1,\dots,-b_p)$ the partition of 0. Recall that the stratum of meromorphic differentials $\RH_1(\mu)$ of genus one is defined to be the moduli space of meromorphic differentials on elliptic curves with orders of zeroes and poles prescribed by $\mu$. The stratum $\RH_1(\mu)$ has a $(p+1)$-dimensional complex orbifold structure given by the period coordinates.

The locus in $\RH_1(\mu)$ consisting of the differentials whose residue at each pole is equal to zero is denoted by $\RR_1(\mu)$. The stratum $\RR_1(\mu)$ may not be connected, and its connected components are classified in \cite{Lee}, except for the stratum $\RR_1(12,-3,-3,-3,-3)$. The codimension of $\RR_1(\mu)$ in $\RH_1(\mu)$ is equal to $p-1$, since the residues at the poles (up to the residue theorem) are part of the period coordinates system. Thus the projectivized residueless stratum $\PP\RR_1(\mu)$ is a one-dimensional complex orbifold. 

On the other hand, $\PP\RR_1(\mu)$ has a geometric smooth compactification $\PP\overline{\RR}_1(\mu)$ called the {\em multi-scale} compactification. Each connected component of $\PP\overline{\RR}_1(\mu)$ has a compact Riemann surface structure. Moreover, it has a complex projective structure given by the period coordinates. By investigating the complex projective structure, we obtain a cellular decomposition of $\PP\overline{\RR}_1(\mu)$. This allows us to classify the connected components of $\PP\overline{\RR}_1(\mu)$ and compute its Euler characteristic. In this paper, we complete the classification of the connected components in the case of $\RR_1(12,-3,-3,-3,-3)$. Also, we compute the Euler characteristic of $\PP\overline{\RR}_1(\mu)$ by enumerating the cells of the cellular decomposition determined by the complex projective structure. 

\subsection{Main results}

Recall that the residueless strata of genus one have {\em two} topological invariants: hyperellipticity (the existence of the hyperelliptic involution and the ramification profile, see Section~\ref{sub:CCRS}), and the rotation number (an integer divisor of $d\coloneqq \gcd(b_1,\dots,b_p)$). In \cite{Lee}, the first author proved that any residueless stratum of genus one, except for $\RR_1(12,-3,-3,-3,-3)$, has a unique non-hyperelliptic component with any given rotation number. We provide the classification of non-hyperelliptic components for this last exceptional stratum.

\begin{thm}\label{main1}
The residueless stratum $\RR_1(12,-3,-3,-3,-3)$ has six connected components. More precisely, it has:
\begin{itemize}
    \item {\em three} hyperelliptic components;
    \item {\em one} non-hyperelliptic component with rotation number $1$;
    \item {\em two} non-hyperelliptic components with rotation number $3$. These two components are isomorphic (they have the same canonical complex projective structure and just differ by the labeling of the poles).
\end{itemize}
\end{thm}

Using the complex projective structure of $\PP\overline{\RR}_1(\mu)$, we compute its Euler characteristic in terms of the following combinatorial quantities: $a=\sum_{i} b_{i}$, $A=\sum_{i} (b_{i})^2$ and $B=\prod_{i} (b_i-1)$.

Orbifold points can only appear when $p \leq 2$. We first give formulas for the case $p=2$ (the case $p=1$ has already been described in \cite{Ta1}).

\begin{thm}\label{main2}
The (orbifold) Euler characteristic $\chi^{orb}\left(\PP\overline{\RR}_1(a,-b_{1},-b_{2})\right)$ of the multi-scale compactification $\PP\overline{\RR}_1(a,-b_{1},-b_{2})$ is given by
\begin{equation}
-\frac{(b_1-1)(b_2-1)}{12}(b_1 b_2+2b_{1}+2b_{2}+2)
 + \frac{1}{2}\left[
\displaystyle\sum_{\substack{1 \leq c_{1} \leq b_{1}-1 \\ 1 \leq c_{2} \leq b_{2}-1\phantom{-}}}  \gcd \left( c_{1}+c_{2}, a \right)
    +  \sum\limits_{c_{1}=1}^{b_{1}-1}  \gcd \left( c_{1}, b_{1} \right)
    + \sum\limits_{c_{2}=1}^{b_{2}-1}  \gcd \left( c_{2}, b_{2} \right)\right] +\varepsilon
\end{equation}
where the correction term $\varepsilon$ is:
\begin{itemize}
    \item $\frac{5}{2}$ if $b_{1},b_{2} \in 4\mathbb{N}^{\ast}$;
    \item $\frac{3}{2}$ if $b_{1} \in 4\mathbb{N}^{\ast}$ and $b_{2} \in 4\mathbb{N}+2$;
    \item $\frac{1}{2}$ otherwise.
\end{itemize}

$\PP\overline{\RR}_1(a,-b_{1},-b_{2})$ contains one orbifold point of order 3 in the connected component of rotation number $\gcd(b_1,b_2,\frac{b_1-b_2}{3})$ if $b_{1},b_{2} \in 3\mathbb{N}^{\ast}$ or $b_{1},b_{2} \in 3\mathbb{N}+2$ (and no orbifold point otherwise).
\end{thm}

When $p \geq 3$, the multi-scale compactification $\PP\overline{\RR}_1(\mu)$ is a (possibly disconnected) compact Riemann surface with no orbifold points. 

\begin{thm}\label{main2+}
When $p \geq 3$, the Euler characteristic of the stratum $\PP\RR_1(\mu)$ is given by  
\begin{equation}
   \chi\left(\PP\RR_1(\mu)\right)=  \frac{p!B}{48}(A-a^2-4a+20)
\end{equation}
where $a=\sum_{i} b_{i}$, $A=\sum_{i} (b_{i})^2$ and $B=\prod_{i} (b_i-1)$.\newline
Also, the Euler characteristic of the multi-scale compactification $\PP\overline{\RR}_1(\mu)$ is given by 

\begin{equation}
   \chi\left(\PP\overline{\RR}_1(\mu)\right)=  \chi \left( \PP\RR_1(\mu) \right)+\sum_{t=1}^p \sum_{\tau\in S_{p}}\sum_{C_i=1}^{b_i-1} \frac{1}{2t} \gcd \left( \sum_{i=1}^t C_{\tau(i)},\sum_{i=1}^t b_{\tau(i)} \right)+\varepsilon
\end{equation}
where the correction term $\varepsilon$ is:
\begin{itemize}
    \item $2$ if $p=3$ and $b_{1},b_{2},b_{3} \in 4\mathbb{N}^{\ast}$;
    \item $2$ if $p=3$ and $b_{1},b_{2},b_{3} \in 4\mathbb{N}+2$;
    \item $3$ if $p=3$, $b_{1} \in 4\mathbb{N}^{\ast}$, $b_{2} \in 4\mathbb{N}+2$ and $b_{3} \in 2\mathbb{N}^{\ast}$;
    \item $0$ otherwise.
\end{itemize}
\end{thm}

The complicated term of $\chi\left(\PP\overline{\RR}_1(\mu)\right)- \chi \left( \PP\RR_1(\mu) \right)$ in the formula counts the number of non-cusp points in the boundary $\partial \PP\overline{\RR}_1(\mu)$.

By forgetting the differential except for the location of the unique zero, we obtain a forgetful map $\PP\RR_1(\mu)\to \mathcal{M}_{1,1}$. This extends to the map $\PP\overline{\RR}_1(\mu)\to \overline{\mathcal{M}}_{1,1}$ between the compactifications. This map is a ramified covering of compact complex orbifolds of dimension one. We compute the degree of the forgetful map $\PP\RR_1(\mu)\to \mathcal{M}_{1,1}$. This counting can be refined by restriction to the connected components of given rotation number.

\begin{thm}\label{main3}
The degree of the forgetful map $\pi:\PP\RR_1(\mu)\to \mathcal{M}_{1,1}$ is equal to $(a+1) p!B.$

Fix some integer $r$ satisfying $r~\vert~d$. Let $\RC_r\subset \PP\overline{\RR}_1(\mu)$ be the union of all connected components with rotation number $r$. Then the degree of the forgetful map $\pi$ restricted to $\RC_r$ is equal to 
\begin{equation}
    \frac{1}{r^2} p![(a+1)B-(-1)^p]\prod_{q} \left(1-\frac{1}{q^2}\right)
\end{equation}
where $q$ runs over all prime divisors of $\frac{d}{r}$.
\end{thm}

\subsection{Relations with other works}

\subsubsection{Modular curves}

When $p=1$, a stratum $\RH_{g}(\mu)$ coincides with its residueless locus $\RR_1(\mu)$. In \cite{Ta1} appears a complete description of strata $\RH_{1}(a,-a)$. The present work is a direct generalization of these results.
\par
Following Abel-Jacobi relations, a differential of $\RH_{1}(a,-a)$ is completely determined (up to a scaling) by the complex structure of the underlying elliptic curve and the choice of a torsion point of order $a$ (or any divisor of $a$). More precisely, the connected component of rotation number $1$ in $\PP\RH_1(a,-a)$ is parametrized by the modular curve $X_{1}(a)$. Connected components of residuless loci  $\RR_1(\mu)$ are complex curves that can be considered as generalized modular curves parametrizing some rigid configurations of points on elliptic curves.

\subsubsection{Generalized strata}\label{subsub:GS}

For a given stratum $\RH_{g}(\mu)$, a partition $R: \boldsymbol{p}=P_{1} \sqcup \dots \sqcup P_{k}$ of the set of (labeled) poles defines a locus called \textit{generalized stratum} $\RH_{g}^R(\mu)$ by requiring the sum of the residues of poles in each part $P_i$ to be zero. In the multi-scale compactification of strata, the boundary is made of products of generalized strata and some combinatorial data called prong-matchings (see \cite{BCGGM}). A description of generalized strata is therefore necessary to fully understand the irreducible components of the boundary divisor and the boundary complex of $\overline{\RH}_g(\mu)$. Since the multi-scale compactification of strata is a smooth compactification with normal crossings boundary, the understanding of the boundary complex should lead to a combinatorial description of the top-weight cohomology ring of each stratum.
\par
For example, the top weight cohomology of moduli spaces with normal crossings boundary can be studied by analyzing the boundary complex, which is done for $\mathcal{M}_g$ in \cite{CGP}, and for $\mathcal{A}_g$ in \cite{BBCMMW}. The residueless strata are special cases of the generalized strata, where the partition is the finest possible. The connected components of the residueless strata for any genus $g$ are classified in \cite{Lee}. In the present paper, we compute the topological invariants of one-dimensional residueless strata, using the description of the boundary. This may reflect a small part of this general project.\newline

\subsubsection{Lamé functions}

On an elliptic curve, a Lamé function of degree $m \geq 0$ is a solution of \textit{Lamé equation}:
$$
\frac{d^{2}W}{dz^{2}}-(m(m+1) \wp(z) + \lambda)W
.
$$
Here $\lambda$ is the so-called \textit{auxillary parameter} and satisfies an algebraic equation involving the degree $m$ and the invariants $g_{2},g_{3}$ of the elliptic curve.
\par
Lamé equation is an ubiquitous equation of mathematical physics. It first appeared by applying the separation of variables methods to the Laplace equation in an elliptic system of coordinates. It is also naturally a one-dimensional Schr\"{o}dinger equation with periodic potential. More recently, it appeared in inhomogeneous cosmological models (see the introduction of \cite{Ma} for background on the applications of Lamé equation in mathematical physics).
\par
In \cite{EGMP}, Lamé functions of degree $m$ are given a geometric interpretation as developing maps of translation surfaces of genus one with one zero of order $2m$ and $m$ residueless double poles. In other words, $\RR_1(2m,-2,\dots,-2)$ is also the moduli space of Lamé functions of degree $m$. The classification of connected components of these loci (for any $m \geq 1$) and their genera are computed in Theorem~1.1 of \cite{EGMP}. Our work generalizes these results to any partition $\mu$.\newline
In the description of $\PP\RR_1(\mu)$, the authors of \cite{EGMP} decompose these surfaces along curves called Lin-Wang curves. The latter have been introduced in the study of Mean Field Equations of the form $\Delta u +e^{u} = \rho \delta_{0}$ (see \cite{CCW}). Lin-Wang curves appear as a special case of our equatorial arcs introduced in Section~\ref{sec:Cellular}.

\begin{rmk}
Formulas of Theorems~\ref{main2} and $\ref{main3}$ may differ from those obtained in \cite{EGMP} because poles are labeled in the present paper.
\end{rmk}

\subsubsection{Isoresidual fibration}

The \textit{isoresidual fibration} is defined by the map sending a differential to the vector of its residues at the (labelled) poles. In the case of a stratum $\RH_{0}(\mu)$ of differentials on $\mathbb{CP}^{1}$, singular fibers belong to a complex hyperplane arrangement called the \textit{resonance arrangement}, see \cite{GT1} for details. The residueless locus of any stratum is expected to be the most singular isoresidual fiber. Since generalized strata (see Section~\ref{subsub:GS}) are loci in strata where some sums of residues vanish, generalized strata appear as preimages of intersections of resonance hyperplanes under the isoresidual fibration.
\par
Isoresidual fibers are complex manifolds (or orbifolds) of dimension $2g+n-1$ where $n$ is the number of zeroes in partition $\mu$. In \cite{GT1}, an extensive study of the case where $g=0$ and $n=1$ has been carried out. When $g=0$ and $n>1$, isoresidual fibers inherit a canonical \textit{translation structure} from the period coordinates. The case $n=2$ where generic isoresidual fibers are translation surfaces is described in \cite{CGPT}.
\par
A similar structure can be defined on the \textit{isoperiodic leaves} of any stratum of meromorphic $1$-forms (see \cite{CDF} for background on the isoperiodic foliation). A complete description of the translation structure of isoperiodic leaves of $\RH_1(1,1,-2)$ is given in \cite{FTZ}.

\subsection{Organization of the paper}

\begin{itemize}
    \item In Section~\ref{sec:TranslationSurfaces}, we provide background on the geometry of translation surfaces and their qualitative invariants (rotation number, hyperelliptic involutions, topology of the core).
    \item In Section~\ref{sec:Moduli}, we introduce the period atlas on strata of meromorphic differentials and the induced complex projective structure on the residueless locus.
    \item In Section~\ref{sec:Cellular}, we define the  equatorial net and describe the cellular decomposition of strata $\PP\RR_1(\mu)$. We deduce the classification of orbifold points appearing in these loci.
    \item In Section~\ref{sec:Multiscale}, we describe the multi-scale compactification of $\PP\RR_1(\mu)$ in terms of nodal curves.
    \item In Section~\ref{sec:Counting}, we count systematically cells of each dimension in the cellular decomposition of $\PP\RR_1(\mu)$ induced by the equatorial net. Deducing the Euler characteristic, we prove Theorems~\ref{main2} and~\ref{main2+}.
    \item In Section~\ref{sec:Degree}, we compute the degree of the forgetful map from $\pi:\PP\RR_1(\mu)\to \mathcal{M}_{1,1}$ and prove Theorem~\ref{main3}.
    \item In Section~\ref{sec:Except}, we draw on the combinatorics of equatorial nets to prove Theorem~\ref{main1}.\newline
\end{itemize}

\paragraph{\bf Acknowledgements.} The first author would like to thank his advisor, Samuel Grushevsky for introducing him to the theory of moduli spaces of differentials and many valuable discussions. Research by the second author is supported by the Beijing Natural Science Foundation IS23005 and the French National Research Agency under the project TIGerS (ANR-24-CE40-3604). The authors are grateful to the Department of Mathematics and Computer Science of the Weizmann Institute of Science for the hospitality in March 2022 when this project was initiated. The authors would also like to thank the anonymous reviewers for their valuable remarks.

\section{Geometry of translation surfaces}\label{sec:TranslationSurfaces}

\subsection{Translation structures}

For a pair $(X,\omega)$ where $X$ is a compact Riemann surface and $\omega$ is a meromorphic $1$-form, we denote by $X^{\ast}$ the complement of the set of zeroes and poles of $\omega$. In the punctured surface $X^{\ast}$, each local primitive of $\omega$ is locally injective.
\par
Injective primitives of $\omega$ form a family of complex-valued \textit{local coordinates} on $X^{\ast}$. On the overlap between their domains of definition, two such local coordinates differ by a constant. Transition maps are \textit{translations} of the complex plane $\mathbb{C}$. In other words, $\omega$ defines a \textit{translation structure} on $X^{\ast}$.
\par
Every translation-invariant geometric property of the complex plane is well-defined in a translation surface. In particular, any translation surface is endowed with a \textit{flat metric}. For close enough points $A,B$ in $X^{\ast}$, the distance according to this flat metric is $\vert \int\limits_{A}^{B} \omega \vert$. More generally, we define for every arc $\gamma$ of $X^{\ast}$ its \textit{period} $P_{\omega}(\gamma)=\int\limits_{\gamma} \omega$. Period $P_{\omega}(\gamma)$ only depends on the homotopy class of $\gamma$ (with fixed endpoints).
\par
In this paper, we will refer to $(X,\omega)$ as a \textit{translation surface}.

\subsection{Local model of singularities}

At a zero of order $k$, a meromorphic $1$-form is locally biholomorphic to $z^{k}dz$. Therefore, a neighborhood of the zero is a cyclic ramified cover of order $k+1$ of the punctured disk. In terms of translation structures, a zero of order $k$ is thus a \textit{conical singularity} of angle $2(k+1)\pi$.

Now we consider poles of order $k \geq 2$ with zero residue (indeed, the residue at simple poles is always nonzero). We start with the case of a double pole. Up to a change of variable $h=z^{-1}$, a neighborhood of a double pole with zero residue is biholomorphic to a neighborhood of the point at infinity in the compactified standard flat plane $(\mathbb{C},dz)$.

For the general case of residueless poles of order $k \geq 3$, a neighborhood of the pole is a cyclic ramified cover of order $k-1$ of a neighborhood of a residueless double pole. In terms of translation structures, it amounts to the geometry at infinity of an infinite flat cone of angle $2(k-1)\pi$.

\begin{rmk}
In this paper, we are mostly interested in the residueless locus of strata of meromorphic $1$-forms. Poles with nonzero residues only appear in irreducible components (of genus zero) of nodal curves forming the boundary in the multi-scale compactification of strata. These latter multi-scale differentials play a purely combinatorial role and we do not need to give a precise geometric interpretation of their poles.
\end{rmk}

\subsection{Volume}\label{sub:volume}

Consider a residueless meromorphic differential $\omega$ on an elliptic curve such that $\omega$ has only one zero. Then the period $P_{\omega}(\gamma)$ of a closed loop $\gamma$ avoiding every singularity of $\omega$ only depends on the class of $\gamma$ in the absolute homology group $H^{1}(\mathbb{T}^{2})$ of an unpunctured torus.
\par
For a symplectic basis $\alpha,\beta$ of $H^{1}(\mathbb{T}^{2})$, $\mathfrak{Im}(\overline{P_{\omega}(\alpha)}P_{\omega}(\beta))$ is called the \textit{volume} of $\omega$ and denoted by $Vol_{\omega}$. It can be checked directly that it does not depend on the choice of a specific symplectic basis. The volume of a differential $\omega$ can be
positive, negative or zero (see \cite{CF} for a discussion of volume in any genus).

\subsection{Rotation number of translation structures}\label{sub:rotation}

Since directions are preserved by translations, the usual winding number can be generalized to translation surfaces.

\begin{defn}
Let $\gamma$ be a smooth oriented loop in a translation surface punctured at the singularities. Assuming that $\gamma$ is parametrized by arc-length, $\frac{\gamma'(t)}{|\gamma'(t)|}$ is a continuous map from $\mathbb{S}^{1}$ to $\mathbb{S}^{1}$. As such, it has a well-defined \textit{topological index} $Ind_{\gamma}$ which depends only on the homotopy class of $\gamma$.
\end{defn}

The topological index of a positively oriented loop around a singularity of order $a$ is clearly $a+1$.\newline

We can then construct a global flat invariant for any translation surface of genus one.

\begin{defn}
We consider a translation surface $(X,\omega)$ where $X$ is an elliptic curve and $\omega$ is a meromorphic $1$-form on $E$. Let $c$ be the greatest common divisor of the orders of the singularities of $\omega$ and $\alpha,\beta$ a pair of closed simple loops of $X^{\ast}$ forming a symplectic basis of the homology of the unpunctured torus.
\par
The \textit{rotation number} of $(X,\omega)$ is defined as $\gcd(Ind_{\alpha},Ind_{\beta},c)$. In particular, it does not depend on the choice of a specific symplectic basis $(\alpha,\beta)$.
\end{defn}

\subsection{Hyperelliptic involutions}

Given a translation surface $(X,\omega)$, a \textit{hyperelliptic involution} $\sigma$ is a biholomorphic involutive automorphism of $X$ such that $X/\sigma \cong \mathbb{CP}^{1}$ and $\sigma^{\ast}\omega=-\omega$.
\par
Any hyperelliptic involution defines, in particular, an \textit{induced permutation} acting on the finite set of (labeled) singularities of $\omega$. An induced permutation is characterized by its fixed points and its pairs of conjugated elements.
\par
A translation surface $(X,\omega)$ is \textit{hyperelliptic} if it admits at least one hyperelliptic involution. A connected component $\RC$ of a stratum $\RR_1(\mu)$ is \textit{hyperelliptic} if every translation surface in $\RC$ is hyperelliptic. 

\subsection{Saddle connections and closed geodesics}

In a translation surface, the differential defines a period for every homotopy class of topological arcs in the punctured surface. However, the flat metric defines geodesic representatives in these classes and they are very useful to describe the geometry of the surface.

\begin{defn} A \textit{saddle connection} is a geodesic segment joining two conical singularities of the translation surface such that all interior points are regular points. A saddle connection is \textit{closed} if its two end points coincide.
\end{defn}

The period of any saddle connection depends only on its relative
homology class in the first relative homology group $H_{1}(X \setminus \lbrace{ poles \rbrace}~\vert~\lbrace{ zeroes \rbrace},\mathbb{Z})$.\newline

The curve-shortening flow applies in translation surfaces as well as on the flat plane. It can be used to prove that every homotopy class of simple loops on a translation surface punctured at the singularities has a geodesic representative. This representative is either a regular closed geodesic or a broken geodesic formed by saddle connections.\newline

A small enough neighborhood of a closed geodesic passing only through regular points is also foliated by closed geodesics with same length and direction (since they are homotopic, they have the same period in the translation structure). Therefore, closed geodesics belong to \textit{cylinders} bounded by saddle connections (the foliation can be extended until it reaches the first singularity).

\subsection{Core}

In a translation structure defined by a meromorphic differential, the \textit{core} is the convex hull of the conical singularities for the flat metric induced by the translation structure. As such, it contains every saddle connection.

\begin{defn}
A subset $M$ of a translation surface is \textit{convex} if every geodesic segment between two points of $M$ in included in $M$. The core $C(X)$ of a translation surface $(X,\phi)$ is the convex hull of its conical singularities (the zeroes of $\phi$). $IC(X)$ is the interior of $C(X)$ in $X$ and $\partial C(X)= C(X) \setminus IC(X)$ is its boundary. The core is said to be degenerate if $IC(X)= \emptyset$. In this case, $C(X)$ is just an embedded graph.
\end{defn}

The general results about the core of a translation surface are summarized into the following statements, see Section~4 in \cite{Ta} for more details.

\begin{prop}
For any translation structure induced by a meromorphic $1$-form on a surface $X$, the surface punctured at the poles retracts to its core. Each connected component of the complement of $C(X)$ is a topological disk containing a unique pole. Conversely, every pole belongs to such a component called \textit{polar domain}. These domains are bounded by finitely many saddle connections. The union of these saddle connections is $\partial C(X)$.
\end{prop}

Let $\alpha_{1},\dots,\alpha_{p}$ be the number of boundary saddle connections of the polar domain of a meromorphic $1$-form in $\mathcal{H}_{1}(a,-b_{1},\dots,-b_{p})$. We denote by $T$ the number of triangles in any geodesic triangulation of the interior of the core and by $|A|$ the maximal number of non-crossing saddle connections that can be drawn. The following lemma follows from the Euler characteristic formula on the torus.

\begin{lem}\label{lem:coreformula}
For any translation surface of $\mathcal{H}_{1}(a,-b_{1},\dots,-b_{p})$, the following formulas hold:
\begin{itemize}
\item $|A|=p+1+T$ (Euler characteristic);
\item $2|A|=3T+\sum \alpha_{j}$ (counting of edges);
\item $2p+2=T+\sum \alpha_{j}$.
\end{itemize}
\end{lem} 

The condition that a translation structure is induced by a residueless differential implies strong limitations on the shape of the core. Indeed, a polar domain whose pole is residueless cannot be bounded by only one saddle connection because this saddle connection would be closed and would have a period equal to zero. Consequently, we have $\alpha_{1},\dots,\alpha_{p} \geq 2$. Lemma~\ref{lem:coreformula} then implies that $T \leq 2$.

\begin{cor}\label{cor:core}
For any translation surface of $\mathcal{R}_{1}(a,-b_{1},\dots,-b_{p})$, the interior of the core can be:
\begin{itemize}
\item Empty (if $T=0$);
\item A triangle (if $T=1$);
\item A quadrilateral, a pair of disjoint triangles or a flat cylinder (if $T=2$).
\end{itemize}
\end{cor}

\subsection{Translation surfaces with nontrivial automorphisms}

Preserving the qualitative features of the core leads to strong constraints on symmetries of residueless differentials.

\begin{prop}\label{prop:symmetries}
If a translation surface $(X,\omega)$ of $\mathcal{R}_{1}(a,-b_{1},\dots,-b_{p})$ has a nontrivial automorphism $\sigma$, then one of the following statements holds:
\begin{enumerate}
    \item $\sigma^{\ast} \omega = - \omega$ and $\sigma^{2}=Id$ (hyperelliptic automorphism);
    \item $\sigma^{3}=Id$ and $p \in \lbrace{ 1,2 \rbrace}$;
    \item $\sigma^{4}=Id$ and $p=1$.
\end{enumerate}
\end{prop}

\begin{proof}
The automorphism $\sigma$ preserves the translation structure so it must fix the (labeled) poles. Moreover, for a given polar domain, it should act as a nontrivial permutation of the boundary saddle connections (otherwise $\sigma$ fixes all the points of the polar domain and thus coincides with $Id$ on the whole surface).

Following Lemma~\ref{lem:coreformula} and Corollary~\ref{cor:core}, the number of sides of any polar domain is either $2$, $3$ or $4$. If one polar domain has two sides, then $\sigma$ permutes the two boundary saddle connections and $\sigma^{2}=Id$. As the two boundary saddle connections are homologous, we have $\sigma^{\ast} \omega = -\omega$ on the polar domain (and thus the whole translation surface).

If no polar domain is bounded by two saddle connections, then equation $2p+2=T+\sum \alpha_{j}$ holds only in three cases:
\begin{itemize}
    \item $T=1$, $p=1$ and the unique polar domain has three boundary saddle connections;
    \item $T=0$, $p=2$ and the two polar domains have three boundary saddle connections;
    \item $T=0$, $p=1$ and the unique polar domain has four boundary saddle connections.
\end{itemize}
\end{proof}

\section{Period coordinates on moduli spaces}\label{sec:Moduli}

\subsection{Strata of translation surfaces}

For integers $a_{1},\dots,a_{n} \geq 1$ and $b_{1},\dots,b_{p} \geq 1$ satisfying $\sum a_{i} - \sum b_{j}=2g-2$ for some genus $g \geq 0$, denote $\mu\coloneqq (a_{1},\dots,a_{n},-b_{1},\dots,-b_{p})$. The stratum of meromorphic differentials $\RH_{g}(\mu)$ is defined to be the moduli space of meromorphic differentials on compact Riemann surfaces of genus $g$ with orders of zeroes and poles prescribed by $\mu$ (up to biholomorphism).

\subsection{Period coordinates}

For any point $(X,\omega)\in \RH_{g}(\mu)$, the integration of the differential $\omega$ along various paths determines a representation of relative homology group $H_{1}(X \setminus \lbrace{ poles \rbrace}, \lbrace{ zeroes \rbrace};\mathbb{Z})$ into $(\mathbb{C},+)$.

For a symplectic basis $\{\beta_1,\alpha_1,\dots,\beta_g,\alpha_g\}$ of $\textnormal{H}_1(X;\mathbb Z)\cong \mathbb Z^{2g}$, the integration of $\omega$ along these curves provides a system of $2g$ local coordinates called \textit{absolute periods} and the residues computed around the poles of $\omega$ provide additional $p-1$ local coordinates on the stratum $\RH_{g}(\mu)$. The integration of $\omega$ along $n-1$ arcs joining the zeroes of $\omega$ provides the last $n-1$ local coordinates (the so-called \textit{relative periods}). These $2g+n+p-2$ local coordinates completely determine the translation structure of $(X,\omega)$. They form the \textit{period atlas} of the stratum $\RH_{g}(\mu)$ and endow it with a structure of a complex-analytic orbifold of dimension $2g+n+p-2$.

\subsection{Action of $\glplus$}

The moduli space of translation surfaces features a natural action of $\glplus$ by postcomposition in the period coordinate charts. In particular, it acts on the periods of the differential inducing the translation structure.\newline
The action does not preserve angles but it preserves the topological indices and the orders of the singularities. Therefore, it preserves strata. Moreover, the zero residue condition is also preserved so there is an action of $\glplus$ on the locus of residueless differentials. In particular, in the stratum $\RH_1(a,-b_{1},\dots,-b_{p})$ with $b_{1},\dots,b_{p} \geq 2$, the residueless locus $\RR_1(a,-b_{1},\dots,-b_{p})$ is a $\glplus$-invariant closed $2$-dimensional subvariety defined by linear equations in the period coordinates. It is an example of a \textit{Teichm\"{u}ller curve in a meromorphic stratum} according to the definition introduced in \cite{MoMu}.\newline

Scalings form a subgroup of $\glplus$ so the quotient group $\pslr$ acts on projectivized fibers.\newline

The \textit{Veech group} of a translation surface is the stabilizer of the action $\glplus$ on it. Veech groups of translation surfaces induced by meromorphic $1$-forms have been classified in \cite{Ta2}.

\subsection{Classification of connected components of strata}

In \cite{Bo}, Boissy provided a classification of connected components of strata of meromorphic differentials. In genus one, the components are classified in terms of rotation numbers (see Section~\ref{sub:rotation}).

\begin{thm}\label{thm:Boissy}
Let $\mathcal{H}=\mathcal{H}_{1}(a_{1},\dots,a_{n},-b_{1},\dots,-b_{p})$ be a stratum of translation surfaces of genus one. We denote by $c$ the greatest common divisor of the orders of the singularities.
\par
For each divisor $\rho$ of $c$ (excepted $c$ itself in the case $n=p=1$), there is exactly one connected component $\mathcal{H}^{\rho}$ of the stratum containing those translation surfaces whose rotation number is equal to $\rho$.
\par
The stratum $\mathcal{H}$ is the disjoint union of the connected components $\mathcal{H}^{\rho}$.
\end{thm}

It has been proved in Theorem~1.1 of \cite{GT} that in each connected component of a stratum in genus $g \geq 1$, the residueless locus is nonempty.

\subsection{Classification of connected components of residueless strata}\label{sub:CCRS}
    
The connected components of the stratum $\RR_1(\mu)$ of residueless meromorphic differentials are classified by Theorem~1.4 and Theorem~1.7 in \cite{Lee}. In addition to the rotation number, we need one more topological invariant: the permutation induced by the hyperelliptic involution. 

We recall the definition of the stratum of hyperelliptic type and the ramification profile from \cite{Lee}, simplified for the single-zero stratum of genus one. 

\begin{defn}
A stratum $\RR_1 (\mu)$ is said to be of {\em hyperelliptic type} if there exists an involution $\RP\in S_{p}$ satisfying $b_i=b_{\RP(i)}$ for each $i$, $\RP$ fixes at most three numbers and for each such fixed $j$, $b_j$ is even. 

If $\mu$ is of hyperelliptic type, the involution $\RP$ that exhibits $\mu$ as being of hyperelliptic type is called a {\em ramification profile} of $\RR_1 (\mu)$. 
\end{defn}

\begin{thm} \label{hyper}
A stratum $\R_1 R(\mu)$ of residueless meromorphic differentials has a hyperelliptic component if and only if the stratum is of hyperelliptic type. There is a one-to-one correspondence between the hyperelliptic components and the ramification profiles of $\RR_1 (\mu)$.
\end{thm}

\begin{rmk}  \label{hyperremark}
    If $(X,\omega)\in \RR_1 (\mu)$ is hyperelliptic, then its hyperelliptic involution $\sigma$ must fix the unique zero of $(X,\omega)$ and induces a ramification profile of $\RR(\mu)$. Therefore, $(X,\omega)$ is contained in the hyperelliptic connected component. In other words, only hyperelliptic components contain the hyperelliptic translation surfaces. Note that this is not true in general if there is more than one zero. 
\end{rmk}

\begin{thm} \label{nonhyper}
Denote $d \coloneqq \gcd(b_1,\dots, b_p)$ and let $r$ be a positive integer divisor of $d$. Then the stratum $\RR_1 (\mu)$ has a unique non-hyperelliptic connected component $\RC_r$ with rotation number $r$, except for the following cases:
\begin{itemize}
    \item If $\mu=(2p,-2^p)$, then the stratum $\RR_1(\mu)$ does not have any non-hyperelliptic component;
    \item If $\mu=(r,-r)$, $\mu=(2r,-2r)$ or $\mu=(2r,-r,-r)$, then the stratum $\RR_1(\mu)$ does not have any non-hyperelliptic component of rotation number $r$;
    \item If $\mu=(12,-3^4)$ and $r=3$, then $\RR_1(\mu)$ has two non-hyperelliptic components of rotation number $3$ (classified in Theorem~\ref{main1}).
\end{itemize}
\end{thm}

\subsection{Complex projective structure on residueless loci}

A neighborhood of a translation surface $(X,\omega)$ in the residueless stratum $\RR_1(\mu)$ can be described by the period coordinates. Let $\alpha,\beta$ be two simple closed curves in $X$ that form a symplectic basis of $H_1(X,\mathbb{Z})$. Then the period coordinates can be given by $\left(\int_{\alpha} \omega, \int_{\beta} \omega\right)\subset \mathbb{C}^2$. Therefore the local neighborhood of $[(X,\omega)]$ in the projectivized stratum $\PP\RR_1(\mu)$ can be parameterized by $\left[\int_{\alpha} \omega, \int_{\beta} \omega\right]\in \mathbb{CP}^1$. 

If we choose another symplectic basis $\{\alpha',\beta'\}$ of $X$, then the two basis $\{[\alpha],[\beta]\}$ and $\{[\alpha'],[\beta']\}$ are related by $\slz$-action. That is, there exists $A\in \slz$ such that $A\circ \left([\alpha],[\beta]\right) = \left([\alpha'],[\beta']\right)$. So the transition function between the two charts $\left[\int_{\alpha} \omega, \int_{\beta} \omega\right]$ and $\left[\int_{\alpha'} \omega, \int_{\beta'} \omega\right]$ is given by the induced $\pslz$-action on $\mathbb{CP}^1$. Therefore, $\PP\RR_1(\mu)$ is endowed with a complex projective structure with real monodromy. 

\section{Cellular decomposition}\label{sec:Cellular}

\subsection{Equatorial net}

As a group of complex projective transformations acting on $\mathbb{CP}^{1}$, $\pslr$ preserves the real equator $\mathbb{RP}^{1}$ and each hemisphere.

\begin{defn}
In a Riemann surface endowed with a complex projective structure with real monodromy, the \textit{equatorial net} is the preimage of $\mathbb{RP}^{1} \subset \mathbb{CP}^{1}$ in each chart of the atlas. It is an embedded graph formed by circular arcs. Each of these arcs is called an \textit{equatorial arc}.\newline
Connected components of the complement of the equatorial net are called the \textit{cells} of the cellular decomposition of the Riemann surface. Each cell is bounded by finitely many \textit{boundary arcs} that are edges of the equatorial net.
\end{defn}

In the complex projective structure defined by the period atlas, the equatorial net of $\mathbb{P}\mathcal{R}_{1}(\mu)$ is formed by differentials whose absolute periods are real collinear (they do not generate a lattice of $\mathbb{C}$).

The sign of the volume of differentials of a given cell depends on whether the projective coordinate belongs to the lower or the upper half-hemisphere (see Section~\ref{sub:volume}).

\subsection{Classification of cells}

As they are parametrized by an open hemisphere of $\mathbb{CP}^{1}$, each cell is at the same time a $\glplus$-orbit in $\mathbb{P}\mathcal{R}_{1}(\mu)$. It follows that for any surface $X$ in a given cell, the topological pair $(X,C(X))$ formed by the embedding of the core $C(X)$ in the surface $X$ is the same up to homeomorphism. The possible pairs $(X,C(X))$ have been classified in Corollary~\ref{cor:core} in terms of the number $T$ of triangles in any triangulation of the interior of $C(X)$. We deduce the corresponding typology for cells of $\PP\RR_1(\mu)$.\newline
A cell is of:
\begin{itemize}
    \item \textit{"degenerate type"} if $C(X)$ has empty interior ($T=0$);
    \item \textit{"1-triangle type"} if the interior of $C(X)$ is a triangle ($T=1$);
    \item \textit{"quadrilateral type"} if the interior of $C(X)$ is a topological quadrilateral ($T=2$);
    \item \textit{"cylinder type"} if the interior of $C(X)$ is a topological cylinder ($T=2$);
    \item \textit{"2-triangle type"} if the interior of $C(X)$ is formed by two disjoint triangles ($T=2$).
\end{itemize}
Moreover, we will split cells of "degenerate type" into two subclasses. A cell will be of \textit{"3-degenerate type"} if for surfaces of the cell exactly two polar domains have three boundary edges (while the others have two boundaries). A cell will be of \textit{"4-degenerate type"} if for surfaces of the cell one polar domain has four boundary edges.\newline
For translation surfaces in each type of cells, the configuration of saddle connections as an embedded graph in the surface is depicted in Figure~\ref{types}. The types of arrows indicates the families of parallel saddle connections. Among the edges marked with a simple arrow, the leftmost edge is identified with the rightmost edge with a translation (preserving in particular the directions). Similarly, among the edges marked with a double arrow, the uppermost edge is identified with the lowermost edge. As a result, we construct a torus, and all four vertices are identified to the unique zero $z$. The faces of the embedded graph are the polar domains (in white) and the interior of the core (shaded).\newline
An explicit example of a translation surface is given in Figure~\ref{iso_2}. In this case, where the differential belongs to a cell of $4$-degenerate type, the surface is formed by only one polar domain with four boundary arcs identified by pairs. The translation structure is entirely determined by the lengths of the two resulting saddle connections and the angle magnitudes at the four inner corners of the polar domain.

\begin{figure}
    \centering
    \input{diagrams/types} 
    \caption{Types of cells} \label{types}
\end{figure}

\begin{rmk}
As we will see in Section~\ref{sub:2cells}, for a same partition $\mu$, $\mathbb{P}\mathcal{R}_{1}(\mu)$ can contain several cells of the same type (determined by the topological pair defined by the core inside the surface). Angles between parallel saddle connections are integer multiples of $2\pi$ and these weights are combinatorial invariants that can be used to distinguish two cells of the same type.
\end{rmk}

\subsection{Exceptionally symmetric cells}\label{sub:ExcSymCel}

For a given cell, the period coordinate is given by a choice of two saddle connections $\alpha,\beta$ that form a symplectic basis of $X$. If the cell does not have a symmetry that sends the symplectic basis $\{\alpha,\beta\}$ to another basis, then the cell is isomorphic to a hemisphere. That is, the period map from the hemisphere is injective into $\PP\RR_1(\mu)$. If there exists no such symplectic basis, then we need to take care of isomorphic translation surfaces given by change of basis.

\begin{figure}
    \centering
    \input{diagrams/cylinderauto} 
    \caption{An isomorphism of a surface in the cell of cylinder type} \label{cylinderauto}
\end{figure}

Any translation surface in a cell of cylinder type has infinitely many choices for $\{\alpha,\beta\}$. Suppose that $(X,\omega)$ is a translation surface in the cell of cylinder type. Consider a symplectic basis consisting of a core curve $\alpha$ and a cross curve $\beta$ of the cylinder. Then there always exists another cross curve $\gamma$ whose homology class is equal to $[\alpha]+[\beta]$. Let $x=\int_\alpha \omega$ and $y=\int_\beta \omega$. Then the period coordinate in this cell can be given by $[x,y]=[1,y/x]\in \mathcal{H}\subset \mathbb{CP}^1$. However, two translation surfaces given by $[x,y]$ and $[x,x+y]$ are isomorphic (see Figure \ref{cylinderauto}). This type of isomorphisms provides a free $\mathbb{Z}$-action on $\mathcal{H}$. Therefore the cells of cylinder type are isomorphic to $\mathcal{H}/\mathbb{Z}$, which can be identified with a punctured disk. The puncture corresponds to a boundary element obtained by sending the length of the cross curve of the cylinder to infinity. The limit object contains two simple poles, so the puncture parametrizes a horizontal boundary element (see Section~\ref{sec:Multiscale}). Also, this puncture is a cusp. 

In the following proposition, we classify exceptionally symmetric cells. They are the cells containing an orbifold point.

\begin{prop}\label{prop:ExcSymCel}
A stratum $\PP\RR_1(\mu)$ has a cell containing an orbifold point if and only if any of the following statements holds:
\begin{itemize}
    \item If $p=1$ and $a \in 2\mathbb{N}$, then the (hyperelliptic) connected component of rotation number $\frac{a}{2}$ has one cell of 4-degenerate type containing an orbifold point of order 2;
    \item If $p=1$ and $a \in 3\mathbb{N}$, then the connected component of rotation number $\frac{a}{3}$ has one cell of 1-triangle type containing an orbifold point of order 3; 
    \item If $p=2$ and $b_{1},b_{2} \in 3\mathbb{N}$, then the connected component of rotation number $\gcd(b_1,b_2,\frac{b_1-b_2}{3})$ has one cell of 3-degenerate type containing an orbifold point of order 3;
    \item If $p=2$ and $b_{1},b_{2} \in 3\mathbb{N}+2$, then the connected component of rotation number $\gcd(b_1,b_2,\frac{b_1-b_2}{3})$ has one cell of 3-degenerate type containing an orbifold point of order 3.
\end{itemize}
We refer to these cells as \textit{exceptionally symmetric cells}. Moreover, there are no other orbifold points in projectivized strata $\PP\RR_1(\mu)$.
\end{prop}

\begin{proof}
Automorphisms of a translation surface $(X,\omega)\in \RR_1(\mu)$ have been classified in Proposition~\ref{prop:symmetries}. Moreover, by Remark~\ref{hyperremark}, every hyperelliptic translation surface is contained in a hyperelliptic component. Therefore, in a non-hyperelliptic component, the only possible automorphism is of order $3$. In contrast, in a hyperelliptic component, orbifold points of order $2$ correspond to flat surfaces with a symmetry of order $4$ (because a symmetry of order $2$ is already satisfied by any neighboring surface). To sum up, orbifold points can appear in $\PP\RR_1(\mu)$ only in two cases:
\begin{itemize}
    \item an orbifold point of order $3$ if a surface representative of the projective class has a nontrivial automorphism $\sigma$ of order $3$ (in this case we have $p=1$ or $p=2$);
    \item an orbifold point of order $2$ if the projective class belongs to a hyperelliptic component of $\PP\RR_1(\mu)$ and a surface representative has a nontrivial automorphism $\sigma$ of order $4$ (in this case we have $p=1$).
\end{itemize}

If $p=2$, the only possible automorphism $\sigma$ satisfies $\sigma^{3}=Id$ and acts nontrivially on the set of boundary saddle connections of each polar domain. It follows that the two polar domains are related to each other by exactly three saddle connections. The projective class thus belongs to a cell of 3-degenerate type (or its boundary), see Figure~\ref{iso_3}. If such cell exists, then the total angle in the polar domain of $q_1$ is equal to $\pi+2\pi b_1$. So each angle $\theta$ is equal to $\frac{\pi+2\pi b_1}{3}$. Similarly, the angle $\phi$ is equal to $\frac{\pi+2\pi b_2}{3}$. Thus, we must have $b_{1}-b_{2} \in 3\mathbb{Z}$. If $b_{1},b_{2} \in 3\mathbb{N}+1$, we get $\theta \equiv \pi \text{ mod } 2\pi$ and this is incompatible with a transitive action of $\sigma$ on the periods of boundary saddle connections. Thus, either we have $b_{1},b_{2} \in 3\mathbb{N}^{\ast}$ or $b_{1},b_{2} \in 3\mathbb{N}+2$.
\par
Note that the three symplectic bases, $\{\alpha,\beta\}, \{-\beta,\gamma\}, \{\gamma,-\alpha\}$ cannot be distinguished from each other. In this case, there exists a $\mathbb{Z}/3\mathbb{Z}$-action on the period coordinate given by isomorphism between translation surfaces. The fixed point is the translation surface depicted in the right of Figure~\ref{iso_3} and it is an orbifold point of order 3. The rotation number of this translation surface is equal to $\gcd(b_1,b_2,\frac{b_1 + 2 b_2}{3},\frac{2b_1+b_2}{3})=\gcd(b_1,b_2,\frac{b_1 - b_2}{3})$.

\begin{figure}
    \centering
    \input{diagrams/iso_3} 
    \caption{The translation surface corresponding to an orbifold point of order 3 in a cell of 1-triangle / 3-degenerate type} \label{iso_3}
\end{figure}

If a translation surface $(X,\omega)$ satisfying $p=1$ has an automorphism $\sigma$ of order 3, then the unique polar domain is bounded by three saddle connections related to a triangle. $(X,\omega)$ belongs to a cell of 1-triangle type (see the left of Figure~\ref{iso_3}).  If such cell exists, then the angle $\theta$ satisfies $\theta\equiv \frac{\pi}{3}\text{ mod } 2\pi$. So the order $a$ of the pole satisfies $3~\vert~a$. Similarly to the case in the previous paragraph, we have $\mathbb{Z}/3\mathbb{Z}$-action on the period coordinate. The fixed point is the translation surface depicted in the left of Figure~\ref{iso_3} and it is an orbifold point of order 3. The rotation number of this translation surface is equal to $\gcd(a,\frac{a}{3})=\frac{a}{3}$.

In the last case, $(X,\omega)$ has an automorphism $\sigma$ of order 4 and belongs to a cell of 4-degenerate type (see the translation surface depicted in Figure~\ref{iso_2}). This translation surface is contained in the hyperelliptic component. If such cell exists, then the total angle of the polar domain is equal to $4\theta=2\pi(a+1)$. So $\theta=\frac{a+1}{2}\pi$. If $a$ is odd, then $\theta$ is a multiple of $\pi$, which is impossible. Thus $a$ is even. Note that the four symplectic basis, $\{\alpha,\beta\}, \{-\beta,\alpha\},\{-\alpha,-\beta\},\{-\beta,-\alpha\}$ cannot be distinguished from each other. In this case, there exists a $\mathbb{Z}/4\mathbb{Z}$-action on the period coordinate given by isomorphism between translation surfaces. The fixed point is the translation surface depicted in Figure~\ref{iso_2}. The hyperelliptic involution corresponds to the subgroup of order 2, and the hyperelliptic involution does not change the projectivized coordinate. Therefore this is an orbifold point of order 2.  

\begin{figure}
    \centering
    \input{diagrams/iso_2} 
    \caption{The translation surface corresponding to an orbifold point of order 2 in a cell of 4-degenerate type} \label{iso_2}
\end{figure}
\end{proof}

The number of boundary arcs of a cell depends on its type and its symmetries.

\begin{prop} \label{adjacentarcs}
Exceptionally symmetric cells are bounded by a unique equatorial arc. For any other cell, the number of boundary arcs depends on its type:
\begin{itemize}
    \item A cell of cylinder type is bounded by one equatorial arc;
    \item A cell of 4-degenerate type is bounded by two arcs;
    \item A cell of quadrilateral type is bounded by four arcs;
    \item A cell of 3-degenerate, 1-triangle, or 2-triangle types is bounded by three arcs.
\end{itemize}
\end{prop}

\begin{proof}
Recall that the equatorial arcs are the loci where two saddle connections forming a symplectic basis have $\mathbb{R}$-colinear periods. Thus the number of equatorial arcs adjacent to a 2-cell is equal to the number of the symplectic basis consisting of saddle connections, counted up to homology. 

If $(X,\omega)$ is contained in a cell of cylinder type, then there are infinitely many homology classes represented by saddle connections. However, any symplectic basis consists of a core curve of the cylinder and a cross curve of the cylinder. By description of the isomorphism with the translation surfaces in this cell given in Section~\ref{sub:ExcSymCel}, all symplectic bases are identified up to these isomorphisms. So there is only one equatorial arc adjacent to a cell of cylinder type. 

For the other types of cells, there are only finitely many saddle connections. The number of distinct symplectic bases are as listed above in the proposition, unless the cell has extra symmetry. This extra symmetry of the cell has a fixed point which is an orbifold point (see Proposition~\ref{prop:ExcSymCel}). In all of these cases the number of boundary saddle connections reduces to one because of the symmetry.
\end{proof}

\section{Multi-scale compactification}\label{sec:Multiscale}

A projectivized residueless stratum $\PP\RR_1 (\mu)$ has a smooth compactification called the {\em moduli space of multi-scale differentials}, which we will denote $\PP\overline{\RR}_1 (\mu)$. It is a one-dimensional compact complex orbifold. The projective structure gives a cellular decomposition of $\PP\overline{\RR}_1 (\mu)$, which will allow us to investigate the topological structure of the stratum. 

In this section, we will refer to the properties of the multi-scale compactification $\PP\overline{\RR}_1(\mu)$. First, we recall the following lemma (proved as Lemma~6.4 in \cite{Lee}) that describes all boundary points of the multi-scale compactification $\PP\overline{\RR}_1 (\mu)$. 

\begin{lem}\label{genus1lemma}
A multi-scale differential $\overline{X} \in \partial\overline{\RR}_1 (\mu)$ is determined by the following combinatorial data:

\begin{itemize}
    \item An integer $0\leq t\leq p$. 
    \item A permutation $\tau\in S_{p}$.
    \item A tuple of integers ${\bf C}=(C_1,\dots,C_p)$ such that $1\leq C_i \leq b_i -1$ for each $i$. Denote $Q_1\coloneqq \sum_{i=1}^t C_{\tau(i)}$ and $Q_2\coloneqq \sum_{i=1}^t (b_{\tau(i)}-C_{\tau(i)})$.
    \item If $t>0$, we also need a prong-matching equivalence class $[(u,v)]$ represented by a prong-matching $(u,v)\in \mathbb{Z}/Q_1\mathbb{Z} \times \mathbb{Z}/Q_2 \mathbb{Z}$. 
\end{itemize}
\end{lem}
We denote this multi-scale differential by $\overline{X}\coloneqq X(t,\tau,{\bf C},[(u,v)])$. Here we explain how this data construct a multi-scale differential. Let $q_i$ be the pole of order $b_i$.

\begin{figure}
    \centering
    \input{diagrams/dualgraph} 
    \caption{Dual graph of $\overline{X}$} \label{dualgraph}
\end{figure}

If $t=0$, then we obtain a multi-scale differential on the nodal rational curve (see the right of Figure~\ref{dualgraph}). That is, an element of $\RH_0(a,-b_1,\dots,-b_p,-1,-1)$ whose residues at the first $p$ poles are zero. Also, the last two (simple) poles are unmarked. Any such translation surface $(\mathbb{CP}^1,\omega)$ has $p+1$ saddle connections, all parallel to each other. Two of them are the core curves of the half-infinite cylinders corresponding to the simple poles. We can label the saddle connections $\beta_1,\dots,\beta_{p+1}$ in clockwise order so that $\beta_1$ and $\beta_{p+1}$ are the core curves. For the multi-scale differential $X(0,\tau,{\bf C})$, the pair of saddle connections $\beta_i$ and $\beta_{i+1}$, for $1\leq i\leq p$, bounds the polar domain of $q_{\tau(i)}$ and forms an angle equal to $2\pi C_{\tau(i)}$. 

If $t \geq 1$, then we obtain a two-level multi-scale differential (see the left of Figure~\ref{dualgraph}). That is, there are two irreducible components of $\overline{X}$, located at distinct levels, intersecting at two nodes $s_1,s_2$. Let $X_0$ denote the top level component and $X_{-1}$ denote the bottom level component. For $X(t,\tau,{\bf C},[(0,v)])$, the component $X_0$ contains $q_{\tau(1)},\dots,q_{\tau(t)}$. There are two zeroes of $X_0$ at the nodes, and there are $t$ (parallel) saddle connections of $X_0$ joining the two zeroes. We can label the saddle connections $\alpha_0,\dots,\alpha_{t-1}$ in clockwise order at $s_1$. The pair of saddle connections $\alpha_{i-1}$ and $\alpha_i$ bounds the polar domain of $q_{\tau(i)}$ and forms an angle equal to $2\pi C_{\tau(i)}$ for each $i$. Then the bottom level component $X_{-1}$ contains remaining poles $q_{\tau(i+1)},\dots,q_{\tau(i+1)}$. It has two more poles at the nodes $s_1,s_2$ with nonzero residues. Similarly to the former case, there are $p-t+1$ saddle connections labeled by $\beta_1,\dots,\beta_{p-t+1}$ so that $\beta_1$ and $\beta_{p-t+1}$ bounds the polar domains of $s_1$ and $s_2$. Then the pair of saddle connections $\beta_i$ and $\beta_{i+1}$, for $1\leq i\leq p-t-2$, bounds the polar domain of $q_{\tau(t+i)}$ and forms an angle equal to $2\pi C_{\tau(i)}$. 

The number of prongs at the nodes $s_1$ and $s_2$ are given by $Q_1$ and $Q_2$, respectively. In $X_0$, the prongs are labeled by $v^+_1,\dots, v^+_{Q_1}$ in clockwise order at $s_1$ and $w^+_1,\dots, w^+_{Q_2}$ in counter-clockwise order at $s_2$, so that $v^+_1$ and $w^+_1$ correspond to the saddle connection $\alpha_0$. In $X_{-1}$, the prongs are labeled by $v^-_1,\dots, v^-_{Q_1}$ in counter-clockwise order at $s_1$ and $w^-_1,\dots, w^-_{Q_2}$ in clockwise order at $s_2$ so that $\beta_1$ lies between $v^-_1,v^-_{Q_1}$, and $\beta_{p-t+1}$ lies between $w^-_1, w^-_{Q_2}$. So the prong-matching that sends $v^-_1$ and $w^-_1$ to $v^+_u$ and $w^+_v$ can be identified with the element $(u,v)\in \mathbb{Z}/Q_1\mathbb{Z} \times \mathbb{Z}/Q_2 \mathbb{Z}$.

The multi-scale differentials $X(0,\tau,{\bf C})$ are also said to be {\em horizontal} boundary elements, since the corresponding level graph has a horizontal loop. The other boundary elements (i.e. when $t>0$) are said to be {\em vertical}, since there are two vertical edges in the corresponding level graph.

\begin{rmk}\label{multiplecount}
The expression $X(t,\tau,{\bf C},[(0,v)])$ given above is not unique for a given multi-scale differential. In particular, if $\sigma_1= \left(\begin{smallmatrix}
1 & 2 & \dots & t & t+1 & \dots & p\\
2 & 3 & \dots & 1 & t+1 & \dots & p
\end{smallmatrix}\right)$, then $X(t,\tau, {\bf C}, [(u,v)]) = X(t,\tau \circ \sigma_1, {\bf C}, [(u+C_{\tau(1)},v+D_{\tau(1)})])$. This relation is given by relabeling the saddle connections $\alpha_i$ in the top level component $X_0$. Also, if $\sigma_2 = \left(\begin{smallmatrix}
1 & \dots & t & t+1 & \dots & p\\
t & \dots & 1 & p & \dots & t+1
\end{smallmatrix}\right)$, then $X(t,\tau, {\bf C}, [(u,v)]) = X(t,\tau \circ \sigma_2, {\bf D}, [(-v,-u)])$. This relation is given by relabeling the two nodes $s_1$ and $s_2$.

If $p>3$, then $t>2$ or $p-t>1$ for any multi-scale differential. That is, the two nodes $s_1$ and $s_2$ can be distinguished implicitly as at least one of the two components $X_{-1}$ and $X_0$ does not have an involution interchanging the two nodes. Therefore, $\sigma_1$ and $\sigma_2$ generates a subgroup of $2t$ elements. So each multi-scale differential in $\partial\overline{\RR}_1 (\mu)$ has $2t$ distinct expressions of the form $X(t,\tau,{\bf C},[(u,v)])$. 

If $p\leq 3$, then there are possibly more symmetries, coming from the hyperelliptic involution that fixes all poles. More precisely, the hyperelliptic involutions that fixes all poles. 

Suppose $1\leq p\leq 2$ and $t=1$. If all $b_i$ are even, then there exists a hyperelliptic component of $\RR_1(\mu)$. For $C_i=\frac{b_i}{2}$, the hyperelliptic involution gives the identity $X(1,Id,{\bf C},[(u,v)])=X(1,Id,{\bf C},[(-v,-u)])$. Note that if $u+v=C_1$, then this gives an isomorphism between two multi-scale differentials. Moreover, if $4~\vert~b_4$, then we have an additional isomorphism between multi-scale differentials $X(1,Id,{\bf C},[(0,\frac{b_1}{4})])=X(1,Id,{\bf C},[(-\frac{b_1}{4},0)])$. So there are $\frac{b_1}{4}+1$ or $\frac{b_1+2}{4}$ prong-matching equivalence classes that give distinct multi-scale differentials if $4~\vert~b_1$ or $4\nmid b_1$, respectively. 

Suppose $2\leq p\leq 3$ and $t=2$. If all $b_i$ are even, then there exists a hyperelliptic component of $\RR_1(\mu)$, where all poles are fixed by the hyperelliptic involution. For $C_i=\frac{b_i}{2}$, the involution gives the identity $X(2,Id,{\bf C},[(0,v)])=X(2,Id,{\bf C},[(0,2C_1-v)])$ of multi-scale differentials. Two prong-matching equivalence classes $[(0,v)]$ and $[(0,2C_1-v)]$ are equal if and only if $v=C_1,\frac{C_1-C_2}{2}$. So there are $\frac{b_1+b_2}{4}+1$ or $\frac{b_1+b_2+2}{4}$ prong-matching equivalence classes that give distinct multi-scale differentials if $4~\vert~b_1+b_2$ or $4\nmid b_1+b_2$, respectively.
\end{rmk}

Moreover, for a given multi-scale differential $\overline{X}$ in the boundary, we can compute the rotation number of a translation surface in the neighborhood of $\overline{X}$. The following is proved as Proposition~7.1 in \cite{Lee}.

\begin{prop} \label{rot}
Let $\overline{X}=X(t,\tau, {\bf C}, [(0,v)])\in \partial\overline{\RR}_1 (\mu)$. The rotation number of a translation surface in the neighborhood of $\overline{X}$ is equal to $\gcd(d,Q_1,\sum_{i=1}^p C_{\tau(i)} + v)$.
\end{prop}

\section{Counting cells}\label{sec:Counting}

In this section, we will count the number of 0-cells, 1-cells, and 2-cells of the cellular decomposition of $\RR_1(\mu)$ given by the canonical complex projective structure. If $\RR_1(\mu)$ does not have a hyperelliptic connected component with trivial ramification profile, or an orbifold point, then the computation is easier and written in simple formulae. The hyperelliptic components or the cells containing orbifold points produce the exceptions in the counting argument, so we will introduce the correction terms, which are zero for most cases, for the formulae counting the number of cells. 

Given a multi-scale differential $\overline{X}\in \partial\PP\overline{\RR}_1(\mu)$, we can describe the neighborhood of $\overline{X}$ by smoothing process called {\em plumbing construction}. For a given prong-matching and the scaling parameter, this allows to construct a translation surface $(X,\omega)$ in the neighborhood of $\overline{X}$. Moreover, we can determine the cell containing $(X,\omega)$.

\subsection{Plumbing construction}

\begin{prop} \label{plumbhorizontal}
Any multi-scale differential $X(0,\tau,{\bf C})$ is contained in a cell of cylinder type. Additionally, it is not adjacent to any equatorial arc. 
\end{prop} 

\begin{proof}
Let $\overline{X}=X(0,\tau,{\bf C})$. By plumbing the node with any scaling parameter $s\in \C$, we obtain a translation surface $(X,\omega)$ containing the cylinder. Therefore, $\overline{X}$ is contained in a cell of cylinder type. 

Now assume the contrary --- that $(X,\omega)$ is contained in an equatorial arc. Let $\alpha$ be a closed geodesic enclosing the cylinder. We can find a cross curve $\beta$ of the cylinder such that $(\alpha,\beta)$ is a symplectic basis of the absolute homology of $X$. Moreover, since $\alpha$ comes from a small circle around a simple pole, $\int_{\alpha} \omega = r$. Also remark that $\left| \int_{\beta} \omega\right|$ diverges as $s\to 0$, since the cylinder degenerates to two half-infinite cylinders. By choosing smaller enough $s\in \C$, we may assume that $\left|\frac{1}{r}\int_{\beta}\omega \right|>1$ for any cross curve $\beta$. That is, $\int_{\beta}\omega = cr$ for some $c\in \mathbb{R}$. Note that a cylinder has infinitely many cross curves. More precisely, for any integer $k$, we can find a cross curve $\beta'$ such that $\int_{\beta'}\omega = (c-k)r$. Choose $\beta'$ such that $0\leq c-k <1$. So we have $\left|\frac{1}{r}\int_{\beta'}\omega \right|<1$, a contradiction. 
\end{proof}

Therefore, a horizontal multi-scale differential represents a cusp in a cell of cylinder type. 

For any vertical multi-scale differential, we can count the number of equatorial arcs adjacent to it, using the description of translation surfaces obtained by plumbing construction. 

\begin{prop} \label{plumbvertical}
A multi-scale differential $X(t,\tau,{\bf C},[(u,v)])$ for $t>0$ is adjacent to $2\lcm(Q_1,Q_2)$ equatorial arcs (thus, it is also adjacent to $2\lcm(Q_1,Q_2)$ 2-cells). 
\end{prop}

\begin{proof}
Let $\overline{X}=X(t,\tau,{\bf C},[(u,v)])$. It has two irreducible components $(X_0,\eta_0)$ and $(X_{-1},\eta_{-1})$. Recall that $X_0$ contains $t$ poles $q_{\tau(1)},\dots,q_{\tau(t)}$ and two zeroes of orders $Q_1-1$ and $Q_2-1$ at the nodes $s_1$ and $s_2$, respectively. There are $t$ (parallel) saddle connections $\alpha_0,\dots, \alpha_{t-1}$ in $X_0$, joining $s_1$ and $s_2$. The domain of $q_{\tau(i)}$ is bounded by $\alpha_i$ and $\alpha_{i+1}$ ($\alpha_t$ if $i=t-1$). By scaling $\eta_0$ if necessary, we may assume that $\int_{\beta_i} \eta_0 =-1$. 

Similarly, $\eta_{-1}$ has two (non-residueless) poles of orders $Q_1+1$ and $Q_2+1$ at the nodes $s_1$ and $s_2$, respectively. There are $p-t+1$ (parallel) saddle connections $\beta_1,\dots, \beta_{p-t+1}$ joining $z$ to itself. The domain of $q_{\tau(t+i)}$ is bounded by $\alpha_i$ and $\alpha_{i+1}$ for each $i=1,\dots, p-t$. The domain of $s_1$ ($s_2$ respectively) is bounded by $\alpha_1$ ($\alpha_{p-t+1})$. By scaling $\eta_{-1}$ if necessary, we may assume that $\int_{\alpha_i}\eta_{-1}=1$. 

By plumbing construction with a prong-matching $(u,v)$ and the scaling parameter $s\in \C$, we obtain a translation surface $(X,\omega)\in \RR_1(\mu)$. The surface $(X,\omega)$ has a saddle connection $\alpha'_j$ that deforms to $\beta_j$ as $s\to 0$, and $\int_{\alpha'_j}\omega=s$. There is also a saddle connection $\beta'_i$ in $(X,\omega)$ that deforms to $\alpha_i$ as $s\to 0$. There exists $M>0$ such that $\left|\int_{\beta'_i}\omega -1\right|<M|s|$. So for small $s$ satisfying $Im(s)<0$, we can conclude that $\frac{\int_{\alpha'_j}\omega}{\int_{\beta'_i}\omega}$ has positive imaginary part. Also when $s$ satisfies $Im(s)>0$, $\frac{\int_{\alpha'_j}\omega}{\int_{\beta'_i}\omega}$ has negative imaginary part. Thus $(X,\omega)$ is contained in an equatorial arc if and only if $s\in \mathbb{R}$. Therefore, for each prong-matching $(u,v)$, we can find two equatorial arcs adjacent to $\overline{X}$.

Consider a small circle $s=\varepsilon e^{i\theta}$, $0\leq \theta\leq 2\pi$, centered at $0$. As the scaling parameter $s$ travels from $\theta=0$ to $\theta=2\pi$, by level rotation action, the prong-matching $(u,v)$ is replaced by $(u-1,v+1)$. The orbit of this action has $\lcm(Q_1,Q_2)$ prong-matchings, so $\overline{X}$ is adjacent to $2\lcm(Q_1,Q_2)$ equatorial arcs. 
\end{proof}

\subsection{Counting 0-cells}

The 0-cells of the cellular decomposition consist of two-level multi-scale differentials $\overline{X}\in \partial \PP\RR_1 (\mu)$. By Lemma~\ref{genus1lemma}, any such $\overline{X}$ can be written as $X(t,\tau,{\bf C},[(u,v)])$ for $t>0$. Also, there are $2t$ distinct ways to write the same $\overline{X}$ (with few exceptions for the cases $p=2,3$). Now we can compute the total number of 0-cells. 

\begin{cor}
The total number $V$ of 0-cells of the cellular decomposition of $\PP\RR_1(\mu)$ is equal to

\begin{equation}
    V=\sum_{t,\tau,{\bf C}} \frac{1}{2t}\gcd(Q_1,Q_2)+\varepsilon_0.
\end{equation}

The correction term $\varepsilon_0$ is given by
\begin{equation}
\varepsilon_0=
    \begin{cases}
    \frac{1}{2} & \text{if } p=1, a\in 4\mathbb{N}+2 \\
    1 & \text{if } p=1, a\in 4\mathbb{N}^{\ast} \\
    1 & \text{if } p=2, b_1,b_2\in 4\mathbb{N}+2 \\ 
    2 & \text{if } p=2, b_1\in 4\mathbb{N}+2,b_2\in 4\mathbb{N}^{\ast} \\ 
    3 & \text{if } p=2, b_1,b_2\in 4\mathbb{N}^{\ast} \\ 
    2 & \text{if } p=3, b_i\in 4\mathbb{N}^{\ast} \text{ for all } i, \text{ or } b_i\in 4\mathbb{N}+2 \text{ for all } i\\ 
    3 & \text{if } p=3, \text{ all } b_i \text{ even and } b_1 \in 4\mathbb{N}^{\ast}, b_2\in 4\mathbb{N}+2 \\ 
    0 & \text{otherwise. }
    \end{cases}
\end{equation}
\end{cor}

\begin{proof}
Assume that all $b_i$ are even. Otherwise the correction term is obviously zero. 

If $p=1$, then the multi-scale differentials $X(1,Id,{\bf C},[(0,v)])$ with $C_1=\frac{b_1}{2}=\frac{a}{2}$ are counted only once, not twice. But the number of prong-matching equivalence classes is also half of the other cases, unless $4\nmid a$. So the correction term is equal to $-\frac{1}{2}\frac{a}{2}+\frac{a+2}{4}=\frac{1}{2}$ if $4\nmid a$.

If $p=2$, then similarly, $X(2,\tau,{\bf C},[(0,v)])$ with $C_i=\frac{b_i}{2}$ is counted only twice, not 4 times. But the number of prong-matching equivalence classes is also half of other cases, unless $4\nmid a$. So the correction term is equal to $\sum_{\tau}(-\frac{1}{4}\frac{a}{2}+\frac{1}{2}(\frac{a+x}{4})=\frac{1}{2}$ if $4\nmid a$. 

If $p=3$, by the same argument as above, we need to add $\frac{1}{2}$ whenever we have a pair of poles whose sum of the orders is divisible by 4. 
\end{proof}

Recall from Lemma~\ref{genus1lemma} that $Q_1=\sum_{i=1}^t C_{\tau(i)}$ and $Q_2=\sum_{i=1}^t (b_{\tau(i)}-C_{\tau(i)})$. 

\subsection{Counting 1-cells}

For each 0-cell $c_0$ given by $X(t,\tau,{\bf C},[(u,v)])$, the number of 1-cells adjacent to $c_0$ (in other words, the order of $c_0$) is equal to $\lcm(Q_1,Q_2)$ by Proposition~\ref{plumbvertical}. Now we can compute the number of 1-cells, using the fact that the number of edges of a graph is equal to one half of the sum of the orders of all vertices. We denote $A=\sum_{i=1}^p b_i^2$ and $B=\prod_{i=1}^p (b_i-1)$. 

\begin{cor} \label{edgecount}
The total number $E$ of 1-cells of the cellular decomposition of $\PP\RR_1(\mu)$ is equal to

\begin{equation}
    E=\frac{p!B}{48}(A+4a+3a^2) +\varepsilon_1.
\end{equation}

The correction term $\varepsilon_1$ is given by

\begin{equation}
\varepsilon_1=
    \begin{cases}
    \frac{a}{4} & \text{if } p=1, \text{ and } a\in 2\mathbb{N}^{\ast} \\
    \frac{a}{2} & \text{if } p=2, \text{ and } b_1,b_2\in 2\mathbb{N}^{\ast} \\ 
    0 & \text{otherwise. }
    \end{cases}
\end{equation}
\end{cor}

\begin{proof}
First, we assume that $p>2$ or some $b_i$ is odd. For fixed $t>0$, $\tau\in S_{p}$ and ${\bf C}$, there are $\gcd(Q_1,Q_2)$ multi-scale differentials of the form $X(t,\tau,{\bf C},[(u,v)])$. Each of the multi-scale differential has $\lcm(Q_1,Q_2)$ choices of prong-matchings and adjacent to $2\lcm(Q_1,Q_2)$ equatorial arcs. Therefore, 
$$2E=\sum_{t>0}\sum_{\tau}\sum_{\bf C}\frac{1}{2t} 2\gcd(Q_1,Q_2)\lcm(Q_1,Q_2)=2\sum_{t,\tau,{\bf C}} \frac{1}{2t}Q_1Q_2.$$

Note that $$Q_1Q_2=\left[\sum_{i=1}^t C_{\tau(i)}\right]\left[\sum_{j=1}^t b_{\tau(i)}-C_{\tau(i)}\right]=\sum_{k=1}^t \left[C_{\tau(k)}b_{\tau(k)}-C_{\tau(k)}^2\right] + \sum_{1\leq i\neq j\leq t} \left[C_{\tau(i)}(b_{\tau(j)}-C_{\tau(j)})\right].$$

The sum $\sum_{t,\tau}\sum_{k=1}^t \frac{1}{2t}\left[C_{\tau(k)}b_{\tau(k)}-C_{\tau(k)}^2\right]$ is symmetric, so each term $C_kb_k-C_k^2$, for $k=1,\dots,p$, appears the same number of times. Since the sum of the coefficients is equal to $p!\frac{p}{2}$, the sum is equal to $$\frac{p!}{2}\sum_{k=1}^p C_kb_k-C_k^2.$$ By taking the summation over ${\bf C}$, we have $$\frac{p!}{2}\sum_{k=1}^p \frac{b_k^2}{2}B -\frac{b_k(2b_k-1)}{6} B = \frac{p!B}{12}(A+a).$$

Similarly, for $p>1$, the sum $\sum_{t,\tau} \sum_{1\leq i\neq j\leq t} \frac{1}{2t} \left[C_{\tau(i)}(b_{\tau(j)}-C_{\tau(j)})\right]$ is also symmetric, and each term $C_i(b_j-C_j)$, for $1\leq i\neq j\leq p$, appears the same number of times. Since the sum of the coefficients is equal to $p!\sum_t \frac{t-1}{2}=\frac{p!}{4} p(p-1)$, the sum is equal to $$\frac{p!}{4}\sum_{1\leq i\neq j\leq p} C_i(b_j-C_j).$$ By taking the summation over ${\bf C}$, we have $$\frac{p!}{4}\sum_{1\leq i\neq j\leq p}\frac{B}{4} b_i b_j  = \frac{p!B}{16} (a^2-A).$$

Finally, we can obtain the formula for the number of 1-cells:
$$E=\frac{p!B}{48}(A+3a^2+4a).$$

Now we assume that all $b_i$ are even. If $p=1,2$ and $t=1$, set $C_i=\frac{b_i}{2}$ for each $i$. Recall from Remark~\ref{multiplecount} that the hyperelliptic involution sends each prong-matching of $X(1,Id,{\bf C},[(0,0)])$ to itself. That is only half of the arcs adjacent to this vertex are included in the formula above. So the correction term is equal to $\frac{b_1}{4}$. If $p=2$, then the same argument applies to $X(1,(1,2),{\bf C},[(0,0)])$ and we have another term $\frac{b_2}{4}$. So the correction term is equal to $$\varepsilon_1=\frac{a}{4}.$$
\end{proof}

\subsection{Counting 2-cells}\label{sub:2cells}

Now we need to determine how many 0-cells are adjacent to 2-cells of a given type. We can count this by again investigating the neighborhood of 0-cells. By Proposition~\ref{plumbvertical}, a given $X(t,\tau,{\bf C},[(u,v)])\in \partial\overline{\RR}_1(\mu)$ for $t>0$ is contained in the boundary of $2\lcm(Q_1,Q_2)$ 2-cells. We introduce some additional notations related to $X(t,\tau,{\bf C},[(u,v)])$. For $1\leq i\leq t$, we denote $c_i\coloneqq \sum_{j=1}^i C_{\tau(j)}$ and $d_i\coloneqq \sum_{j=1}^i b_{\tau(j)}-C_{\tau(j)}$. 

The $2\lcm(Q_1,Q_2)$ cells correspond to the choice of $\lcm(Q_1,Q_2)$ prong-matchings and the choice of scaling parameter $s\in \mathcal{HS}_-$ (the lower open hemisphere) or $\mathcal{HS}_+$ (the upper open hemisphere), used for the plumbing construction. These choices also determine the types of the cells as follows:

\begin{lem}
Let $\overline{X}=X(t,Id,{\bf C},[(u,v)])\in \partial\overline{\RR}
_1(\mu)$ be a 0-cell, and $(X,\omega)$ be a translation surface obtained by plumbing construction from $\overline{X}$, with a prong-matching $(u,v)$ and the scaling parameter $s\in \C$.

If $(u,v)=(c_i,d_i)$ for some $i$ and $s\in \mathcal{HS}_+$, then $(X,\omega)$ is contained in a cell of cylinder type (quadrilateral type, respectively) if $t=p$ ($1\leq t\leq p-1$).

If $(u,v)=(c_i,d_j)$ for $i\neq j$ and $s\in \mathcal{HS}_+$, then $(X,\omega)$ is contained in a cell of quadrilateral type (2-triangle type, respectively) if $t=p$ ($1\leq t\leq p-1$). 

If $u=c_i$ for some $i$ and $v\neq d_j$ for any $j$, and $s\in \mathcal{HS}_+$, then $(X,\omega)$ is contained in a cell of 1-triangle type. Also, symmetrically, if $v=d_i$ for some $i$ and $u\neq c_j$ for any $j$, and $s\in \mathcal{HS}_+$, then $(X,\omega)$ is contained in a cell of 1-triangle type. 

If $c_i< u< c_{i+1}$ and $d_i< v< d_{i+1}$, then $(X,\omega)$ is contained in a cell of 4-degenerate type. Also, if $u=c_i$ or $c_{i+1}$ and $d_i< u< d_{i+1}$, and $s\in \mathcal{HS}_-$, then $(X,\omega)$ is contained in a cell of 4-degenerate type. Finally, if $v=d_i$ or $d_{i+1}$ and $c_i< u< c_{i+1}$, and $s\in \mathcal{HS}_-$,  then $(X,\omega)$ is contained in a cell of 4-degenerate type.

Otherwise, $(X,\omega)$ is contained in a cell of 3-degenerate type. 
\end{lem}

\begin{proof}
The choice of $(u,v)$ determined the configuration of the saddle connections of $(X,\omega)$. For example, the translation surface $(X,\omega)$ is depicted in Figure~\ref{plumb} when $c_{t-1}<u<c_t$ and $d_{j-1} < v < d_j$.

\begin{figure}
    \centering
    \input{diagrams/plumb}   
    \caption{A translation surface obtained by plumbing with prong-matching $(u,v)$, $c_{t-1}<u<c_t$ and $d_{j-1} < v < d_j$} \label{plumb}
\end{figure}

This certainly gives a cell of 4-degenerate type. Moreover, observe that $s\in \mathcal{HS}_+$(or $\mathcal{HS}_-$) if and only if the periods $\beta'_i$ are contained in $\mathcal{HS}_+$(or $\mathcal{HS}_-$), while the periods of $\alpha_j$ are fixed to be $-1$. That is, in order to obtain a cell of 1-triangle, cylinder, 2-triangle and quadrilateral types, we must have $s\in \mathcal{HS}_+$.

Now it is immediate that each case of $(u,v)$ listed in the lemma gives the corresponding types of cells, depicted in Figure~\ref{types}. 
\end{proof}

Now we can count the number of cells of each type, containing the given 0-cell in their boundary. 

\begin{prop}
Fix $1\leq t\leq p$, $\tau\in S_{p}$ and ${\bf C}$. Consider $\gcd(Q_1,Q_2)$ 0-cells of the form $X(t,\tau,{\bf C},[(u,v)])$. The numbers of cells of each type containing the given $\gcd(Q_1,Q_2)$ 0-cells in their boundary are: 
\begin{itemize}
    \item $t$ cells of cylinder type (quadrilateral type, resp) if $t=p$ ($1\leq t\leq p-1$);
    \item $t^2-t$ cells of quadrilateral type (2-triangle type, resp) if $t=p$ ($1\leq t\leq p-1$);
    \item $t\left(\sum_{i=1}^t b_{\tau(i)} -2t\right)$ cells of 1-triangle type;
    \item $\sum_{i=1}^t \left( 2C_{\tau(i)} (b_{\tau(i)}-C_{\tau(i)})-b_{\tau(i)}+1\right)$ cells of 4-degenerate type.
\end{itemize}
\end{prop}

\begin{proof}
We need to count possible choices of $(u,v)$ for each case. If $(u,v)=(c_i,d_i)$ for some $i$, there are $t$ such choices because $1\leq i\leq t$. If $(u,v)=(c_i,d_j)$ for $i\neq j$, there are $t(t-1)$ choices for the pair $(i,j)$. This gives the first two cases in the proposition. 

Now we count the cells of 1-triangle type. There are $t$ choices for $u=c_i$, and $Q_2-t$ choices for $v\neq d_j$ for any $j$. So we have $t(Q_2-t)$ pairs. Similarly, there are $t(Q_1-t)$ choices of pairs such that $v=d_i$ for some $i$ and $u\neq d_j$ for any $j$. Thus we have $t(Q_1+Q_2-2t)=t\left(\sum_{i=1}^t (b_{\tau(i)} -2t\right)$ cells of 1-triangle type. 

Finally, we count the cells of 4-degenerate type. For each $1\leq i\leq t$, there are $(C_{\tau(i)}-1)(b_{\tau(i)}-C_{\tau(i)}-1)$ possible choices of $(u,v)$ such that $c_{i-1}< u< c_i$, $d_{i-1}< u< d_i$. This gives $2\left[(C_{\tau(i)})(b_{\tau(i)}-C_{\tau(i)}) -b_{\tau(i)}+1\right]$ cells. Also, for $u=c_{i-1}$ or $c_i$, there are $b_{\tau(i)}-C_{\tau(i)}-1$ choices for $v$. This gives $2(b_{\tau(i)}-C_{\tau(i)}-1)$ cells. Finally, for $v=d_{i-1}$ or $d_i$, there are $C_{\tau(i)}-1$ choices for $u$. This gives $2(C_{\tau(i)}-1)$ cells. By taking summation over $i$, we have $\sum_{i=1}^t 2C_{\tau(i)} (b_{\tau(i)}-C_{\tau(i)})-t$ cells of 4-degenerate type. 
\end{proof}

Observe that the number of the cells of 3-degenerate type can be computed by subtracting the number of the other cells from $2\lcm(Q_1,Q_2) \gcd(Q_1,Q_2)=2 Q_1Q_2$. 

By combining the above relations and Corollary~\ref{edgecount}, we have

\begin{prop}\label{facecount}
The number of 2-cells $F$ is equal to 
\begin{equation}
    F=\frac{p!B}{24}(A+a^2+10)+\varepsilon_2+\eta.
\end{equation}

where the correction term $\varepsilon_2$ is given by
    
\begin{equation}
\varepsilon_2=
    \begin{cases}
    \frac{a+1}{4} & \text{if } p=1 \text{ and } a\in 2\mathbb{N}^{\ast} \\
    \frac{a-1}{2} & \text{if } p=2 \text { and } b_1,b_2\in 2\mathbb{N}^{\ast} \\ 
    0 & \text{otherwise }
    \end{cases}
\end{equation}

while the orbifold term $\eta$ is given by 
\begin{equation}
\eta=\begin{cases}
     \frac{1}{2} & \text{if } p=1 \text{ and } a\in 6\mathbb{N}+2, 6\mathbb{N}+2 \\
     \frac{2}{3} & \text{if } p=1 \text{ and } a\in 6\mathbb{N}+3 \\
     \frac{5}{6} & \text{if } p=1 \text{ and } a\in 6\mathbb{N}^{\ast}  \\
     \frac{2}{3} & \text{if } p=2 \text{ and } b_1,b_2\in 3\mathbb{N}^{\ast}  \\ 
     \frac{2}{3} & \text{if } p=2 \text{ and } b_1,b_2\in 3\mathbb{N}+2  \\ 
     0 & \text{otherwise. }
     \end{cases}
\end{equation}
\end{prop}

\begin{proof}
Each cell is a topological disk (possibly punctured in the case of a chamber of cylinder type) bounded by finitely many boundary arcs that meet at (inner) corners of the cell. For each type of cell, we count the total number of corners. We denote the number of corners of cylinder, quadrilateral, 4-degenerate, 3-degenerate, 1-triangle, and 2-triangle type by $K_{cyl}$, $K_q$, $K_4$, $K_3$, $K_{1t}$ and $K_{2t}$, respectively. First, assume that $p>2$ or some $b_i$ is odd. Also assume that $\RR_1(\mu)$ does not have 2-cells with orbifold points. 

By Proposition~\ref{adjacentarcs}, the number of 2-cells is equal to 

\begin{equation}\label{face}
    F=K_{cyl}+\frac{1}{4}K_q+\frac{1}{2}K_4+\frac{1}{3}(K_3+K_{1t}+K_{2t}).
\end{equation}
Also, in terms of the corners, the number of 1-cells is equal to \begin{equation}
    E=\frac{1}{2}(K_{cyl}+K_q+K_4+K_3+K_{1t}+K_{2t}).
\end{equation}

So $3F-2E=2K_{cyl} - \frac{1}{4}K_q + \frac{1}{2}K_4$ and $F=\frac{2}{3}E+\frac{2}{3}K_{cyl}-\frac{1}{12}K_q+\frac{1}{6}K_4$. We can compute each term in the formula as follows:

$$K_{cyl}=\sum_{t,\tau,{\bf C}}\frac{1}{2t}t  = \frac{1}{2}p!B;$$

$$K_q=\sum_{\tau,{\bf C}}\left[\frac{1}{2p}(p^2-p)+\sum_{t=1}^{p-1}\frac{1}{2t} t\right] =(p-1)p!B;$$

$$K_4=\sum_{t,\tau,{\bf C}}\frac{1}{2t} \left[\sum _{i=1}^{t}2C_{\tau(i)} (b_{\tau(i)}-C_{\tau(i)})-b_{\tau(i)}+1 \right].$$
As in the proof of Corollary~\ref{edgecount}, we use the symmetry of this summation. In $\sum_{t,\tau}\frac{1}{2t}\left[\sum _{i=1}^{t}2C_{\tau(i)} (b_{\tau(i)}-C_{\tau(i)})\right]$, each term $C_i(b_i-C_i)$, for $i=1,\dots, p$, appears the same number of times. So it is equal to $p!\sum_{i=1}^p C_i(b_i-C_i)$. Similarly, in $\sum_{t,\tau}\frac{1}{2t} b_{\tau(i)}$, each form $b_i$, for $i=1,\dots,p$, appears the same number of times. So it is equal to $\frac{p!}{2}a$. By taking summation over ${\bf C}$, we obtain $$K_4=\frac{p!B}{6}(A-2a+3p).$$

Now we take care of the cases where $p=1,2$ and all $b_i$ are even, or $\RR_1(\mu)$ has a cell with an orbifold point (see Proposition~\ref{prop:ExcSymCel}). There are two correction terms $\varepsilon_2$ and $\varepsilon'_2$. The first one comes from the hyperelliptic involutions and the second one comes from the cells containing orbifold points. 

First, we compute the correction term $\varepsilon_2$. We will use the formula~\ref{face}. Recall from Remark~\ref{multiplecount} that for certain 0-cells, only half of the prong-matchings contribute to the above formula. Assume $C_i=\frac{b_i}{2}$.

Suppose $p=1$. We need to take care of the multi-scale differential $X(1,Id,{\bf C},[(0,0)])$. Around this, there is one cell of cylinder type and $a-1$ cells of 4-degenerate type. So the correction term is equal to $$\varepsilon_2=\frac{1}{2}\left(1+\frac{1}{2}(a-1)\right)=\frac{a+1}{4}.$$

Suppose $p=2$. Around the multi-scale differential $X(1,Id,{\bf C},[(0,0)])$, there are two cells of quadrilateral type and $b_1-2$ cells of 4-degenerate type. Hence we have a correction term $\frac{1}{2}\left(\frac{2}{4}+\frac{1}{2}(b_1-2)\right)=\frac{b_1-1}{4}$. The same argument applies to $X(1,(1,2),{\bf C},[(0,0)])$, and this gives another term $\frac{b_2-1}{4}$. So the correction term is equal to $$\varepsilon_2=\frac{a-1}{2}.$$

Now we compute the second correction term $\varepsilon'_2$. Suppose $p=1$. If $a\in 3\mathbb{N}^{\ast}$, then there exists one 2-cell of 1-triangle type with an orbifold point of order 3. This cell is bounded by only one equatorial arc, so we need a correction term $\frac{2}{3}$. Also if $a\in 2\mathbb{N}^{\ast}$, then there exists one 2-cell of 4-degenerate type with an orbifold point of order 2. This cell is also bounded by only one equatorial arc, so we need a correction term $\frac{1}{2}$. Suppose $p=2$. If $b_1,b_2\in 3\mathbb{N}^{\ast}$ or $b_1,b_2\in 3\mathbb{N}+2$, then there exists one 2-cell of 3-degenerate type with an orbifold point of order 3. This cell is bounded by only one equatorial arc, so we need a correction term $\frac{2}{3}$. 
\end{proof}

Finally, using the cellular decomposition, we can compute the Euler characteristic of $\PP\overline{\RR}_1(\mu)$.

\begin{equation}
\chi \left(\PP\overline{\RR}_1(\mu) \right) = F-E+V+(\varepsilon_0-\varepsilon_1+\varepsilon_2+\eta)=\frac{p!B}{48}(A-a^2-4a+20)+\sum_{t,\tau,{\bf C}} \frac{1}{2t}\gcd(Q_1,Q_2)+\varepsilon+\eta
\end{equation}

where $\varepsilon\coloneqq \varepsilon_0-\varepsilon_1+\varepsilon_2.$ In particular, if $p\geq 3$, then $\eta=0$ and $\sum_{t,\tau,{\bf C}} \frac{1}{2t}\gcd(Q_1,Q_2)+\varepsilon$ counts the number of non-cusp points in the boundary. This proves Theorem~\ref{main2+}.

\begin{rmk}
If the stratum $\PP\RR_1(\mu)$ is connected, which is true for general $\mu$ (see Theorem~1.7 of \cite{Lee}), then Theorem~\ref{main2} gives the genus formula for the stratum: $$g(\PP\overline{\RR}_1(\mu)) =1- \frac{1}{2}\chi\left(\PP\overline{\RR}_1(\mu) \right).$$ Otherwise, the formula gives the sum of the Euler characteristics for all connected components of $\PP\overline{\RR}_1(\mu)$. In order to compute the Euler characteristic of each connected component, we need to classify 0,1 and 2-cells in terms of rotation number and ramification profile. This is not a hard task for a given $\mu$, but no reasonably short general formula seems to hold for every $\mu$.
\end{rmk}

\subsection{The orbifold Euler characteristic}

Since we have classified the orbifold points of $\PP\overline{\RR}_1(\mu)$, we can easily compute the orbifold Euler characteristic $\chi^{orb} \left(\PP\overline{\RR}_1(\mu) \right)$ now. The orbifold Euler characteristic accounts for the local group action by weighting contributions from the orbifold points. So in the alternating sum, the contribution of an orbifold point of order $k$ is equal to $\frac{1}{k}$. 

If $p\geq 3$, then $\PP\overline{\RR}_1(\mu)$ has no orbifold point and has exactly $K_{cyl}$ cusp (one in each cell of cylinder type). So $\chi^{orb} \left(\PP\overline{\RR}_1(\mu) \right)=\chi \left(\PP\overline{\RR}_1(\mu) \right)-K_{cyl}$. The case $p=1$ is already described in \cite{Ta1}, we focus on the case $p=2$ here.
For each orbifold point of finite order $k$, we need to subtract $\frac{k-1}{k}$ from the usual Euler characteristic. However, it is immediate that the orbifold term $\eta$ from Proposition~\ref{facecount} exactly computes the amount we have to subtract. Also, each cell of cylinder type contains one cusp. So we need to further subtract $K_{cyl}$. Recall that $K_{cyl}=\frac{1}{2}p!B=(b_1-1)(b_2-1)$. Therefore, the orbifold Euler characteristic $\chi^{orb} \left(\PP\overline{\RR}_1(\mu) \right)$ is equal to 

\begin{equation}
-\frac{(b_1-1)(b_2-1)}{12}(b_1 b_2+2(b_1+b_2)+2)+\sum_{t,\tau,{\bf C}} \frac{1}{2t}\gcd(Q_1,Q_2)+\varepsilon
\end{equation} and this proves Theorem~\ref{main2}.

\begin{rmk}
When $p=2$ and $a=b_1+b_2$ is prime, we can obtain a rather simpler formula for the orbifold Euler characteristic. Note that for any $1\leq c_1\leq b_1-1$ and $1\leq c_2\leq b_2-1$, we have $\gcd(c_1+c_2,a)=1$. So
$$\sum_{\substack{1 \leq c_{1} \leq b_{1}-1 \\ 1 \leq c_{2} \leq b_{2}-1\phantom{-}}}  \gcd \left( c_{1}+c_{2}, a \right) =(b_1-1)(b_2-1).$$
Also, $\sum_{c_1=1}^{b_1-1} \gcd(c_1,b_1)+\sum_{c_2=1}^{b_2-1} \gcd(c_2,b_2)=\sum_{d|b_1} d\phi(\frac{b_1}{d}) + \sum_{e|b_2} e\phi(\frac{b_2}{e})$. Finally, we have $$\chi^{orb} \left(\PP\overline{\RR}_1(\mu) \right)=-\frac{(b_1-1)(b_2-1)}{12}(b_1 b_2+2(b_1+b_2)-4)+\frac{1}{2}\sum_{d|b_1} d\phi(\frac{b_1}{d}) + \frac{1}{2}\sum_{e|b_2} e\phi(\frac{b_2}{e}).$$
\end{rmk}

\section{Degree of the forgetful map $\PP\RR_1(\mu)\to \mathcal{M}_{1,1}$}\label{sec:Degree}

Let $(X,\omega)\in \RR_1(\mu)=\RR_1(a,-b_1,\dots, -b_p)$. By forgetting $\omega$ but remembering the unique zero $z$ of order $a=b_1+\dots+b_p$, we obtain a pointed elliptic curve $(X,z)\in \mathcal{M}_{1,1}$. This gives a forgetful map $\pi:\PP\RR(\mu)\to \mathcal{M}_{1,1}$. Since both $\PP\RR(\mu)$ and $\mathcal{M}_{1,1}$ are Riemann surfaces (with orbifold points), the map $\pi$ extends to $\PP\RR_1(\mu)\to \overline{\mathcal{M}}_{1,1}$, also denoted by $\pi$. This is a ramified covering of $\overline{\mathcal{M}}_{1,1}\cong \mathbb{CP}^1$. The (orbifold) degree of $\pi$ for some special cases is computed in \cite{CG} and \cite{EGMP}. We can generalize their argument to compute the orbifold degree of $\pi$ for an arbitrary $\mu$ in terms of $a=b_{1}+\dots+b_{p}$ and $B=\prod_{i} (b_i-1)$. 

\begin{thm}\label{thm:degree}
The degree of the forgetful map $\pi:\PP\RR_1(\mu)\to \mathcal{M}_{1,1}$ is equal to $(a+1) p!B.$
\end{thm}

\begin{proof}
We will compute the degree of $\pi^{-1}(\infty)$, where $\infty \in \overline{\mathcal{M}}_{1,1}$ is the boundary point that parameterizes the rational nodal curve. This is the number of multi-scale differentials in the boundary of $\PP\RR_1(\mu)$, counted up to multiplicity. A multi-scale differential $\overline{X}$ over $\infty$ is of the form $X(t,\tau, {\bf C},[(u,v)])$. Note that every element of $\overline{\mathcal{M}}_{1,1}$ has an involution, corresponding to the group inversion. At the point $\infty$, this involution is given by the automorphism interchanging the two points over the node after normalization. So we may label these two points to resolve orbifoldness of $\overline{\mathcal{M}}_{1,1}$. 

If $t=0$, then the underlying curve of $\overline{X}$ is the rational nodal curve. So the map $\pi$ is not ramified at these points. Here, two simple poles at the nodes are labeled since it comes from the node of $\infty$. There are $p!$ choices for $\tau$ and $b_i-1$ choices for each $C_i$. Thus we have $p!B$ such differentials. 

If $t>0$, then the underlying curve is the banana curve, the union of two rational curves intersecting at two nodes $s_1$ and $s_2$. After forgetting all the poles, this is a semi-stable nodal curve as one rational component only contains two special points, which are the nodes. For given $t$, $\tau$ and ${\bf C}$, the number of prongs at the nodes $s_1$ and $s_2$ are given by $Q_1\coloneqq \sum_{i=1}^t C_{\tau(i)}$ and $Q_2\coloneqq \sum_{i=1}^t [b_{\tau(i)}-C_{\tau(i)}]$, respectively. There are $Q_1 Q_2$ prong-matchings and the number of prong-matching equivalence classes is equal to $\gcd(Q_1,Q_2)$.

Denote $a_1\coloneqq \lcm(Q_1,Q_2)/Q_1$ and $a_2\coloneqq \operatorname{lcm}(Q_1,Q_2)/Q_2$. Let $U\subset \PP\RR_1(\mu)$ be a local coordinate neighborhood parameterized by $u$ centered at $\overline{X}$. Then the induced family of semi-stable curve $f_{\pi} :\chi\to \PP\RR_1(\mu)$ is given by the local equations of the form $x_1 y_1=u^{a_1}$ at $s_1$, and $x_2 y_2=c u^{a_2}$ with some $|c|=1$ at $s_2$. By Lemma~5.1 of \cite{CG}, $\pi$ is ramified at $\overline{X}$ of degree $a_1+a_2=(Q_1+Q_2)/\gcd(Q_1,Q_2)$. Recall from Remark~\ref{multiplecount} that each multi-scale differential is counted $t$ times by labeling of the saddle connections of the top level component. Here, two nodes $s_1,s_2$ are considered labeled. Therefore, up to multiplicity, the number of multi-scale differentials on the banana curve over $\infty \in \overline{\mathcal{M}}_{1,1}$ is equal to $$\sum_{t=1}^p \frac{1}{t} \sum_{\tau} \sum_{1\leq C_i\leq b_i-1} (Q_1+Q_2) = \sum_\tau \sum_{t=1}^p \frac{1}{t}  \sum_{\bf C} \sum_{i=1}^t b_{\tau(i)}=p! B\left(\sum_{i=1}^p b_i\right)=a p!B.$$
Therefore the degree of $\pi$ is $(a+1)p!B$.
\end{proof}

\begin{lem}
Fix $r~\vert~d$. Let $\RC_{\geq r}\subset \PP\overline{\RR}_1(\mu)$ be the union of connected components with rotation number divisible by $r$. Then the degree of the forgetful map $\pi$ restricted to $\RC_{\geq r}$ is equal to 
$$\frac{1}{r^2}p![(a+1)B-(-1)^p]+p!(-1)^p.$$
\end{lem}

\begin{proof}
Fix $1\leq t\leq p$ and $\tau\in S_{p}$. Consider multi-scale differentials of the form $\overline{X}=X(t,\tau,{\bf C},[(0,v)])$. For a given ${\bf C}$, there are $\gcd(Q_1,Q_2)$ choices of prong-matching equivalence classes. They are represented by $(0,v)$ for $v=1,\dots, \gcd(Q_1,Q_2)$. 

Using Proposition~\ref{rot}, the rotation number of $X(t,\tau,{\bf C},[(0,v)])$ is given by $\gcd(d,Q_1,\sum_{i=1}^p C_i +v)$. If this is divisible by $r$, then ${\bf C}$ satisfies $r~\vert~\sum_{i=1}^t C_{\tau(i)}$ and $v$ satisfies $r~\vert~\sum_{i=1}^p C_i +v$. Note that there are always $\frac{1}{r}\gcd(Q_1,Q_2)$ choices for $v$. As in the proof of Theorem~\ref{thm:degree}, the map $\pi$ is ramified at these points to the degree $\frac{Q_1+Q_2}{\gcd(Q_1,Q_2)}$. Therefore, the number $D_1$ of two-level multi-scale differentials (up to multiplicity) in $\overline{\RC}_{\geq r}$ over $\infty\in \overline{\mathcal{M}}_{1,1}$ is equal to

$$\frac{1}{r}\sum_{t,\tau} \sum_{{\bf C},r~\vert~\sum_{i=1}^t C_i} \frac{1}{t} (Q_1+Q_2) = \frac{1}{r}\sum_{t,\tau} \sum_{{\bf C},r~\vert~\sum_{i=1}^t C_{\tau(i)}} \frac{1}{t}\sum_{i=1}^t b_{\tau(i)}.$$

Thus we need to count the number of choices for $1\leq C_{\tau(i)}\leq b_{\tau(i)}-1$, $i=1,\dots,t$, such that $r~\vert~\sum_{i=1}^t C_{\tau(i)}$. For convenience, we assume that $\tau=Id$. We claim that this number is equal to $\frac{1}{r}\left[\prod_{i=1}^t (b_i-1) -(-1)^t\right]+(-1)^t$. In order to compute this, first consider $A\coloneqq \{(C_j)|0\leq C_i\leq b_i-1,r|\sum_{i=1}^t C_i\}$, a set slightly larger than what we are looking for. The set $A$ has $\frac{1}{r}\prod_{i=1}^t b_i $ elements, since whenever $C_1,\dots,C_{t-1}$ are fixed, there are exactly $\frac{b_t}{r}$ choices for $C_t$ such that $r~\vert~\prod_{i=1}^t C_i$. Denote $A_k=\coloneqq \{(C_j)\in A|C_k=0\}$. Then $A_k$ has $\frac{1}{rb_i} \prod_{i=1}^t b_i$ elements. Similarly, for any proper subset $I\subset \{1,\dots, t\}$, the set $\{(C_j)\in A|C_k=0,k\in I\}=\cap_{k\in I}A_k$ has $\frac{1}{r}\prod_{i\notin I} b_i $ elements. Therefore, the number we want can be computed by the following alternating sum $$|A|-\sum_{i=1}^t |A_i|+\sum_{i\neq j} |A_i\cap A_j|+\dots+(-1)^t = \frac{1}{r}\left[\prod_{i=1}^t (b_i-1) -(-1)^t\right]+(-1)^t.$$

Now we can continue to compute $$D_1=\frac{1}{r}\sum_{t,\tau} \frac{1}{t}\sum_{i=1}^t b_{\tau(i)}\left(\frac{1}{r}\left[\prod_{i=1}^t (b_{\tau(i)}-1) -(-1)^t\right]+(-1)^t\right)\prod_{j=t+1}^p (b_{\tau(j)}-1)$$ $$=\frac{1}{r^2}\sum_{t,\tau}\frac{1}{t} B \sum_{i=1}^t b_{\tau(i)} +\frac{r-1}{r^2}\sum_{t,\tau}\frac{(-1)^t}{t}\sum_{i=1}^t b_{\tau(i)} \prod_{j=t+1}^p (b_{\tau(j)}-1).$$

The first term is equal to $\frac{1}{r^2}ap!B$. It remains to compute the second term. Now we fix $t$, and consider $$\sum_{\tau}\frac{(-1)^t}{t}\sum_{i=1}^t b_{\tau(i)} \prod_{j=t+1}^p (b_{\tau(j)}-1) = \sum_{\tau}\left[\frac{(-1)^t}{t}\sum_{i=1}^t (b_{\tau(i)}-1) \prod_{j=t+1}^p (b_{\tau(j)}-1) + (-1)^t \prod_{j=t+1}^p (b_{\tau(j)}-1)\right].$$

The first term in the sum is symmetric, and it computes $$(-1)^t (t-1)!(p-t+1)! \sum_{|I|=p-t+1} \prod_{i\in I}(b_i-1)$$ where $I\subset \{1,\dots, p\}$ runs over all subsets with $p-t+1$ elements. The second term in the sum is also symmetric, and this computes $$(-1)^t t! (p-t)! \sum_{|J|=p-t} \prod_{i\in J}(b_i-1)$$ where $J\subset \{1,\dots, p\}$ runs over all subsets with $p-t$ elements. Thus by taking sum over $t$, most terms are canceled and we obtain $p!(-B+(-1)^p)$. Finally, we have $$\frac{1}{r^2}ap!B + \frac{r-1}{r^2}p!(-B+(-1)^p) = \frac{1}{r^2}p!\left[(a+1)B-(-1)^p\right]-\frac{1}{r}p!(B-(-1)^p).$$

The number $D_2$ of horizontal multi-scale differentials in $\overline{\RC}_{\geq r}$ over $\infty \in \overline{M}_{1,1}$ is equal to 

$$\sum_{\tau} \sum_{{\bf C},r|\sum_{i=1}^p C_i} 1=\frac{1}{r}p!(B-(-1)^p)+p!(-1)^p.$$ Therefore, the degree $D_1+D_2$ of $\pi$ restricted to $\RC_{\geq r}$ is equal to $$ \frac{1}{r^2}p!\left[(a+1)B-(-1)^p\right] + p!(-1)^p.$$
\end{proof}

As a result, we can compute the degree of $\pi$ restricted to the union $\RC_r\subset \PP\overline{\RR}_1(\mu)$ of components with given rotation number:
$$\frac{1}{r^2} p![(a+1)B-(-1)^p]\prod_{q|\frac{d}{r}} \left(1-\frac{1}{q^2}\right)$$ where $q$ runs over all prime divisors of $\frac{d}{r}$. 

Combined with Theorem~\ref{thm:degree}, this proves Theorem~\ref{main3}.\newline

In particular, the degree of the map $\RC_1\to \mathcal{M}_{1,1}$ is equal to $$p![(a+1)B-(-1)^p]\prod_{q~\vert~d} \left(1-\frac{1}{q^2}\right).$$ Observe that this formula directly generalizes the degree formula $N^2\prod_{q~\vert~N}\left(1-\frac{1}{q^2}\right)$ of the forgetful map from a modular curve $X_{1}(N)$. 

\begin{rmk}
It is easy to check this formula generalizes already known cases below:
\begin{itemize}
    \item The degree of $\PP\RR_1(a+2,-2,-a) \to \mathcal{M}_{1,1}$ is equal to $2(a+3)(a-1)$, as given in \cite{CG}.
    \item The degree of $\PP\RR_1(a,-a) \to \mathcal{M}_{1,1}$ is equal to $(a+1)(a-1)$, which is the number of nontrivial $a$-torsion points on an elliptic curve. 
    \item The degree of $\PP\RR_1(2p,-2^p)\to \mathcal{M}_{1,1}$ is equal to $p!(2p+1)$. By forgetting the labeling of the double poles, we obtain the degree equal to $2p+1$. This is computed in \cite{EGMP}. 
\end{itemize}
\end{rmk}

\begin{rmk}
Instead of the unique zero, we can mark the pole of order $b_i$ and obtain a similar forgetful map $\pi_i : \PP\RR_1(\mu)\to \mathcal{M}_{1,1}$. In fact, the proofs of Theorem~\ref{thm:degree} and Theorem~\ref{main3} do not use the fact that the zero is the marked point. Therefore, the degree of the map $\pi_i$ is equal to the degree of $\pi$, for any $i=1,\dots,p$. This can be also deduced from the fact that for any two points $x_1,x_2$ on an elliptic curve $E$ with a fixed identity element $e\in E$, there is an isomorphism between curves with one marked point $(E,x_1)$ and $(E,x_2)$ given by the translation $x\mapsto x+(x_2-x_1)$. So for general $(E,x)$, there is a bijection between the elements of $\pi^{-1}(E,x)$ and $\pi_i^{-1}(E,x)$. 
\end{rmk}

\section{The exceptional stratum $\RR_1(12,-3,-3,-3,-3)$}\label{sec:Except}

Following results of \cite{Lee}, it is already known that $\RR_1(12,-3,-3,-3,-3)$ has $3$ hyperelliptic components and a unique non-hyperelliptic component with rotation number $1$. In order to prove Theorem~\ref{main1}, it remains to prove that there are exactly {\em two} non-hyperelliptic components with rotation number $3$. 

\begin{proof}[Proof of Theorem~\ref{main1}]
Let $\RC$ be a non-hyperelliptic component of $\PP\overline{\RR}_1(12,-3,-3,-3,-3)$ with rotation number $3$. Up to relabeling the poles, $\RC$ contains a multi-scale differential of the form $X(4,Id, {\bf C}, [(0,v)])$. By Proposition~\ref{rot}, this satisfies $\sum_{i=1}^4 C_i=6$ and $v=0$ or $3$. Also, since $[(0,3)]=[(C_1,3-C_1)]$, we have $X(4,Id, {\bf C}, [(0,3)])=X(4,(1,2,3,4), {\bf C},[(0,0)])$. So we may assume $v=0$.

Moreover, if $C_1+C_4=C_2+C_3=3$, then $X(4,Id,{\bf C},[(0,0)])$ is contained in the boundary of hyperelliptic components. Also, $X(4,Id,(1,2,2,1),[(0,0)])=X(4,(41)(23),(2,1,1,2),[(0,0)])$. Therefore, we can conclude that $\RC$ contains $X(4,\tau,(1,2,2,1),[(0,0)])$ for some $\tau\in S_{4}$. Let $X_{\tau}$ denote this multi-scale differential. We claim that $X_{\tau_1} = X_{\tau_2}$ if and only if $\tau_1 \circ \tau_2^{-1}\in A_{4} \trianglelefteq S_{4}$. 

Consider $X_{Id}=X(4,Id,(1,2,2,1),[(0,0)])$. It is a cusp of a cell of cylinder type, which we denote by $\mathrm{Cyl}(1,2,3,4)$. It is bounded by one equatorial arc, denoted by $DA(1,2,3,4)$. Note that the numbers are labels of the poles. We can think of the cells and arcs with the same translation structure with other labeling. Now we will describe incidence relations between the cells and the arcs, starting from $DA(1,2,3,4)$, until we exhaust all cells. The cells are in Figure~\ref{fig_cells} and the arcs are listed in Figures~\ref{fig_arcs}~and~\ref{fig_arcs2}.

\begin{figure}
    \centering
    \input{diagrams/cells} 
    \caption{The cells of the connected component} \label{fig_cells}
\end{figure}

\begin{figure}
    \centering
    \input{diagrams/arcs} 
    \caption{The arcs of the connected component} \label{fig_arcs}
\end{figure}

\begin{figure}
    \centering
    \input{diagrams/arcs2} 
    \caption{The arcs of the connected component (continued)} \label{fig_arcs2}
\end{figure}

\begin{itemize}
    \item The arc $DA(1,2,3,4)$ also bounds a cell of 3-degenerate type, denoted by $U(1,2,3,4)$. This cell is also bounded by two more arcs, denoted by $DB(1,2,3,4)$ and $DC(1,2,3,4)$. 
    \item The arc $DB(1,2,3,4)$ also bounds a cell of 1-triangle type, denoted by $T(1,2,3,4)$. This cell is also bounded by two more arcs, denoted by $DE(1,2,3,4)$ and $DF(1,2,3,4)$.
    \item The arc $DB(1,2,3,4)$ also bounds a cell of 1-triangle type, denoted by $S(1,2,3,4)$. This cell is also bounded by two more arcs, denoted by $DG(1,2,3,4)$ and $DH(1,2,3,4)$.
    \item The arc $DE(1,2,3,4)$ also bounds a cell of 4-degenerate type, denoted by $Q(1,2,3,4)$. This cell is also bounded by one more arc, denoted by $DI(1,2,3,4)$.
    \item The arc $DF(1,2,3,4)$ also bounds a cell of 3-degenerate type, denoted by $R(1,2,3,4)$. This cell is also bounded by two more arcs. One of them is $DH(4,3,2,1)$ (i.e. $DH$ with a different labeling). The other one is denoted by $DJ(1,2,3,4)$.
    \item The arc $DG(1,2,3,4)$ also bounds a cell of 4-degenerate type, denoted by $P(1,2,3,4)$. This cell is also bounded by one more arc, denoted by $DK(1,2,3,4)$. 
    \item The arc $DJ(1,2,3,4)$ also bounds a cell of 2-triangle type, denoted by $N(1,2,3,4)$. This cell is also bounded by two more arcs, $DI(1,3,4,2)$ and $DK(1,3,4,2)$. 
\end{itemize}

So we exhausted all cells of the connected component containing $X_{Id}$, up to relabeling the poles. The incidence relation between the cells and the arcs, up to labeling, is depicted in Figure~\ref{fig_graph}.

\begin{figure}
    \centering
    \input{diagrams/cell_graph} 
    \caption{Incidence relation of the cells and the arcs} \label{fig_graph}
\end{figure}

When we travel around a cycle of the graph, the labeling of the pole is changed by the permutation written inside each cycle. The permutations in the graph generates $A_{4}$. Therefore, we can conclude that $X_{Id}$ and $X_{\tau}$ are contained in the same connected component if and only if $\tau\in A_{4}$. This proves that there are two non-hyperelliptic connected components of $\RR_1(12,-3,-3,-3,-3)$, corresponding to the cosets in $S_{4}/A_{4}$. It is obvious that the two connected components have same canonical complex projective structure, since they just differ by the labeling of the poles. 
\end{proof}

\end{document}

%% file: diagrams/types.tex
\tikzset{every picture/.style={line width=0.75pt}} 

\begin{tikzpicture}[x=0.75pt,y=0.75pt,yscale=-1,xscale=1]

\draw  [fill={rgb, 255:red, 200; green, 200; blue, 200 }  ,fill opacity=1 ] (280.61,504.32) -- (102.49,339.99) -- (280.61,339.99) -- cycle ;
\draw [fill={rgb, 255:red, 255; green, 255; blue, 255 }  ,fill opacity=1 ]   (280.61,504.32) .. controls (257.64,413.99) and (177.05,357.65) .. (102.49,339.99) ;
\draw    (280.61,504.32) .. controls (191.36,488.58) and (117.56,429.06) .. (102.49,339.99) ;
\draw    (102.49,339.99) .. controls (58.81,370.7) and (59.19,480.51) .. (102.49,504.32) ;
\draw    (102.49,504.32) .. controls (149.94,543.87) and (241.82,539.26) .. (280.61,504.32) ;
\draw    (102.49,339.99) -- (280.61,504.32) ;
\draw    (102.49,339.99) .. controls (26.05,384.53) and (21.53,472.07) .. (102.49,504.32) ;
\draw    (102.49,504.32) .. controls (130.74,575.73) and (252.37,575.73) .. (280.61,504.32) ;
\draw   (67.1,424.46) -- (69.73,419.47) -- (72.37,424.46) ;
\draw   (99.86,424.46) -- (102.49,419.47) -- (105.13,424.46) ;
\draw   (277.97,423.69) -- (280.61,418.7) -- (283.24,423.69) ;
\draw   (41.12,423.31) -- (43.75,418.31) -- (46.39,423.31) ;
\draw  [fill={rgb, 255:red, 0; green, 0; blue, 0 }  ,fill opacity=1 ] (278.16,504.32) .. controls (278.16,502.94) and (279.26,501.82) .. (280.61,501.82) .. controls (281.96,501.82) and (283.06,502.94) .. (283.06,504.32) .. controls (283.06,505.7) and (281.96,506.81) .. (280.61,506.81) .. controls (279.26,506.81) and (278.16,505.7) .. (278.16,504.32) -- cycle ;
\draw  [fill={rgb, 255:red, 0; green, 0; blue, 0 }  ,fill opacity=1 ] (278.16,339.99) .. controls (278.16,338.61) and (279.26,337.49) .. (280.61,337.49) .. controls (281.96,337.49) and (283.06,338.61) .. (283.06,339.99) .. controls (283.06,341.37) and (281.96,342.48) .. (280.61,342.48) .. controls (279.26,342.48) and (278.16,341.37) .. (278.16,339.99) -- cycle ;
\draw  [fill={rgb, 255:red, 0; green, 0; blue, 0 }  ,fill opacity=1 ] (100.05,339.99) .. controls (100.05,338.61) and (101.14,337.49) .. (102.49,337.49) .. controls (103.85,337.49) and (104.94,338.61) .. (104.94,339.99) .. controls (104.94,341.37) and (103.85,342.48) .. (102.49,342.48) .. controls (101.14,342.48) and (100.05,341.37) .. (100.05,339.99) -- cycle ;
\draw  [fill={rgb, 255:red, 0; green, 0; blue, 0 }  ,fill opacity=1 ] (100.05,504.32) .. controls (100.05,502.94) and (101.14,501.82) .. (102.49,501.82) .. controls (103.85,501.82) and (104.94,502.94) .. (104.94,504.32) .. controls (104.94,505.7) and (103.85,506.81) .. (102.49,506.81) .. controls (101.14,506.81) and (100.05,505.7) .. (100.05,504.32) -- cycle ;
\draw [fill={rgb, 255:red, 0; green, 0; blue, 0 }  ,fill opacity=1 ]   (280.61,339.99) -- (280.61,504.32) ;
\draw    (102.49,339.99) -- (102.49,504.32) ;
\draw    (102.49,504.32) -- (280.61,504.32) ;
\draw  [fill={rgb, 255:red, 200; green, 200; blue, 200 }  ,fill opacity=1 ] (101.89,637.33) -- (280.01,801.12) -- (101.89,801.12) -- cycle ;
\draw  [fill={rgb, 255:red, 200; green, 200; blue, 200 }  ,fill opacity=1 ] (280.01,801.12) -- (101.89,636.79) -- (280.01,636.79) -- cycle ;
\draw [fill={rgb, 255:red, 255; green, 255; blue, 255 }  ,fill opacity=1 ]   (280.01,801.12) .. controls (257.04,710.79) and (176.45,654.45) .. (101.89,636.79) ;
\draw [fill={rgb, 255:red, 254; green, 254; blue, 254 }  ,fill opacity=1 ]   (280.01,801.12) .. controls (190.76,785.38) and (116.96,725.86) .. (101.89,636.79) ;
\draw    (101.89,636.79) .. controls (58.21,667.5) and (58.59,777.31) .. (101.89,801.12) ;
\draw    (101.89,801.12) .. controls (149.34,840.67) and (241.22,836.06) .. (280.01,801.12) ;
\draw    (101.89,636.79) -- (280.01,801.12) ;
\draw    (101.89,636.79) .. controls (25.45,681.33) and (20.93,768.87) .. (101.89,801.12) ;
\draw    (101.89,801.12) .. controls (130.14,872.53) and (251.77,872.53) .. (280.01,801.12) ;
\draw   (66.5,721.26) -- (69.13,716.27) -- (71.77,721.26) ;
\draw   (99.26,721.26) -- (101.89,716.27) -- (104.53,721.26) ;
\draw   (277.37,720.49) -- (280.01,715.5) -- (282.64,720.49) ;
\draw   (40.52,720.11) -- (43.15,715.11) -- (45.79,720.11) ;
\draw  [fill={rgb, 255:red, 0; green, 0; blue, 0 }  ,fill opacity=1 ] (277.56,801.12) .. controls (277.56,799.74) and (278.66,798.62) .. (280.01,798.62) .. controls (281.36,798.62) and (282.46,799.74) .. (282.46,801.12) .. controls (282.46,802.5) and (281.36,803.61) .. (280.01,803.61) .. controls (278.66,803.61) and (277.56,802.5) .. (277.56,801.12) -- cycle ;
\draw  [fill={rgb, 255:red, 0; green, 0; blue, 0 }  ,fill opacity=1 ] (277.56,636.79) .. controls (277.56,635.41) and (278.66,634.29) .. (280.01,634.29) .. controls (281.36,634.29) and (282.46,635.41) .. (282.46,636.79) .. controls (282.46,638.17) and (281.36,639.28) .. (280.01,639.28) .. controls (278.66,639.28) and (277.56,638.17) .. (277.56,636.79) -- cycle ;
\draw  [fill={rgb, 255:red, 0; green, 0; blue, 0 }  ,fill opacity=1 ] (99.45,636.79) .. controls (99.45,635.41) and (100.54,634.29) .. (101.89,634.29) .. controls (103.25,634.29) and (104.34,635.41) .. (104.34,636.79) .. controls (104.34,638.17) and (103.25,639.28) .. (101.89,639.28) .. controls (100.54,639.28) and (99.45,638.17) .. (99.45,636.79) -- cycle ;
\draw  [fill={rgb, 255:red, 0; green, 0; blue, 0 }  ,fill opacity=1 ] (99.45,801.12) .. controls (99.45,799.74) and (100.54,798.62) .. (101.89,798.62) .. controls (103.25,798.62) and (104.34,799.74) .. (104.34,801.12) .. controls (104.34,802.5) and (103.25,803.61) .. (101.89,803.61) .. controls (100.54,803.61) and (99.45,802.5) .. (99.45,801.12) -- cycle ;
\draw [fill={rgb, 255:red, 0; green, 0; blue, 0 }  ,fill opacity=1 ]   (280.01,636.79) -- (280.01,801.12) ;
\draw    (101.89,636.79) -- (101.89,801.12) ;
\draw    (101.89,801.12) -- (280.01,801.12) ;
\draw   (101.1,36.88) -- (280.82,36.88) -- (280.82,203.05) -- (101.1,203.05) -- cycle ;
\draw    (280.82,203.05) .. controls (262.96,116.86) and (176.33,54.74) .. (101.1,36.88) ;
\draw    (280.82,203.05) .. controls (190.77,187.14) and (116.3,126.96) .. (101.1,36.88) ;
\draw    (101.1,36.88) .. controls (57.02,67.94) and (57.4,178.98) .. (101.1,203.05) ;
\draw    (101.1,203.05) .. controls (148.97,243.04) and (241.69,238.38) .. (280.82,203.05) ;
\draw    (101.1,36.88) -- (280.82,203.05) ;
\draw    (101.1,36.88) .. controls (23.97,81.92) and (19.41,170.44) .. (101.1,203.05) ;
\draw    (101.1,203.05) .. controls (129.6,275.27) and (252.32,275.27) .. (280.82,203.05) ;
\draw   (65.38,122.3) -- (68.04,117.25) -- (70.7,122.3) ;
\draw   (98.44,122.3) -- (101.1,117.25) -- (103.76,122.3) ;
\draw   (278.16,121.52) -- (280.82,116.47) -- (283.48,121.52) ;
\draw   (39.17,121.13) -- (41.83,116.09) -- (44.48,121.13) ;
\draw  [fill={rgb, 255:red, 0; green, 0; blue, 0 }  ,fill opacity=1 ] (278.35,203.05) .. controls (278.35,201.66) and (279.46,200.53) .. (280.82,200.53) .. controls (282.19,200.53) and (283.29,201.66) .. (283.29,203.05) .. controls (283.29,204.45) and (282.19,205.58) .. (280.82,205.58) .. controls (279.46,205.58) and (278.35,204.45) .. (278.35,203.05) -- cycle ;
\draw  [fill={rgb, 255:red, 0; green, 0; blue, 0 }  ,fill opacity=1 ] (278.35,36.88) .. controls (278.35,35.49) and (279.46,34.36) .. (280.82,34.36) .. controls (282.19,34.36) and (283.29,35.49) .. (283.29,36.88) .. controls (283.29,38.28) and (282.19,39.41) .. (280.82,39.41) .. controls (279.46,39.41) and (278.35,38.28) .. (278.35,36.88) -- cycle ;
\draw  [fill={rgb, 255:red, 0; green, 0; blue, 0 }  ,fill opacity=1 ] (98.63,36.88) .. controls (98.63,35.49) and (99.74,34.36) .. (101.1,34.36) .. controls (102.46,34.36) and (103.57,35.49) .. (103.57,36.88) .. controls (103.57,38.28) and (102.46,39.41) .. (101.1,39.41) .. controls (99.74,39.41) and (98.63,38.28) .. (98.63,36.88) -- cycle ;
\draw  [fill={rgb, 255:red, 0; green, 0; blue, 0 }  ,fill opacity=1 ] (98.63,203.05) .. controls (98.63,201.66) and (99.74,200.53) .. (101.1,200.53) .. controls (102.46,200.53) and (103.57,201.66) .. (103.57,203.05) .. controls (103.57,204.45) and (102.46,205.58) .. (101.1,205.58) .. controls (99.74,205.58) and (98.63,204.45) .. (98.63,203.05) -- cycle ;
\draw   (419.5,35.28) -- (599.22,35.28) -- (599.22,201.45) -- (419.5,201.45) -- cycle ;
\draw    (419.5,35.28) .. controls (375.42,66.34) and (375.8,177.38) .. (419.5,201.45) ;
\draw    (419.5,201.45) .. controls (467.37,241.44) and (560.09,236.78) .. (599.22,201.45) ;
\draw    (419.5,35.28) .. controls (342.37,80.32) and (337.81,168.84) .. (419.5,201.45) ;
\draw    (419.5,201.45) .. controls (448,273.67) and (570.72,273.67) .. (599.22,201.45) ;
\draw   (383.78,120.7) -- (386.44,115.65) -- (389.1,120.7) ;
\draw   (416.84,120.7) -- (419.5,115.65) -- (422.16,120.7) ;
\draw   (596.56,119.92) -- (599.22,114.87) -- (601.88,119.92) ;
\draw   (357.57,119.53) -- (360.23,114.49) -- (362.88,119.53) ;
\draw  [fill={rgb, 255:red, 0; green, 0; blue, 0 }  ,fill opacity=1 ] (596.75,201.45) .. controls (596.75,200.06) and (597.86,198.93) .. (599.22,198.93) .. controls (600.59,198.93) and (601.69,200.06) .. (601.69,201.45) .. controls (601.69,202.85) and (600.59,203.98) .. (599.22,203.98) .. controls (597.86,203.98) and (596.75,202.85) .. (596.75,201.45) -- cycle ;
\draw  [fill={rgb, 255:red, 0; green, 0; blue, 0 }  ,fill opacity=1 ] (596.75,35.28) .. controls (596.75,33.89) and (597.86,32.76) .. (599.22,32.76) .. controls (600.59,32.76) and (601.69,33.89) .. (601.69,35.28) .. controls (601.69,36.68) and (600.59,37.81) .. (599.22,37.81) .. controls (597.86,37.81) and (596.75,36.68) .. (596.75,35.28) -- cycle ;
\draw  [fill={rgb, 255:red, 0; green, 0; blue, 0 }  ,fill opacity=1 ] (417.03,35.28) .. controls (417.03,33.89) and (418.14,32.76) .. (419.5,32.76) .. controls (420.86,32.76) and (421.97,33.89) .. (421.97,35.28) .. controls (421.97,36.68) and (420.86,37.81) .. (419.5,37.81) .. controls (418.14,37.81) and (417.03,36.68) .. (417.03,35.28) -- cycle ;
\draw  [fill={rgb, 255:red, 0; green, 0; blue, 0 }  ,fill opacity=1 ] (417.03,201.45) .. controls (417.03,200.06) and (418.14,198.93) .. (419.5,198.93) .. controls (420.86,198.93) and (421.97,200.06) .. (421.97,201.45) .. controls (421.97,202.85) and (420.86,203.98) .. (419.5,203.98) .. controls (418.14,203.98) and (417.03,202.85) .. (417.03,201.45) -- cycle ;
\draw  [fill={rgb, 255:red, 200; green, 200; blue, 200 }  ,fill opacity=1 ] (420.79,640.8) -- (602.39,640.8) -- (602.39,802.1) -- (420.79,802.1) -- cycle ;
\draw [color={rgb, 255:red, 238; green, 18; blue, 18 }  ,draw opacity=1 ]   (602.39,640.8) -- (420.79,802.1) ;
\draw [color={rgb, 255:red, 255; green, 0; blue, 0 }  ,draw opacity=1 ]   (420.79,640.8) -- (602.39,802.1) ;
\draw    (420.79,640.8) .. controls (376.25,670.95) and (376.64,778.73) .. (420.79,802.1) ;
\draw    (420.79,802.1) .. controls (469.17,840.91) and (562.85,836.39) .. (602.39,802.1) ;
\draw    (420.79,640.8) .. controls (342.85,684.52) and (338.24,770.44) .. (420.79,802.1) ;
\draw    (420.79,802.1) .. controls (449.58,872.19) and (573.6,872.19) .. (602.39,802.1) ;
\draw   (384.7,723.71) -- (387.39,718.81) -- (390.07,723.71) ;
\draw   (418.1,723.71) -- (420.79,718.81) -- (423.48,723.71) ;
\draw   (599.7,722.95) -- (602.39,718.06) -- (605.08,722.95) ;
\draw   (358.21,722.58) -- (360.89,717.68) -- (363.58,722.58) ;
\draw  [fill={rgb, 255:red, 0; green, 0; blue, 0 }  ,fill opacity=1 ] (599.9,802.1) .. controls (599.9,800.74) and (601.01,799.65) .. (602.39,799.65) .. controls (603.77,799.65) and (604.89,800.74) .. (604.89,802.1) .. controls (604.89,803.45) and (603.77,804.54) .. (602.39,804.54) .. controls (601.01,804.54) and (599.9,803.45) .. (599.9,802.1) -- cycle ;
\draw  [fill={rgb, 255:red, 0; green, 0; blue, 0 }  ,fill opacity=1 ] (599.9,640.8) .. controls (599.9,639.45) and (601.01,638.35) .. (602.39,638.35) .. controls (603.77,638.35) and (604.89,639.45) .. (604.89,640.8) .. controls (604.89,642.15) and (603.77,643.25) .. (602.39,643.25) .. controls (601.01,643.25) and (599.9,642.15) .. (599.9,640.8) -- cycle ;
\draw  [fill={rgb, 255:red, 0; green, 0; blue, 0 }  ,fill opacity=1 ] (418.29,640.8) .. controls (418.29,639.45) and (419.41,638.35) .. (420.79,638.35) .. controls (422.17,638.35) and (423.28,639.45) .. (423.28,640.8) .. controls (423.28,642.15) and (422.17,643.25) .. (420.79,643.25) .. controls (419.41,643.25) and (418.29,642.15) .. (418.29,640.8) -- cycle ;
\draw  [fill={rgb, 255:red, 0; green, 0; blue, 0 }  ,fill opacity=1 ] (418.29,802.1) .. controls (418.29,800.74) and (419.41,799.65) .. (420.79,799.65) .. controls (422.17,799.65) and (423.28,800.74) .. (423.28,802.1) .. controls (423.28,803.45) and (422.17,804.54) .. (420.79,804.54) .. controls (419.41,804.54) and (418.29,803.45) .. (418.29,802.1) -- cycle ;
\draw  [fill={rgb, 255:red, 200; green, 200; blue, 200 }  ,fill opacity=1 ] (418.39,379.6) -- (599.99,379.6) -- (599.99,540.9) -- (418.39,540.9) -- cycle ;
\draw [color={rgb, 255:red, 238; green, 18; blue, 18 }  ,draw opacity=1 ]   (599.99,379.6) -- (418.39,540.9) ;
\draw [color={rgb, 255:red, 255; green, 0; blue, 0 }  ,draw opacity=1 ]   (418.39,379.6) -- (599.99,540.9) ;
\draw    (418.39,379.6) .. controls (373.85,409.75) and (374.24,517.53) .. (418.39,540.9) ;
\draw    (418.39,379.6) .. controls (340.45,423.32) and (335.84,509.24) .. (418.39,540.9) ;
\draw   (382.3,462.51) -- (384.99,457.61) -- (387.67,462.51) ;
\draw   (415.7,462.51) -- (418.39,457.61) -- (421.08,462.51) ;
\draw   (597.3,461.75) -- (599.99,456.86) -- (602.68,461.75) ;
\draw   (355.81,461.38) -- (358.49,456.48) -- (361.18,461.38) ;
\draw  [fill={rgb, 255:red, 0; green, 0; blue, 0 }  ,fill opacity=1 ] (597.5,540.9) .. controls (597.5,539.54) and (598.61,538.45) .. (599.99,538.45) .. controls (601.37,538.45) and (602.49,539.54) .. (602.49,540.9) .. controls (602.49,542.25) and (601.37,543.34) .. (599.99,543.34) .. controls (598.61,543.34) and (597.5,542.25) .. (597.5,540.9) -- cycle ;
\draw  [fill={rgb, 255:red, 0; green, 0; blue, 0 }  ,fill opacity=1 ] (597.5,379.6) .. controls (597.5,378.25) and (598.61,377.15) .. (599.99,377.15) .. controls (601.37,377.15) and (602.49,378.25) .. (602.49,379.6) .. controls (602.49,380.95) and (601.37,382.05) .. (599.99,382.05) .. controls (598.61,382.05) and (597.5,380.95) .. (597.5,379.6) -- cycle ;
\draw  [fill={rgb, 255:red, 0; green, 0; blue, 0 }  ,fill opacity=1 ] (415.89,379.6) .. controls (415.89,378.25) and (417.01,377.15) .. (418.39,377.15) .. controls (419.77,377.15) and (420.88,378.25) .. (420.88,379.6) .. controls (420.88,380.95) and (419.77,382.05) .. (418.39,382.05) .. controls (417.01,382.05) and (415.89,380.95) .. (415.89,379.6) -- cycle ;
\draw  [fill={rgb, 255:red, 0; green, 0; blue, 0 }  ,fill opacity=1 ] (415.89,540.9) .. controls (415.89,539.54) and (417.01,538.45) .. (418.39,538.45) .. controls (419.77,538.45) and (420.88,539.54) .. (420.88,540.9) .. controls (420.88,542.25) and (419.77,543.34) .. (418.39,543.34) .. controls (417.01,543.34) and (415.89,542.25) .. (415.89,540.9) -- cycle ;
\draw    (19.44,290.27) -- (646.64,290.27) ;
\draw    (21.84,599.87) -- (340.24,600.99) ;
\draw    (341.04,319.39) -- (648.24,320.19) ;
\draw    (341.04,319.39) -- (340.24,600.99) ;
\draw   (182.5,32.85) -- (186.6,36.42) -- (182.5,40)(186.5,32.85) -- (190.6,36.42) -- (186.5,40) ;
\draw   (185,199.35) -- (189.1,202.92) -- (185,206.5)(189,199.35) -- (193.1,202.92) -- (189,206.5) ;
\draw   (186,227.35) -- (190.1,230.92) -- (186,234.5)(190,227.35) -- (194.1,230.92) -- (190,234.5) ;
\draw   (184.5,253.35) -- (188.6,256.92) -- (184.5,260.5)(188.5,253.35) -- (192.6,256.92) -- (188.5,260.5) ;
\draw   (161.28,153.71) -- (158.18,142.69) -- (169.13,146.02) -- (161.69,146.28) -- cycle ;
\draw   (192.78,129.71) -- (189.68,118.69) -- (200.63,122.02) -- (193.19,122.28) -- cycle ;
\draw   (216.78,106.21) -- (213.68,95.19) -- (224.63,98.52) -- (217.19,98.78) -- cycle ;
\draw   (501,31.35) -- (505.1,34.92) -- (501,38.5)(505,31.35) -- (509.1,34.92) -- (505,38.5) ;
\draw   (503.5,197.85) -- (507.6,201.42) -- (503.5,205)(507.5,197.85) -- (511.6,201.42) -- (507.5,205) ;
\draw   (504.5,225.85) -- (508.6,229.42) -- (504.5,233)(508.5,225.85) -- (512.6,229.42) -- (508.5,233) ;
\draw   (503,251.85) -- (507.1,255.42) -- (503,259)(507,251.85) -- (511.1,255.42) -- (507,259) ;
\draw   (183.5,336.45) -- (187.6,340.03) -- (183.5,343.6)(187.5,336.45) -- (191.6,340.03) -- (187.5,343.6) ;
\draw   (185,500.95) -- (189.1,504.53) -- (185,508.1)(189,500.95) -- (193.1,504.53) -- (189,508.1) ;
\draw   (186,528.95) -- (190.1,532.53) -- (186,536.1)(190,528.95) -- (194.1,532.53) -- (190,536.1) ;
\draw   (185.5,553.95) -- (189.6,557.53) -- (185.5,561.1)(189.5,553.95) -- (193.6,557.53) -- (189.5,561.1) ;
\draw   (505.5,375.95) -- (509.6,379.53) -- (505.5,383.1)(509.5,375.95) -- (513.6,379.53) -- (509.5,383.1) ;
\draw   (510.5,536.95) -- (514.6,540.53) -- (510.5,544.1)(514.5,536.95) -- (518.6,540.53) -- (514.5,544.1) ;
\draw   (183.5,632.95) -- (187.6,636.53) -- (183.5,640.1)(187.5,632.95) -- (191.6,636.53) -- (187.5,640.1) ;
\draw   (185,797.45) -- (189.1,801.03) -- (185,804.6)(189,797.45) -- (193.1,801.03) -- (189,804.6) ;
\draw   (186,825.45) -- (190.1,829.03) -- (186,832.6)(190,825.45) -- (194.1,829.03) -- (190,832.6) ;
\draw   (185.5,850.45) -- (189.6,854.03) -- (185.5,857.6)(189.5,850.45) -- (193.6,854.03) -- (189.5,857.6) ;
\draw   (507.5,636.95) -- (511.6,640.53) -- (507.5,644.1)(511.5,636.95) -- (515.6,640.53) -- (511.5,644.1) ;
\draw   (509,798.45) -- (513.1,802.03) -- (509,805.6)(513,798.45) -- (517.1,802.03) -- (513,805.6) ;
\draw   (510,826.45) -- (514.1,830.03) -- (510,833.6)(514,826.45) -- (518.1,830.03) -- (514,833.6) ;
\draw   (509.5,851.45) -- (513.6,855.03) -- (509.5,858.6)(513.5,851.45) -- (517.6,855.03) -- (513.5,858.6) ;
\draw   (158.28,452.31) -- (155.18,441.29) -- (166.13,444.62) -- (158.69,444.88) -- cycle ;
\draw   (188.78,428.31) -- (185.68,417.29) -- (196.63,420.62) -- (189.19,420.88) -- cycle ;
\draw   (212.78,404.81) -- (209.68,393.79) -- (220.63,397.12) -- (213.19,397.38) -- cycle ;
\draw   (152.78,744.81) -- (149.68,733.79) -- (160.63,737.12) -- (153.19,737.38) -- cycle ;
\draw   (184.28,720.81) -- (181.18,709.79) -- (192.13,713.12) -- (184.69,713.38) -- cycle ;
\draw   (208.28,698.31) -- (205.18,687.29) -- (216.13,690.62) -- (208.69,690.88) -- cycle ;

\draw (210.02,415.57) node [anchor=north west][inner sep=0.75pt]   [align=left] {...};
\draw (49.05,411.16) node [anchor=north west][inner sep=0.75pt]   [align=left] {...};
\draw (204.52,538.25) node [anchor=north west][inner sep=0.75pt]  [rotate=-90] [align=left] {...};
\draw (290.21,501.48) node    {$z$};
\draw (180.47,570.2) node  [font=\Large] [align=left] {triangle type};
\draw (209.42,716.37) node [anchor=north west][inner sep=0.75pt]   [align=left] {...};
\draw (48.45,707.96) node [anchor=north west][inner sep=0.75pt]   [align=left] {...};
\draw (203.92,835.05) node [anchor=north west][inner sep=0.75pt]  [rotate=-90] [align=left] {...};
\draw (288.61,797.28) node    {$z$};
\draw (190.07,873.98) node  [font=\Large] [align=left] {bitriangle type};
\draw (199.4,110.18) node [anchor=north west][inner sep=0.75pt]  [rotate=-318.15] [align=left] {...};
\draw (47.22,108.92) node [anchor=north west][inner sep=0.75pt]   [align=left] {...};
\draw (200.57,238.6) node [anchor=north west][inner sep=0.75pt]  [rotate=-90] [align=left] {...};
\draw (289.51,199.17) node    {$z$};
\draw (185.28,267.91) node  [font=\Large] [align=left] {3-degenerate type};
\draw (365.62,107.32) node [anchor=north west][inner sep=0.75pt]   [align=left] {...};
\draw (518.97,237) node [anchor=north west][inner sep=0.75pt]  [rotate=-90] [align=left] {...};
\draw (607.91,197.57) node    {$z$};
\draw (510.88,267.11) node  [font=\Large] [align=left] {4-degenerate type};
\draw (366.41,710.55) node [anchor=north west][inner sep=0.75pt]   [align=left] {...};
\draw (524.69,835.29) node [anchor=north west][inner sep=0.75pt]  [rotate=-90] [align=left] {...};
\draw (607.22,792.33) node [anchor=north west][inner sep=0.75pt]    {$z$};
\draw (518.88,869.98) node  [font=\Large] [align=left] {quadrilateral type};
\draw (364.01,449.35) node [anchor=north west][inner sep=0.75pt]   [align=left] {...};
\draw (603.82,530.13) node [anchor=north west][inner sep=0.75pt]    {$z$};
\draw (507.67,570.31) node  [font=\Large] [align=left] {cylinder type};
\draw (25,23) node [anchor=north west][inner sep=0.75pt]    {$t=0$};
\draw (24,307) node [anchor=north west][inner sep=0.75pt]    {$t=1$};
\draw (25,619) node [anchor=north west][inner sep=0.75pt]    {$t=2$};

\end{tikzpicture}

%% file: diagrams/cylinderauto.tex
\tikzset{every picture/.style={line width=0.75pt}} 

\begin{tikzpicture}[x=0.75pt,y=0.75pt,yscale=-1,xscale=1]

\draw  [fill={rgb, 255:red, 200; green, 200; blue, 200 }  ,fill opacity=1 ] (409.5,182.95) -- (409.5,20.6) -- (590.7,181.03) -- (590.7,343.38) -- cycle ;
\draw  [fill={rgb, 255:red, 200; green, 200; blue, 200 }  ,fill opacity=1 ] (88.39,100.6) -- (269.99,100.6) -- (269.99,261.9) -- (88.39,261.9) -- cycle ;
\draw    (88.39,100.6) .. controls (43.85,130.75) and (44.24,238.53) .. (88.39,261.9) ;
\draw    (88.39,100.6) .. controls (10.45,144.32) and (5.84,230.24) .. (88.39,261.9) ;
\draw   (52.3,183.51) -- (54.99,178.61) -- (57.67,183.51) ;
\draw   (85.7,183.51) -- (88.39,178.61) -- (91.08,183.51) ;
\draw   (267.3,182.75) -- (269.99,177.86) -- (272.68,182.75) ;
\draw   (25.81,182.38) -- (28.49,177.48) -- (31.18,182.38) ;
\draw  [fill={rgb, 255:red, 0; green, 0; blue, 0 }  ,fill opacity=1 ] (267.5,261.9) .. controls (267.5,260.54) and (268.61,259.45) .. (269.99,259.45) .. controls (271.37,259.45) and (272.49,260.54) .. (272.49,261.9) .. controls (272.49,263.25) and (271.37,264.34) .. (269.99,264.34) .. controls (268.61,264.34) and (267.5,263.25) .. (267.5,261.9) -- cycle ;
\draw  [fill={rgb, 255:red, 0; green, 0; blue, 0 }  ,fill opacity=1 ] (267.5,100.6) .. controls (267.5,99.25) and (268.61,98.15) .. (269.99,98.15) .. controls (271.37,98.15) and (272.49,99.25) .. (272.49,100.6) .. controls (272.49,101.95) and (271.37,103.05) .. (269.99,103.05) .. controls (268.61,103.05) and (267.5,101.95) .. (267.5,100.6) -- cycle ;
\draw  [fill={rgb, 255:red, 0; green, 0; blue, 0 }  ,fill opacity=1 ] (85.89,100.6) .. controls (85.89,99.25) and (87.01,98.15) .. (88.39,98.15) .. controls (89.77,98.15) and (90.88,99.25) .. (90.88,100.6) .. controls (90.88,101.95) and (89.77,103.05) .. (88.39,103.05) .. controls (87.01,103.05) and (85.89,101.95) .. (85.89,100.6) -- cycle ;
\draw  [fill={rgb, 255:red, 0; green, 0; blue, 0 }  ,fill opacity=1 ] (85.89,261.9) .. controls (85.89,260.54) and (87.01,259.45) .. (88.39,259.45) .. controls (89.77,259.45) and (90.88,260.54) .. (90.88,261.9) .. controls (90.88,263.25) and (89.77,264.34) .. (88.39,264.34) .. controls (87.01,264.34) and (85.89,263.25) .. (85.89,261.9) -- cycle ;
\draw    (409.39,20.6) .. controls (364.85,50.75) and (365.24,158.53) .. (409.39,181.9) ;
\draw    (409.39,20.6) .. controls (331.45,64.32) and (326.84,150.24) .. (409.39,181.9) ;
\draw   (373.3,103.51) -- (375.99,98.61) -- (378.67,103.51) ;
\draw   (406.7,103.51) -- (409.39,98.61) -- (412.08,103.51) ;
\draw   (346.81,102.38) -- (349.49,97.48) -- (352.18,102.38) ;
\draw  [fill={rgb, 255:red, 0; green, 0; blue, 0 }  ,fill opacity=1 ] (588.5,181.9) .. controls (588.5,180.54) and (589.61,179.45) .. (590.99,179.45) .. controls (592.37,179.45) and (593.49,180.54) .. (593.49,181.9) .. controls (593.49,183.25) and (592.37,184.34) .. (590.99,184.34) .. controls (589.61,184.34) and (588.5,183.25) .. (588.5,181.9) -- cycle ;
\draw  [fill={rgb, 255:red, 0; green, 0; blue, 0 }  ,fill opacity=1 ] (406.89,20.6) .. controls (406.89,19.25) and (408.01,18.15) .. (409.39,18.15) .. controls (410.77,18.15) and (411.88,19.25) .. (411.88,20.6) .. controls (411.88,21.95) and (410.77,23.05) .. (409.39,23.05) .. controls (408.01,23.05) and (406.89,21.95) .. (406.89,20.6) -- cycle ;
\draw  [fill={rgb, 255:red, 0; green, 0; blue, 0 }  ,fill opacity=1 ] (406.89,181.9) .. controls (406.89,180.54) and (408.01,179.45) .. (409.39,179.45) .. controls (410.77,179.45) and (411.88,180.54) .. (411.88,181.9) .. controls (411.88,183.25) and (410.77,184.34) .. (409.39,184.34) .. controls (408.01,184.34) and (406.89,183.25) .. (406.89,181.9) -- cycle ;
\draw   (587.95,261.6) -- (590.55,256.59) -- (593.23,261.56) ;
\draw   (187.76,102.75) -- (183.5,100.14) -- (187.42,97.04)(183.77,102.99) -- (179.51,100.38) -- (183.43,97.28) ;
\draw   (188.72,264.81) -- (184.5,262.14) -- (188.47,259.1)(184.73,264.99) -- (180.5,262.32) -- (184.47,259.28) ;
\draw   (509.56,112.36) -- (508.01,107.61) -- (513.01,107.8)(506.36,109.95) -- (504.82,105.2) -- (509.82,105.38) ;
\draw   (496.96,264.51) -- (496.01,259.6) -- (500.95,260.41)(494.09,261.73) -- (493.15,256.82) -- (498.08,257.62) ;
\draw  [fill={rgb, 255:red, 0; green, 0; blue, 0 }  ,fill opacity=1 ] (588.2,343.38) .. controls (588.2,342.02) and (589.32,340.93) .. (590.7,340.93) .. controls (592.08,340.93) and (593.2,342.02) .. (593.2,343.38) .. controls (593.2,344.73) and (592.08,345.82) .. (590.7,345.82) .. controls (589.32,345.82) and (588.2,344.73) .. (588.2,343.38) -- cycle ;

\draw (34.01,170.35) node [anchor=north west][inner sep=0.75pt]   [align=left] {...};
\draw (273.82,251.13) node [anchor=north west][inner sep=0.75pt]    {$z$};
\draw (355.01,90.35) node [anchor=north west][inner sep=0.75pt]   [align=left] {...};
\draw (595.82,171.13) node [anchor=north west][inner sep=0.75pt]    {$z$};
\draw (100.48,182) node    {$\alpha $};
\draw (180.48,273) node    {$\beta $};
\draw (421.98,110) node    {$\alpha $};
\draw (485.98,267.5) node    {$\gamma $};
\draw (301,167.98) node [anchor=north west][inner sep=0.75pt]  [font=\LARGE]  {$\cong $};

\end{tikzpicture}

%% file: diagrams/iso_3.tex
\tikzset{every picture/.style={line width=0.75pt}} 

\begin{tikzpicture}[x=0.75pt,y=0.75pt,yscale=-1,xscale=1]

\draw  [fill={rgb, 255:red, 200; green, 200; blue, 200 }  ,fill opacity=1 ] (278.27,205.04) -- (100.16,40.71) -- (278.27,40.71) -- cycle ;
\draw    (100.16,40.71) -- (278.27,205.04) ;
\draw   (97.53,125.18) -- (100.16,120.19) -- (102.8,125.18) ;
\draw   (275.64,124.41) -- (278.27,119.42) -- (280.91,124.41) ;
\draw   (184.01,114.46) -- (185.8,119.84) -- (180.42,118.39) ;
\draw   (189.59,202.54) -- (194.49,205.23) -- (189.59,207.92) ;
\draw   (184.32,38.21) -- (189.22,40.9) -- (184.32,43.59) ;
\draw  [fill={rgb, 255:red, 0; green, 0; blue, 0 }  ,fill opacity=1 ] (275.83,205.04) .. controls (275.83,203.66) and (276.92,202.54) .. (278.27,202.54) .. controls (279.63,202.54) and (280.72,203.66) .. (280.72,205.04) .. controls (280.72,206.42) and (279.63,207.53) .. (278.27,207.53) .. controls (276.92,207.53) and (275.83,206.42) .. (275.83,205.04) -- cycle ;
\draw  [fill={rgb, 255:red, 0; green, 0; blue, 0 }  ,fill opacity=1 ] (275.83,40.71) .. controls (275.83,39.33) and (276.92,38.21) .. (278.27,38.21) .. controls (279.63,38.21) and (280.72,39.33) .. (280.72,40.71) .. controls (280.72,42.09) and (279.63,43.2) .. (278.27,43.2) .. controls (276.92,43.2) and (275.83,42.09) .. (275.83,40.71) -- cycle ;
\draw  [fill={rgb, 255:red, 0; green, 0; blue, 0 }  ,fill opacity=1 ] (97.71,40.71) .. controls (97.71,39.33) and (98.81,38.21) .. (100.16,38.21) .. controls (101.51,38.21) and (102.61,39.33) .. (102.61,40.71) .. controls (102.61,42.09) and (101.51,43.2) .. (100.16,43.2) .. controls (98.81,43.2) and (97.71,42.09) .. (97.71,40.71) -- cycle ;
\draw  [fill={rgb, 255:red, 0; green, 0; blue, 0 }  ,fill opacity=1 ] (97.71,205.04) .. controls (97.71,203.66) and (98.81,202.54) .. (100.16,202.54) .. controls (101.51,202.54) and (102.61,203.66) .. (102.61,205.04) .. controls (102.61,206.42) and (101.51,207.53) .. (100.16,207.53) .. controls (98.81,207.53) and (97.71,206.42) .. (97.71,205.04) -- cycle ;
\draw [fill={rgb, 255:red, 0; green, 0; blue, 0 }  ,fill opacity=1 ]   (278.27,40.71) -- (278.27,205.04) ;
\draw    (100.16,40.71) -- (100.16,205.04) ;
\draw    (100.16,205.04) -- (278.27,205.04) ;
\draw    (119.97,205.89) .. controls (123.97,193.89) and (112.47,183.89) .. (100.47,186.89) ;
\draw    (99.97,65.89) .. controls (112.97,67.39) and (113.67,64.91) .. (116.17,55.41) ;
\draw    (250.17,204.91) .. controls (249.67,196.41) and (252.67,191.41) .. (260.17,188.91) ;
\draw   (536.94,206.24) -- (358.83,41.91) -- (536.94,41.91) -- cycle ;
\draw    (358.83,41.91) -- (536.94,206.24) ;
\draw   (356.19,126.38) -- (358.83,121.39) -- (361.46,126.38) ;
\draw   (534.31,125.61) -- (536.94,120.62) -- (539.58,125.61) ;
\draw   (442.68,115.66) -- (444.46,121.04) -- (439.08,119.59) ;
\draw   (448.26,203.74) -- (453.16,206.43) -- (448.26,209.12) ;
\draw   (442.99,39.41) -- (447.88,42.1) -- (442.99,44.79) ;
\draw  [fill={rgb, 255:red, 0; green, 0; blue, 0 }  ,fill opacity=1 ] (534.49,206.24) .. controls (534.49,204.86) and (535.59,203.74) .. (536.94,203.74) .. controls (538.29,203.74) and (539.39,204.86) .. (539.39,206.24) .. controls (539.39,207.62) and (538.29,208.74) .. (536.94,208.74) .. controls (535.59,208.74) and (534.49,207.62) .. (534.49,206.24) -- cycle ;
\draw  [fill={rgb, 255:red, 0; green, 0; blue, 0 }  ,fill opacity=1 ] (534.49,41.91) .. controls (534.49,40.53) and (535.59,39.41) .. (536.94,39.41) .. controls (538.29,39.41) and (539.39,40.53) .. (539.39,41.91) .. controls (539.39,43.29) and (538.29,44.4) .. (536.94,44.4) .. controls (535.59,44.4) and (534.49,43.29) .. (534.49,41.91) -- cycle ;
\draw  [fill={rgb, 255:red, 0; green, 0; blue, 0 }  ,fill opacity=1 ] (356.38,41.91) .. controls (356.38,40.53) and (357.48,39.41) .. (358.83,39.41) .. controls (360.18,39.41) and (361.28,40.53) .. (361.28,41.91) .. controls (361.28,43.29) and (360.18,44.4) .. (358.83,44.4) .. controls (357.48,44.4) and (356.38,43.29) .. (356.38,41.91) -- cycle ;
\draw  [fill={rgb, 255:red, 0; green, 0; blue, 0 }  ,fill opacity=1 ] (356.38,206.24) .. controls (356.38,204.86) and (357.48,203.74) .. (358.83,203.74) .. controls (360.18,203.74) and (361.28,204.86) .. (361.28,206.24) .. controls (361.28,207.62) and (360.18,208.74) .. (358.83,208.74) .. controls (357.48,208.74) and (356.38,207.62) .. (356.38,206.24) -- cycle ;
\draw [fill={rgb, 255:red, 0; green, 0; blue, 0 }  ,fill opacity=1 ]   (536.94,41.91) -- (536.94,206.24) ;
\draw    (358.83,41.91) -- (358.83,206.24) ;
\draw    (358.83,206.24) -- (536.94,206.24) ;
\draw    (378.63,207.09) .. controls (382.63,195.09) and (371.13,185.09) .. (359.13,188.09) ;
\draw    (358.63,67.09) .. controls (369.33,64.55) and (369.33,63.88) .. (374.83,56.61) ;
\draw    (508.83,206.11) .. controls (508.33,197.61) and (511.33,192.61) .. (518.83,190.11) ;
\draw    (374.83,56.61) .. controls (383.33,52.55) and (383.33,48.55) .. (385.33,41.21) ;
\draw    (518.67,42.55) .. controls (522,51.88) and (524.67,57.21) .. (537.33,57.21) ;
\draw    (518.27,189.08) .. controls (522.67,183.88) and (528,181.21) .. (538,183.21) ;
\draw   (144.11,154.72) .. controls (144.11,153.65) and (144.96,152.78) .. (146.01,152.78) .. controls (147.06,152.78) and (147.91,153.65) .. (147.91,154.72) .. controls (147.91,155.8) and (147.06,156.67) .. (146.01,156.67) .. controls (144.96,156.67) and (144.11,155.8) .. (144.11,154.72) -- cycle ;
\draw   (406.77,161.39) .. controls (406.77,160.32) and (407.62,159.45) .. (408.67,159.45) .. controls (409.72,159.45) and (410.57,160.32) .. (410.57,161.39) .. controls (410.57,162.46) and (409.72,163.33) .. (408.67,163.33) .. controls (407.62,163.33) and (406.77,162.46) .. (406.77,161.39) -- cycle ;
\draw   (483.44,99.39) .. controls (483.44,98.32) and (484.29,97.45) .. (485.34,97.45) .. controls (486.39,97.45) and (487.24,98.32) .. (487.24,99.39) .. controls (487.24,100.46) and (486.39,101.33) .. (485.34,101.33) .. controls (484.29,101.33) and (483.44,100.46) .. (483.44,99.39) -- cycle ;

\draw (287.88,202.2) node    {$z$};
\draw (129.03,179.51) node    {$\theta $};
\draw (119.03,78.51) node    {$\theta $};
\draw (237.53,191.51) node    {$\theta $};
\draw (189.15,219.41) node    {$\alpha $};
\draw (87.81,124.75) node    {$\beta $};
\draw (173.15,130.75) node    {$\gamma $};
\draw (546.55,203.4) node    {$z$};
\draw (387.7,180.71) node    {$\theta $};
\draw (375.7,79.05) node    {$\theta $};
\draw (496.2,192.71) node    {$\theta $};
\draw (447.81,220.61) node    {$\alpha $};
\draw (346.48,125.95) node    {$\beta $};
\draw (433.81,130.95) node    {$\gamma $};
\draw (399.7,55.05) node    {$\phi $};
\draw (514.36,61.71) node    {$\phi $};
\draw (522.36,168.38) node    {$\phi $};
\draw (148.99,146.9) node [anchor=north west][inner sep=0.75pt]  [font=\scriptsize]  {$q_{1}$};
\draw (411.65,153.56) node [anchor=north west][inner sep=0.75pt]  [font=\scriptsize]  {$q_{1}$};
\draw (488.32,91.56) node [anchor=north west][inner sep=0.75pt]  [font=\scriptsize]  {$q_{2}$};

\end{tikzpicture}

%% file: diagrams/iso_2.tex
\tikzset{every picture/.style={line width=0.75pt}} 

\begin{tikzpicture}[x=0.75pt,y=0.75pt,yscale=-1,xscale=1]

\draw   (77.23,41.56) -- (256.96,41.56) -- (256.96,207.73) -- (77.23,207.73) -- cycle ;
\draw   (74.57,126.97) -- (77.23,121.93) -- (79.89,126.97) ;
\draw   (254.3,126.2) -- (256.96,121.15) -- (259.62,126.2) ;
\draw  [fill={rgb, 255:red, 0; green, 0; blue, 0 }  ,fill opacity=1 ] (254.49,207.73) .. controls (254.49,206.34) and (255.59,205.21) .. (256.96,205.21) .. controls (258.32,205.21) and (259.43,206.34) .. (259.43,207.73) .. controls (259.43,209.12) and (258.32,210.25) .. (256.96,210.25) .. controls (255.59,210.25) and (254.49,209.12) .. (254.49,207.73) -- cycle ;
\draw  [fill={rgb, 255:red, 0; green, 0; blue, 0 }  ,fill opacity=1 ] (254.49,41.56) .. controls (254.49,40.16) and (255.59,39.03) .. (256.96,39.03) .. controls (258.32,39.03) and (259.43,40.16) .. (259.43,41.56) .. controls (259.43,42.95) and (258.32,44.08) .. (256.96,44.08) .. controls (255.59,44.08) and (254.49,42.95) .. (254.49,41.56) -- cycle ;
\draw  [fill={rgb, 255:red, 0; green, 0; blue, 0 }  ,fill opacity=1 ] (74.76,41.56) .. controls (74.76,40.16) and (75.87,39.03) .. (77.23,39.03) .. controls (78.6,39.03) and (79.7,40.16) .. (79.7,41.56) .. controls (79.7,42.95) and (78.6,44.08) .. (77.23,44.08) .. controls (75.87,44.08) and (74.76,42.95) .. (74.76,41.56) -- cycle ;
\draw  [fill={rgb, 255:red, 0; green, 0; blue, 0 }  ,fill opacity=1 ] (74.76,207.73) .. controls (74.76,206.34) and (75.87,205.21) .. (77.23,205.21) .. controls (78.6,205.21) and (79.7,206.34) .. (79.7,207.73) .. controls (79.7,209.12) and (78.6,210.25) .. (77.23,210.25) .. controls (75.87,210.25) and (74.76,209.12) .. (74.76,207.73) -- cycle ;
\draw   (164.33,38.25) -- (168.43,41.11) -- (164.33,43.97)(168.33,38.25) -- (172.43,41.11) -- (168.33,43.97) ;
\draw   (166.33,205.25) -- (170.43,208.11) -- (166.33,210.97)(170.33,205.25) -- (174.43,208.11) -- (170.33,210.97) ;
\draw    (96.63,207.09) .. controls (100.63,195.09) and (89.13,185.09) .. (77.13,188.09) ;
\draw    (77.63,62.09) .. controls (90.63,63.59) and (99.63,53.09) .. (95.63,41.59) ;
\draw    (235.63,41.59) .. controls (231.63,56.59) and (242.13,66.59) .. (256.63,61.59) ;
\draw    (257.13,182.09) .. controls (241.63,181.09) and (229.63,192.09) .. (233.63,207.59) ;

\draw (265.64,203.85) node    {$z$};
\draw (105.7,68.21) node    {$\theta $};
\draw (222.7,177.21) node    {$\theta $};
\draw (225.2,68.21) node    {$\theta $};
\draw (106.7,179.71) node    {$\theta $};
\draw (177.15,225.15) node    {$\alpha $};
\draw (63.81,123.81) node    {$\beta $};

\end{tikzpicture}

%% file: diagrams/dualgraph.tex
\tikzset{every picture/.style={line width=0.75pt}} 

\begin{tikzpicture}[x=0.75pt,y=0.75pt,yscale=-1,xscale=1]

\draw    (65.22,149.1) .. controls (114.22,43.1) and (233.22,43.1) .. (273.22,150.1) ;
\draw    (64.22,85.1) .. controls (140.22,219.1) and (209.22,219.1) .. (272.22,86.1) ;
\draw  [fill={rgb, 255:red, 255; green, 255; blue, 255 }  ,fill opacity=1 ] (169.8,185.46) .. controls (169.8,184.17) and (170.85,183.12) .. (172.14,183.12) .. controls (173.43,183.12) and (174.48,184.17) .. (174.48,185.46) .. controls (174.48,186.75) and (173.43,187.8) .. (172.14,187.8) .. controls (170.85,187.8) and (169.8,186.75) .. (169.8,185.46) -- cycle ;
\draw  [fill={rgb, 255:red, 0; green, 0; blue, 0 }  ,fill opacity=1 ] (135.4,75.66) .. controls (135.4,74.37) and (136.45,73.32) .. (137.74,73.32) .. controls (139.03,73.32) and (140.08,74.37) .. (140.08,75.66) .. controls (140.08,76.95) and (139.03,78) .. (137.74,78) .. controls (136.45,78) and (135.4,76.95) .. (135.4,75.66) -- cycle ;
\draw  [fill={rgb, 255:red, 0; green, 0; blue, 0 }  ,fill opacity=1 ] (82,117.66) .. controls (82,116.37) and (83.05,115.32) .. (84.34,115.32) .. controls (85.63,115.32) and (86.68,116.37) .. (86.68,117.66) .. controls (86.68,118.95) and (85.63,120) .. (84.34,120) .. controls (83.05,120) and (82,118.95) .. (82,117.66) -- cycle ;
\draw  [fill={rgb, 255:red, 0; green, 0; blue, 0 }  ,fill opacity=1 ] (253.4,116.66) .. controls (253.4,115.37) and (254.45,114.32) .. (255.74,114.32) .. controls (257.03,114.32) and (258.08,115.37) .. (258.08,116.66) .. controls (258.08,117.95) and (257.03,119) .. (255.74,119) .. controls (254.45,119) and (253.4,117.95) .. (253.4,116.66) -- cycle ;
\draw  [fill={rgb, 255:red, 0; green, 0; blue, 0 }  ,fill opacity=1 ] (118.2,161.46) .. controls (118.2,160.17) and (119.25,159.12) .. (120.54,159.12) .. controls (121.83,159.12) and (122.88,160.17) .. (122.88,161.46) .. controls (122.88,162.75) and (121.83,163.8) .. (120.54,163.8) .. controls (119.25,163.8) and (118.2,162.75) .. (118.2,161.46) -- cycle ;
\draw  [fill={rgb, 255:red, 0; green, 0; blue, 0 }  ,fill opacity=1 ] (218.2,164.26) .. controls (218.2,162.97) and (219.25,161.92) .. (220.54,161.92) .. controls (221.83,161.92) and (222.88,162.97) .. (222.88,164.26) .. controls (222.88,165.55) and (221.83,166.6) .. (220.54,166.6) .. controls (219.25,166.6) and (218.2,165.55) .. (218.2,164.26) -- cycle ;
\draw  [fill={rgb, 255:red, 0; green, 0; blue, 0 }  ,fill opacity=1 ] (206,77.06) .. controls (206,75.77) and (207.05,74.72) .. (208.34,74.72) .. controls (209.63,74.72) and (210.68,75.77) .. (210.68,77.06) .. controls (210.68,78.35) and (209.63,79.4) .. (208.34,79.4) .. controls (207.05,79.4) and (206,78.35) .. (206,77.06) -- cycle ;
\draw    (441.16,96.92) .. controls (426.36,122.92) and (382.72,143.88) .. (384.54,121.6) .. controls (386.36,99.32) and (573.03,167.28) .. (589.96,178.92) ;
\draw  [fill={rgb, 255:red, 0; green, 0; blue, 0 }  ,fill opacity=1 ] (416.6,120.06) .. controls (416.6,118.77) and (417.65,117.72) .. (418.94,117.72) .. controls (420.23,117.72) and (421.28,118.77) .. (421.28,120.06) .. controls (421.28,121.35) and (420.23,122.4) .. (418.94,122.4) .. controls (417.65,122.4) and (416.6,121.35) .. (416.6,120.06) -- cycle ;
\draw  [fill={rgb, 255:red, 255; green, 255; blue, 255 }  ,fill opacity=1 ] (448.2,128.26) .. controls (448.2,126.97) and (449.25,125.92) .. (450.54,125.92) .. controls (451.83,125.92) and (452.88,126.97) .. (452.88,128.26) .. controls (452.88,129.55) and (451.83,130.6) .. (450.54,130.6) .. controls (449.25,130.6) and (448.2,129.55) .. (448.2,128.26) -- cycle ;
\draw  [fill={rgb, 255:red, 0; green, 0; blue, 0 }  ,fill opacity=1 ] (488.2,141.06) .. controls (488.2,139.77) and (489.25,138.72) .. (490.54,138.72) .. controls (491.83,138.72) and (492.88,139.77) .. (492.88,141.06) .. controls (492.88,142.35) and (491.83,143.4) .. (490.54,143.4) .. controls (489.25,143.4) and (488.2,142.35) .. (488.2,141.06) -- cycle ;
\draw  [fill={rgb, 255:red, 0; green, 0; blue, 0 }  ,fill opacity=1 ] (544.2,160.06) .. controls (544.2,158.77) and (545.25,157.72) .. (546.54,157.72) .. controls (547.83,157.72) and (548.88,158.77) .. (548.88,160.06) .. controls (548.88,161.35) and (547.83,162.4) .. (546.54,162.4) .. controls (545.25,162.4) and (544.2,161.35) .. (544.2,160.06) -- cycle ;

\draw (68.36,116.24) node    {$s_{1}$};
\draw (271.36,119.24) node    {$s_{2}$};
\draw (171.36,199.84) node    {$z$};
\draw (140.16,56.84) node    {$q_{\tau ( 1)}$};
\draw (205.36,57.04) node    {$q_{\tau ( t)}$};
\draw (143.16,142.24) node    {$q_{\tau ( t+1)}$};
\draw (211.16,141.64) node    {$q_{\tau ( p)}$};
\draw (166.8,53.8) node [anchor=north west][inner sep=0.75pt]   [align=left] {...};
\draw (174,137.4) node [anchor=north west][inner sep=0.75pt]   [align=left] {...};
\draw (284.28,162.44) node    {$X_{0}$};
\draw (280.09,72.04) node    {$X_{-1}$};
\draw (447.76,138.64) node    {$z$};
\draw (497.36,125.04) node    {$q_{1}$};
\draw (555.76,142.24) node    {$q_{p}$};
\draw (524.14,130.45) node [anchor=north west][inner sep=0.75pt]  [rotate=-20.96] [align=left] {...};
\draw (155.6,228.2) node [anchor=north west][inner sep=0.75pt]    {$t >0$};
\draw (459.6,226.6) node [anchor=north west][inner sep=0.75pt]    {$t=0$};

\end{tikzpicture}

%% file: diagrams/plumb.tex
\tikzset{every picture/.style={line width=0.75pt}} 

\begin{tikzpicture}[x=0.75pt,y=0.75pt,yscale=-1,xscale=1]

\draw   (109.9,41.08) -- (289.62,41.08) -- (289.62,207.25) -- (109.9,207.25) -- cycle ;
\draw    (289.62,207.25) .. controls (271.76,121.06) and (185.13,58.94) .. (109.9,41.08) ;
\draw    (289.62,207.25) .. controls (199.57,191.34) and (125.1,131.16) .. (109.9,41.08) ;
\draw    (109.9,41.08) .. controls (65.82,72.14) and (66.2,183.18) .. (109.9,207.25) ;
\draw    (109.9,207.25) .. controls (157.77,247.24) and (250.49,242.58) .. (289.62,207.25) ;
\draw    (109.9,41.08) -- (289.62,207.25) ;
\draw    (109.9,41.08) .. controls (32.77,86.12) and (28.21,174.64) .. (109.9,207.25) ;
\draw    (109.9,207.25) .. controls (138.4,279.47) and (261.12,279.47) .. (289.62,207.25) ;
\draw   (74.18,126.5) -- (76.84,121.45) -- (79.5,126.5) ;
\draw   (107.24,126.5) -- (109.9,121.45) -- (112.56,126.5) ;
\draw   (286.96,125.72) -- (289.62,120.67) -- (292.28,125.72) ;
\draw   (47.97,125.33) -- (50.63,120.29) -- (53.28,125.33) ;
\draw   (164.33,149.03) -- (162.94,143.47) -- (168.24,145.34) ;
\draw   (193.21,122.24) -- (191.82,116.68) -- (197.12,118.55) ;
\draw   (224.36,105.54) -- (222.98,99.99) -- (228.28,101.86) ;
\draw   (176.01,136.2) .. controls (176.01,135.13) and (176.86,134.26) .. (177.91,134.26) .. controls (178.96,134.26) and (179.81,135.13) .. (179.81,136.2) .. controls (179.81,137.28) and (178.96,138.14) .. (177.91,138.14) .. controls (176.86,138.14) and (176.01,137.28) .. (176.01,136.2) -- cycle ;
\draw   (250.11,78.35) .. controls (250.11,77.28) and (250.96,76.41) .. (252.01,76.41) .. controls (253.06,76.41) and (253.91,77.28) .. (253.91,78.35) .. controls (253.91,79.43) and (253.06,80.3) .. (252.01,80.3) .. controls (250.96,80.3) and (250.11,79.43) .. (250.11,78.35) -- cycle ;
\draw   (140.68,171.92) .. controls (140.68,170.85) and (141.53,169.98) .. (142.58,169.98) .. controls (143.63,169.98) and (144.48,170.85) .. (144.48,171.92) .. controls (144.48,173) and (143.63,173.86) .. (142.58,173.86) .. controls (141.53,173.86) and (140.68,173) .. (140.68,171.92) -- cycle ;
\draw   (89.76,127.27) .. controls (89.76,126.2) and (90.61,125.33) .. (91.66,125.33) .. controls (92.71,125.33) and (93.56,126.2) .. (93.56,127.27) .. controls (93.56,128.35) and (92.71,129.22) .. (91.66,129.22) .. controls (90.61,129.22) and (89.76,128.35) .. (89.76,127.27) -- cycle ;
\draw   (193.11,220.84) .. controls (193.11,219.77) and (193.96,218.9) .. (195.01,218.9) .. controls (196.06,218.9) and (196.91,219.77) .. (196.91,220.84) .. controls (196.91,221.91) and (196.06,222.78) .. (195.01,222.78) .. controls (193.96,222.78) and (193.11,221.91) .. (193.11,220.84) -- cycle ;
\draw   (200.14,204.73) -- (205.08,207.45) -- (200.14,210.17) ;
\draw   (200.9,232.68) -- (205.84,235.4) -- (200.9,238.12) ;
\draw   (201.28,258.31) -- (206.22,261.03) -- (201.28,263.74) ;
\draw   (194.82,38.56) -- (199.76,41.28) -- (194.82,43.99) ;
\draw  [fill={rgb, 255:red, 0; green, 0; blue, 0 }  ,fill opacity=1 ] (287.15,207.25) .. controls (287.15,205.86) and (288.26,204.73) .. (289.62,204.73) .. controls (290.99,204.73) and (292.09,205.86) .. (292.09,207.25) .. controls (292.09,208.65) and (290.99,209.78) .. (289.62,209.78) .. controls (288.26,209.78) and (287.15,208.65) .. (287.15,207.25) -- cycle ;
\draw  [fill={rgb, 255:red, 0; green, 0; blue, 0 }  ,fill opacity=1 ] (287.15,41.08) .. controls (287.15,39.69) and (288.26,38.56) .. (289.62,38.56) .. controls (290.99,38.56) and (292.09,39.69) .. (292.09,41.08) .. controls (292.09,42.48) and (290.99,43.61) .. (289.62,43.61) .. controls (288.26,43.61) and (287.15,42.48) .. (287.15,41.08) -- cycle ;
\draw  [fill={rgb, 255:red, 0; green, 0; blue, 0 }  ,fill opacity=1 ] (107.43,41.08) .. controls (107.43,39.69) and (108.54,38.56) .. (109.9,38.56) .. controls (111.26,38.56) and (112.37,39.69) .. (112.37,41.08) .. controls (112.37,42.48) and (111.26,43.61) .. (109.9,43.61) .. controls (108.54,43.61) and (107.43,42.48) .. (107.43,41.08) -- cycle ;
\draw  [fill={rgb, 255:red, 0; green, 0; blue, 0 }  ,fill opacity=1 ] (107.43,207.25) .. controls (107.43,205.86) and (108.54,204.73) .. (109.9,204.73) .. controls (111.26,204.73) and (112.37,205.86) .. (112.37,207.25) .. controls (112.37,208.65) and (111.26,209.78) .. (109.9,209.78) .. controls (108.54,209.78) and (107.43,208.65) .. (107.43,207.25) -- cycle ;
\draw    (94.14,191.84) .. controls (97.14,183.34) and (103.14,181.34) .. (110.14,186.34) ;
\draw    (130.14,206.84) .. controls (134.64,212.34) and (134.14,215.34) .. (128.64,218.84) ;
\draw    (275.14,193.84) .. controls (268.14,192.84) and (267.14,197.84) .. (269.14,202.34) ;

\draw (208.2,114.38) node [anchor=north west][inner sep=0.75pt]  [rotate=-318.15] [align=left] {...};
\draw (56.02,113.12) node [anchor=north west][inner sep=0.75pt]   [align=left] {...};
\draw (151.63,149.71) node    {$\alpha '_{0}$};
\draw (207.79,96.84) node    {$\alpha '_{j-1}$};
\draw (271.62,126.6) node [anchor=north west][inner sep=0.75pt]    {$\alpha '_{j}$};
\draw (38.68,133.34) node    {$\alpha '_{j}$};
\draw (125.17,134.45) node    {$\alpha '_{t-1}$};
\draw (244.3,183.22) node  [font=\scriptsize]  {$2\pi C_{1}$};
\draw (96.63,170.06) node  [font=\tiny]  {$2\pi C_{t-1}$};
\draw (155.02,215.72) node  [font=\scriptsize]  {$2\pi C_{t+1}$};
\draw (180.77,128.37) node [anchor=north west][inner sep=0.75pt]  [font=\scriptsize]  {$p_{1}$};
\draw (254.99,70.52) node [anchor=north west][inner sep=0.75pt]  [font=\scriptsize]  {$p_{j}$};
\draw (145.56,164.09) node [anchor=north west][inner sep=0.75pt]  [font=\scriptsize]  {$p_{t}$};
\draw (85.2,111.29) node [anchor=north west][inner sep=0.75pt]  [font=\scriptsize]  {$p_{t-1}$};
\draw (196.91,213.01) node [anchor=north west][inner sep=0.75pt]  [font=\scriptsize]  {$p_{t+1}$};
\draw (209.37,242.8) node [anchor=north west][inner sep=0.75pt]  [rotate=-90] [align=left] {...};
\draw (189.29,185.66) node [anchor=north west][inner sep=0.75pt]    {$\beta '_{1}$};
\draw (196.59,272.84) node    {$\beta '_{n+1-t}$};
\draw (177.03,20.16) node [anchor=north west][inner sep=0.75pt]    {$\beta '_{n+1-t}$};
\draw (300.31,203.37) node    {$z$};

\end{tikzpicture}

%% file: diagrams/cells.tex
\tikzset{every picture/.style={line width=0.75pt}} 

\begin{tikzpicture}[x=0.75pt,y=0.75pt,yscale=-1,xscale=1]

\draw  [draw opacity=0][fill={rgb, 255:red, 200; green, 200; blue, 200 }  ,fill opacity=1 ] (413.12,875.7) -- (580.13,714.31) -- (580.13,875.7) -- cycle ;
\draw [fill={rgb, 255:red, 255; green, 255; blue, 255 }  ,fill opacity=1 ]   (413.07,876.02) .. controls (499.34,851.59) and (568.34,791.09) .. (579.53,715.55) ;
\draw  [draw opacity=0][fill={rgb, 255:red, 200; green, 200; blue, 200 }  ,fill opacity=1 ] (580.41,714.97) -- (413.07,876.02) -- (413.4,714.62) -- cycle ;
\draw [fill={rgb, 255:red, 255; green, 255; blue, 255 }  ,fill opacity=1 ]   (413.12,875.7) .. controls (430.39,799.28) and (488.39,732.78) .. (579.58,715.24) ;
\draw   (579.58,715.24) -- (579.3,875.87) -- (413.12,875.7) -- (413.4,715.07) -- cycle ;
\draw    (579.58,715.24) -- (413.12,875.7) ;
\draw   (499.28,717.45) -- (494.17,715.23) -- (499.14,712.7) ;
\draw   (494.66,873.41) -- (499.7,875.79) -- (494.65,878.16) ;
\draw   (498.29,789.8) -- (503.77,788.29) -- (502.13,793.12) ;
\draw   (410.35,800.14) -- (413.07,795.72) -- (415.78,800.14) ;
\draw   (576.5,795.53) -- (579.27,791.14) -- (581.94,795.59) ;
\draw  [fill={rgb, 255:red, 0; green, 0; blue, 0 }  ,fill opacity=1 ] (413.13,873.49) .. controls (414.52,873.49) and (415.65,874.48) .. (415.65,875.7) .. controls (415.65,876.92) and (414.51,877.91) .. (413.12,877.91) .. controls (411.73,877.91) and (410.6,876.92) .. (410.6,875.7) .. controls (410.6,874.48) and (411.73,873.49) .. (413.13,873.49) -- cycle ;
\draw  [fill={rgb, 255:red, 0; green, 0; blue, 0 }  ,fill opacity=1 ] (579.3,873.67) .. controls (580.69,873.67) and (581.82,874.66) .. (581.82,875.88) .. controls (581.82,877.1) and (580.69,878.09) .. (579.29,878.09) .. controls (577.9,878.08) and (576.77,877.09) .. (576.77,875.87) .. controls (576.77,874.66) and (577.91,873.67) .. (579.3,873.67) -- cycle ;
\draw  [fill={rgb, 255:red, 0; green, 0; blue, 0 }  ,fill opacity=1 ] (579.58,713.03) .. controls (580.98,713.03) and (582.11,714.02) .. (582.1,715.24) .. controls (582.1,716.46) and (580.97,717.45) .. (579.58,717.45) .. controls (578.18,717.45) and (577.05,716.46) .. (577.06,715.24) .. controls (577.06,714.02) and (578.19,713.03) .. (579.58,713.03) -- cycle ;
\draw  [fill={rgb, 255:red, 0; green, 0; blue, 0 }  ,fill opacity=1 ] (413.41,712.86) .. controls (414.81,712.86) and (415.93,713.85) .. (415.93,715.07) .. controls (415.93,716.29) and (414.8,717.27) .. (413.4,717.27) .. controls (412.01,717.27) and (410.88,716.28) .. (410.88,715.06) .. controls (410.89,713.84) and (412.02,712.86) .. (413.41,712.86) -- cycle ;
\draw    (580.63,876.61) .. controls (619.25,834.02) and (615.61,750.52) .. (580.16,715.31) ;
\draw   (605.2,798.6) -- (607.71,793.61) -- (610.57,798.41) ;
\draw    (413.4,714.62) .. controls (450.33,675) and (535.33,671) .. (579.58,714.8) ;
\draw   (498.86,685.92) -- (493.67,683.87) -- (498.56,681.17) ;
\draw  [fill={rgb, 255:red, 200; green, 200; blue, 200 }  ,fill opacity=1 ] (97.66,32.93) -- (263.86,32.93) -- (263.86,194.22) -- (97.66,194.22) -- cycle ;
\draw [color={rgb, 255:red, 238; green, 18; blue, 18 }  ,draw opacity=1 ]   (263.86,32.93) -- (97.66,194.22) ;
\draw [color={rgb, 255:red, 255; green, 0; blue, 0 }  ,draw opacity=1 ]   (97.66,32.93) -- (263.86,194.22) ;
\draw    (97.66,32.93) .. controls (56.9,63.08) and (63.99,170.24) .. (97.66,194.22) ;
\draw    (97.66,32.93) .. controls (30.08,61.89) and (25.65,161.89) .. (97.66,194.22) ;
\draw   (67.37,115.84) -- (69.83,110.94) -- (72.29,115.84) ;
\draw   (95.2,115.84) -- (97.66,110.94) -- (100.12,115.84) ;
\draw   (261.4,115.08) -- (263.86,110.18) -- (266.32,115.08) ;
\draw   (43.33,115.71) -- (45.79,110.81) -- (48.25,115.71) ;
\draw   (181.11,191.77) -- (185.68,194.41) -- (181.11,197.05) ;
\draw   (176.19,30.48) -- (180.76,33.12) -- (176.19,35.75) ;
\draw  [fill={rgb, 255:red, 0; green, 0; blue, 0 }  ,fill opacity=1 ] (261.57,194.22) .. controls (261.57,192.87) and (262.6,191.77) .. (263.86,191.77) .. controls (265.12,191.77) and (266.14,192.87) .. (266.14,194.22) .. controls (266.14,195.58) and (265.12,196.67) .. (263.86,196.67) .. controls (262.6,196.67) and (261.57,195.58) .. (261.57,194.22) -- cycle ;
\draw  [fill={rgb, 255:red, 0; green, 0; blue, 0 }  ,fill opacity=1 ] (261.57,32.93) .. controls (261.57,31.57) and (262.6,30.48) .. (263.86,30.48) .. controls (265.12,30.48) and (266.14,31.57) .. (266.14,32.93) .. controls (266.14,34.28) and (265.12,35.38) .. (263.86,35.38) .. controls (262.6,35.38) and (261.57,34.28) .. (261.57,32.93) -- cycle ;
\draw  [fill={rgb, 255:red, 0; green, 0; blue, 0 }  ,fill opacity=1 ] (95.37,32.93) .. controls (95.37,31.57) and (96.39,30.48) .. (97.66,30.48) .. controls (98.92,30.48) and (99.94,31.57) .. (99.94,32.93) .. controls (99.94,34.28) and (98.92,35.38) .. (97.66,35.38) .. controls (96.39,35.38) and (95.37,34.28) .. (95.37,32.93) -- cycle ;
\draw  [fill={rgb, 255:red, 0; green, 0; blue, 0 }  ,fill opacity=1 ] (95.37,194.22) .. controls (95.37,192.87) and (96.39,191.77) .. (97.66,191.77) .. controls (98.92,191.77) and (99.94,192.87) .. (99.94,194.22) .. controls (99.94,195.58) and (98.92,196.67) .. (97.66,196.67) .. controls (96.39,196.67) and (95.37,195.58) .. (95.37,194.22) -- cycle ;
\draw    (264.69,195.26) .. controls (300.1,152.74) and (296.9,69.24) .. (264.51,33.97) ;
\draw    (264.69,195.26) .. controls (335.97,151.45) and (336.58,74.39) .. (264.51,33.97) ;
\draw   (287.3,117.3) -- (289.61,112.31) -- (292.21,117.12) ;
\draw   (315.92,118.24) -- (318.49,113.41) -- (320.83,118.38) ;
\draw   (579.25,31.98) -- (578.97,192.61) -- (412.79,192.44) -- (413.07,31.8) -- cycle ;
\draw    (413.08,31.8) .. controls (373.01,74.55) and (377.52,157.42) .. (412.79,192.44) ;
\draw    (579.25,31.98) -- (412.79,192.44) ;
\draw   (410.02,116.87) -- (412.74,112.46) -- (415.45,116.87) ;
\draw   (382.3,117.78) -- (384.56,113.12) -- (387.7,117.24) ;
\draw   (576.17,112.26) -- (578.94,107.88) -- (581.61,112.32) ;
\draw  [fill={rgb, 255:red, 0; green, 0; blue, 0 }  ,fill opacity=1 ] (412.8,190.23) .. controls (414.19,190.23) and (415.32,191.22) .. (415.32,192.44) .. controls (415.31,193.66) and (414.18,194.64) .. (412.79,194.64) .. controls (411.4,194.64) and (410.27,193.65) .. (410.27,192.43) .. controls (410.27,191.21) and (411.4,190.23) .. (412.8,190.23) -- cycle ;
\draw  [fill={rgb, 255:red, 0; green, 0; blue, 0 }  ,fill opacity=1 ] (578.97,190.4) .. controls (580.36,190.41) and (581.49,191.4) .. (581.49,192.61) .. controls (581.49,193.83) and (580.35,194.82) .. (578.96,194.82) .. controls (577.57,194.82) and (576.44,193.83) .. (576.44,192.61) .. controls (576.44,191.39) and (577.57,190.4) .. (578.97,190.4) -- cycle ;
\draw  [fill={rgb, 255:red, 0; green, 0; blue, 0 }  ,fill opacity=1 ] (579.25,29.77) .. controls (580.65,29.77) and (581.77,30.76) .. (581.77,31.98) .. controls (581.77,33.2) and (580.64,34.18) .. (579.24,34.18) .. controls (577.85,34.18) and (576.72,33.19) .. (576.73,31.97) .. controls (576.73,30.75) and (577.86,29.77) .. (579.25,29.77) -- cycle ;
\draw  [fill={rgb, 255:red, 0; green, 0; blue, 0 }  ,fill opacity=1 ] (413.08,29.59) .. controls (414.47,29.59) and (415.6,30.58) .. (415.6,31.8) .. controls (415.6,33.02) and (414.47,34.01) .. (413.07,34.01) .. controls (411.68,34) and (410.55,33.02) .. (410.55,31.8) .. controls (410.56,30.58) and (411.69,29.59) .. (413.08,29.59) -- cycle ;
\draw    (580.3,193.34) .. controls (618.92,150.75) and (615.28,67.26) .. (579.82,32.05) ;
\draw   (604.87,115.34) -- (607.38,110.35) -- (610.24,115.15) ;
\draw  [draw opacity=0][fill={rgb, 255:red, 200; green, 200; blue, 200 }  ,fill opacity=1 ] (266.69,257.09) -- (99.35,418.13) -- (99.68,256.74) -- cycle ;
\draw   (266.25,256.98) -- (265.97,417.61) -- (99.79,417.44) -- (100.07,256.8) -- cycle ;
\draw    (100.08,256.8) .. controls (60.01,299.55) and (64.52,382.42) .. (99.79,417.44) ;
\draw    (266.25,256.98) -- (99.79,417.44) ;
\draw   (97.02,341.87) -- (99.74,337.46) -- (102.45,341.87) ;
\draw   (69.3,342.78) -- (71.56,338.12) -- (74.7,342.24) ;
\draw   (263.17,337.26) -- (265.94,332.88) -- (268.61,337.32) ;
\draw  [fill={rgb, 255:red, 0; green, 0; blue, 0 }  ,fill opacity=1 ] (99.8,415.23) .. controls (101.19,415.23) and (102.32,416.22) .. (102.32,417.44) .. controls (102.31,418.66) and (101.18,419.64) .. (99.79,419.64) .. controls (98.4,419.64) and (97.27,418.65) .. (97.27,417.43) .. controls (97.27,416.21) and (98.4,415.23) .. (99.8,415.23) -- cycle ;
\draw  [fill={rgb, 255:red, 0; green, 0; blue, 0 }  ,fill opacity=1 ] (265.97,415.4) .. controls (267.36,415.41) and (268.49,416.4) .. (268.49,417.61) .. controls (268.49,418.83) and (267.35,419.82) .. (265.96,419.82) .. controls (264.57,419.82) and (263.44,418.83) .. (263.44,417.61) .. controls (263.44,416.39) and (264.57,415.4) .. (265.97,415.4) -- cycle ;
\draw  [fill={rgb, 255:red, 0; green, 0; blue, 0 }  ,fill opacity=1 ] (266.25,254.77) .. controls (267.65,254.77) and (268.77,255.76) .. (268.77,256.98) .. controls (268.77,258.2) and (267.64,259.18) .. (266.24,259.18) .. controls (264.85,259.18) and (263.72,258.19) .. (263.73,256.97) .. controls (263.73,255.75) and (264.86,254.77) .. (266.25,254.77) -- cycle ;
\draw  [fill={rgb, 255:red, 0; green, 0; blue, 0 }  ,fill opacity=1 ] (100.08,254.59) .. controls (101.47,254.59) and (102.6,255.58) .. (102.6,256.8) .. controls (102.6,258.02) and (101.47,259.01) .. (100.07,259.01) .. controls (98.68,259) and (97.55,258.02) .. (97.55,256.8) .. controls (97.56,255.58) and (98.69,254.59) .. (100.08,254.59) -- cycle ;
\draw    (267.3,418.34) .. controls (305.92,375.75) and (302.28,292.26) .. (266.82,257.05) ;
\draw   (291.87,340.34) -- (294.38,335.35) -- (297.24,340.15) ;
\draw    (100.24,257.73) .. controls (137.17,218.12) and (222.17,214.12) .. (266.41,257.91) ;
\draw   (185.69,229.03) -- (180.5,226.98) -- (185.39,224.28) ;
\draw  [draw opacity=0][fill={rgb, 255:red, 200; green, 200; blue, 200 }  ,fill opacity=1 ] (578.39,419.33) -- (411.57,257.96) -- (578.39,257.96) -- cycle ;
\draw   (578.25,258.13) -- (577.97,418.76) -- (411.79,418.59) -- (412.07,257.96) -- cycle ;
\draw    (412.07,257.95) .. controls (372.01,300.7) and (376.52,383.57) .. (411.79,418.59) ;
\draw   (492.83,255.66) -- (497.88,258.04) -- (492.83,260.42) ;
\draw   (493.33,416.3) -- (498.37,418.68) -- (493.32,421.05) ;
\draw   (409.01,343.03) -- (411.74,338.61) -- (414.45,343.03) ;
\draw   (381.29,343.94) -- (383.56,339.27) -- (386.7,343.39) ;
\draw   (575.17,338.42) -- (577.94,334.03) -- (580.61,338.48) ;
\draw  [fill={rgb, 255:red, 0; green, 0; blue, 0 }  ,fill opacity=1 ] (411.79,416.38) .. controls (413.19,416.38) and (414.32,417.37) .. (414.31,418.59) .. controls (414.31,419.81) and (413.18,420.8) .. (411.79,420.8) .. controls (410.39,420.8) and (409.26,419.81) .. (409.27,418.59) .. controls (409.27,417.37) and (410.4,416.38) .. (411.79,416.38) -- cycle ;
\draw  [fill={rgb, 255:red, 0; green, 0; blue, 0 }  ,fill opacity=1 ] (577.97,416.56) .. controls (579.36,416.56) and (580.49,417.55) .. (580.49,418.77) .. controls (580.48,419.99) and (579.35,420.98) .. (577.96,420.97) .. controls (576.56,420.97) and (575.44,419.98) .. (575.44,418.76) .. controls (575.44,417.54) and (576.57,416.56) .. (577.97,416.56) -- cycle ;
\draw  [fill={rgb, 255:red, 0; green, 0; blue, 0 }  ,fill opacity=1 ] (578.25,255.92) .. controls (579.64,255.92) and (580.77,256.91) .. (580.77,258.13) .. controls (580.77,259.35) and (579.64,260.34) .. (578.24,260.34) .. controls (576.85,260.34) and (575.72,259.35) .. (575.72,258.13) .. controls (575.72,256.91) and (576.86,255.92) .. (578.25,255.92) -- cycle ;
\draw  [fill={rgb, 255:red, 0; green, 0; blue, 0 }  ,fill opacity=1 ] (412.08,255.75) .. controls (413.47,255.75) and (414.6,256.74) .. (414.6,257.96) .. controls (414.6,259.17) and (413.46,260.16) .. (412.07,260.16) .. controls (410.68,260.16) and (409.55,259.17) .. (409.55,257.95) .. controls (409.55,256.73) and (410.68,255.74) .. (412.08,255.75) -- cycle ;
\draw    (412.07,257.96) .. controls (449,218.34) and (534,214.34) .. (578.25,258.13) ;
\draw   (493.33,224.66) -- (498.38,227.04) -- (493.33,229.42) ;
\draw    (577.46,418.78) -- (411.57,257.96) ;
\draw   (498.95,346.47) -- (497.64,340.93) -- (502.4,342.75) ;
\draw    (578.8,419.67) .. controls (617.42,377.08) and (613.77,293.58) .. (578.32,258.37) ;
\draw   (603.36,341.66) -- (605.88,336.67) -- (608.74,341.47) ;
\draw   (266.58,485.82) -- (266.3,646.45) -- (100.12,646.28) -- (100.4,485.65) -- cycle ;
\draw    (100.41,485.64) .. controls (60.34,528.39) and (64.85,611.26) .. (100.12,646.28) ;
\draw   (181.17,483.35) -- (186.21,485.73) -- (181.16,488.11) ;
\draw   (181.66,643.99) -- (186.7,646.37) -- (181.65,648.74) ;
\draw   (97.35,570.72) -- (100.07,566.3) -- (102.78,570.72) ;
\draw   (69.63,571.63) -- (71.89,566.96) -- (75.03,571.08) ;
\draw   (263.5,566.11) -- (266.27,561.72) -- (268.94,566.17) ;
\draw  [fill={rgb, 255:red, 0; green, 0; blue, 0 }  ,fill opacity=1 ] (100.13,644.07) .. controls (101.52,644.07) and (102.65,645.06) .. (102.65,646.28) .. controls (102.65,647.5) and (101.51,648.49) .. (100.12,648.49) .. controls (98.73,648.49) and (97.6,647.5) .. (97.6,646.28) .. controls (97.6,645.06) and (98.73,644.07) .. (100.13,644.07) -- cycle ;
\draw  [fill={rgb, 255:red, 0; green, 0; blue, 0 }  ,fill opacity=1 ] (266.3,644.25) .. controls (267.69,644.25) and (268.82,645.24) .. (268.82,646.46) .. controls (268.82,647.68) and (267.69,648.67) .. (266.29,648.66) .. controls (264.9,648.66) and (263.77,647.67) .. (263.77,646.45) .. controls (263.77,645.23) and (264.91,644.25) .. (266.3,644.25) -- cycle ;
\draw  [fill={rgb, 255:red, 0; green, 0; blue, 0 }  ,fill opacity=1 ] (266.58,483.61) .. controls (267.98,483.61) and (269.11,484.6) .. (269.1,485.82) .. controls (269.1,487.04) and (267.97,488.03) .. (266.58,488.03) .. controls (265.18,488.03) and (264.05,487.04) .. (264.06,485.82) .. controls (264.06,484.6) and (265.19,483.61) .. (266.58,483.61) -- cycle ;
\draw  [fill={rgb, 255:red, 0; green, 0; blue, 0 }  ,fill opacity=1 ] (100.41,483.44) .. controls (101.81,483.44) and (102.93,484.43) .. (102.93,485.65) .. controls (102.93,486.86) and (101.8,487.85) .. (100.4,487.85) .. controls (99.01,487.85) and (97.88,486.86) .. (97.88,485.64) .. controls (97.89,484.42) and (99.02,483.43) .. (100.41,483.44) -- cycle ;
\draw    (267.63,647.18) .. controls (306.25,604.59) and (302.61,521.1) .. (267.16,485.89) ;
\draw   (292.2,569.18) -- (294.71,564.19) -- (297.57,568.99) ;
\draw    (100.4,485.65) .. controls (137.33,446.03) and (222.33,442.03) .. (266.58,485.82) ;
\draw   (181.67,452.35) -- (186.71,454.73) -- (181.66,457.11) ;
\draw   (579.08,483.26) -- (578.8,643.9) -- (412.62,643.72) -- (412.9,483.09) -- cycle ;
\draw    (579.08,483.26) -- (412.62,643.72) ;
\draw   (498.78,485.47) -- (493.67,483.25) -- (498.64,480.72) ;
\draw   (494.16,641.43) -- (499.2,643.82) -- (494.15,646.19) ;
\draw   (497.86,557.15) -- (503.53,556.66) -- (501.04,561.11) ;
\draw   (409.85,568.16) -- (412.57,563.75) -- (415.28,568.16) ;
\draw   (576,563.55) -- (578.77,559.17) -- (581.44,563.61) ;
\draw  [fill={rgb, 255:red, 0; green, 0; blue, 0 }  ,fill opacity=1 ] (412.63,641.52) .. controls (414.02,641.52) and (415.15,642.51) .. (415.15,643.73) .. controls (415.15,644.95) and (414.01,645.93) .. (412.62,645.93) .. controls (411.23,645.93) and (410.1,644.94) .. (410.1,643.72) .. controls (410.1,642.5) and (411.23,641.52) .. (412.63,641.52) -- cycle ;
\draw  [fill={rgb, 255:red, 0; green, 0; blue, 0 }  ,fill opacity=1 ] (578.8,641.69) .. controls (580.19,641.69) and (581.32,642.68) .. (581.32,643.9) .. controls (581.32,645.12) and (580.19,646.11) .. (578.79,646.11) .. controls (577.4,646.11) and (576.27,645.12) .. (576.27,643.9) .. controls (576.27,642.68) and (577.41,641.69) .. (578.8,641.69) -- cycle ;
\draw  [fill={rgb, 255:red, 0; green, 0; blue, 0 }  ,fill opacity=1 ] (579.08,481.06) .. controls (580.48,481.06) and (581.61,482.05) .. (581.6,483.27) .. controls (581.6,484.49) and (580.47,485.47) .. (579.08,485.47) .. controls (577.68,485.47) and (576.55,484.48) .. (576.56,483.26) .. controls (576.56,482.04) and (577.69,481.06) .. (579.08,481.06) -- cycle ;
\draw  [fill={rgb, 255:red, 0; green, 0; blue, 0 }  ,fill opacity=1 ] (412.91,480.88) .. controls (414.31,480.88) and (415.43,481.87) .. (415.43,483.09) .. controls (415.43,484.31) and (414.3,485.3) .. (412.9,485.29) .. controls (411.51,485.29) and (410.38,484.3) .. (410.38,483.08) .. controls (410.39,481.87) and (411.52,480.88) .. (412.91,480.88) -- cycle ;
\draw    (580.13,644.63) .. controls (618.75,602.04) and (615.11,518.54) .. (579.66,483.33) ;
\draw   (604.7,566.62) -- (607.21,561.63) -- (610.07,566.43) ;
\draw    (412.9,482.65) .. controls (449.83,443.03) and (534.83,439.03) .. (579.08,482.82) ;
\draw   (498.36,453.94) -- (493.17,451.89) -- (498.06,449.2) ;
\draw   (266.25,714.49) -- (265.97,875.12) -- (99.79,874.95) -- (100.07,714.31) -- cycle ;
\draw    (100.07,714.31) .. controls (60.01,757.06) and (64.52,839.93) .. (99.79,874.95) ;
\draw   (180.83,712.02) -- (185.88,714.4) -- (180.83,716.77) ;
\draw   (181.33,872.66) -- (186.37,875.04) -- (181.32,877.41) ;
\draw   (97.01,799.38) -- (99.74,794.97) -- (102.45,799.38) ;
\draw   (69.29,800.29) -- (71.56,795.63) -- (74.7,799.75) ;
\draw   (263.17,794.77) -- (265.94,790.39) -- (268.61,794.83) ;
\draw  [fill={rgb, 255:red, 0; green, 0; blue, 0 }  ,fill opacity=1 ] (99.79,872.74) .. controls (101.19,872.74) and (102.32,873.73) .. (102.31,874.95) .. controls (102.31,876.17) and (101.18,877.16) .. (99.79,877.15) .. controls (98.39,877.15) and (97.26,876.16) .. (97.27,874.94) .. controls (97.27,873.72) and (98.4,872.74) .. (99.79,872.74) -- cycle ;
\draw  [fill={rgb, 255:red, 0; green, 0; blue, 0 }  ,fill opacity=1 ] (265.97,872.92) .. controls (267.36,872.92) and (268.49,873.91) .. (268.49,875.13) .. controls (268.48,876.34) and (267.35,877.33) .. (265.96,877.33) .. controls (264.56,877.33) and (263.44,876.34) .. (263.44,875.12) .. controls (263.44,873.9) and (264.57,872.91) .. (265.97,872.92) -- cycle ;
\draw  [fill={rgb, 255:red, 0; green, 0; blue, 0 }  ,fill opacity=1 ] (266.25,712.28) .. controls (267.64,712.28) and (268.77,713.27) .. (268.77,714.49) .. controls (268.77,715.71) and (267.64,716.7) .. (266.24,716.69) .. controls (264.85,716.69) and (263.72,715.7) .. (263.72,714.48) .. controls (263.72,713.26) and (264.86,712.28) .. (266.25,712.28) -- cycle ;
\draw  [fill={rgb, 255:red, 0; green, 0; blue, 0 }  ,fill opacity=1 ] (100.08,712.1) .. controls (101.47,712.1) and (102.6,713.09) .. (102.6,714.31) .. controls (102.6,715.53) and (101.46,716.52) .. (100.07,716.52) .. controls (98.68,716.52) and (97.55,715.53) .. (97.55,714.31) .. controls (97.55,713.09) and (98.68,712.1) .. (100.08,712.1) -- cycle ;
\draw    (267.3,875.85) .. controls (305.92,833.26) and (302.27,749.77) .. (266.82,714.56) ;
\draw   (291.86,797.85) -- (294.38,792.86) -- (297.24,797.66) ;
\draw    (100.07,714.31) .. controls (137,674.69) and (222,670.69) .. (266.25,714.49) ;
\draw   (181.33,681.02) -- (186.38,683.4) -- (181.33,685.77) ;

\draw (586.99,884.44) node  [rotate=-359.9]  {$z$};
\draw (587.92,844.23) node  [font=\scriptsize,rotate=-359.9]  {$4\pi $};
\draw (453.45,849.99) node  [font=\scriptsize,rotate=-359.9]  {$4\pi $};
\draw (545.7,731.78) node  [font=\scriptsize,rotate=-359.9]  {$2\pi $};
\draw (434.92,842.52) node  [font=\scriptsize,rotate=-359.9]  {$4\pi $};
\draw (561.73,744.67) node  [font=\scriptsize,rotate=-359.9]  {$2\pi $};
\draw (589.24,745.13) node  [font=\scriptsize,rotate=-359.9]  {$2\pi $};
\draw (515.28,801.57) node    {$1$};
\draw (591.78,801.24) node    {$2$};
\draw (547.7,704.84) node  [font=\scriptsize,rotate=-359.9]  {$2\pi $};
\draw (439.21,705.97) node  [font=\scriptsize,rotate=-359.9]  {$4\pi $};
\draw (491.61,697.63) node    {$4$};
\draw (485.11,780.46) node    {$3$};
\draw (464,896.33) node [anchor=north west][inner sep=0.75pt]    {$N( 1,2,3,4)$};
\draw (268.77,195.45) node [anchor=north west][inner sep=0.75pt]    {$z$};
\draw (272.78,158.4) node  [font=\scriptsize]  {$2\pi $};
\draw (291.54,160.9) node  [font=\scriptsize]  {$4\pi $};
\draw (64.2,153.9) node  [font=\scriptsize]  {$4\pi $};
\draw (87.91,161.4) node  [font=\scriptsize]  {$2\pi $};
\draw (62.96,75.12) node  [font=\scriptsize,rotate=-359.9]  {$2\pi $};
\draw (85.06,70.58) node  [font=\scriptsize,rotate=-359.9]  {$4\pi $};
\draw (271.86,67.4) node  [font=\scriptsize]  {$4\pi $};
\draw (294.92,71.4) node  [font=\scriptsize]  {$2\pi $};
\draw (57.41,112.46) node    {$3$};
\draw (276.5,112.63) node    {$1$};
\draw (303.07,112.8) node    {$2$};
\draw (82.99,112.13) node    {$4$};
\draw (129.5,204.9) node [anchor=north west][inner sep=0.75pt]    {$\mathrm{Cyl}( 1,2,3,4)$};
\draw (586.66,201.18) node  [rotate=-359.9]  {$z$};
\draw (587.58,160.97) node  [font=\scriptsize,rotate=-359.9]  {$4\pi $};
\draw (567.12,179) node  [font=\scriptsize,rotate=-359.9]  {$\pi ^{+}$};
\draw (440.12,182.22) node  [font=\scriptsize,rotate=-359.9]  {$\pi ^{+}$};
\draw (559.87,39.51) node  [font=\scriptsize,rotate=-359.9]  {$\pi ^{+}$};
\draw (425.09,164.25) node  [font=\scriptsize,rotate=-359.9]  {$\pi ^{+}$};
\draw (403.04,159.29) node  [font=\scriptsize,rotate=-359.9]  {$4\pi $};
\draw (587.91,62.97) node  [font=\scriptsize,rotate=-359.9]  {$2\pi $};
\draw (567.9,56.5) node  [font=\scriptsize,rotate=-359.9]  {$3\pi ^{+}$};
\draw (425.88,45.75) node  [font=\scriptsize,rotate=-359.9]  {$3\pi ^{+}$};
\draw (403.92,64.79) node  [font=\scriptsize,rotate=-359.9]  {$2\pi $};
\draw (398.78,114.13) node    {$3$};
\draw (535.94,114.8) node    {$1$};
\draw (590.94,115.46) node    {$2$};
\draw (449.28,114.8) node    {$4$};
\draw (459.67,206.17) node [anchor=north west][inner sep=0.75pt]    {$U( 1,2,3,4)$};
\draw (273.66,426.18) node  [rotate=-359.9]  {$z$};
\draw (274.58,385.97) node  [font=\scriptsize,rotate=-359.9]  {$4\pi $};
\draw (254.12,404) node  [font=\scriptsize,rotate=-359.9]  {$\pi ^{+}$};
\draw (127.12,407.22) node  [font=\scriptsize,rotate=-359.9]  {$\pi ^{+}$};
\draw (90.04,384.29) node  [font=\scriptsize,rotate=-359.9]  {$4\pi $};
\draw (274.91,287.97) node  [font=\scriptsize,rotate=-359.9]  {$2\pi $};
\draw (254.9,281.5) node  [font=\scriptsize,rotate=-359.9]  {$3\pi ^{+}$};
\draw (90.92,289.79) node  [font=\scriptsize,rotate=-359.9]  {$2\pi $};
\draw (85.78,339.13) node    {$3$};
\draw (222.94,339.8) node    {$1$};
\draw (277.94,340.46) node    {$2$};
\draw (146.67,424.17) node [anchor=north west][inner sep=0.75pt]    {$T( 1,2,3,4)$};
\draw (234.54,247.95) node  [font=\scriptsize,rotate=-359.9]  {$2\pi $};
\draw (126.05,249.08) node  [font=\scriptsize,rotate=-359.9]  {$4\pi $};
\draw (178.44,240.74) node    {$4$};
\draw (585.66,427.33) node  [rotate=-359.9]  {$z$};
\draw (549.61,408.16) node  [font=\scriptsize,rotate=-359.9]  {$3\pi $};
\draw (426.62,402.88) node  [font=\scriptsize,rotate=-359.9]  {$2\pi $};
\draw (546.37,248.17) node  [font=\scriptsize,rotate=-359.9]  {$4\pi $};
\draw (423.09,283.9) node  [font=\scriptsize,rotate=-359.9]  {$2\pi $};
\draw (402.04,385.44) node  [font=\scriptsize,rotate=-359.9]  {$4\pi $};
\draw (437.88,249.3) node  [font=\scriptsize,rotate=-359.9]  {$2\pi $};
\draw (402.41,290.84) node  [font=\scriptsize,rotate=-359.9]  {$2\pi $};
\draw (398.78,342.8) node    {$3$};
\draw (451.94,342.96) node    {$4$};
\draw (490.28,240.96) node    {$1$};
\draw (458.17,426.67) node [anchor=north west][inner sep=0.75pt]    {$S( 1,2,3,4)$};
\draw (586.08,387.29) node  [font=\scriptsize,rotate=-359.9]  {$4\pi $};
\draw (587.41,288.19) node  [font=\scriptsize,rotate=-359.9]  {$2\pi $};
\draw (589.94,344.3) node    {$2$};
\draw (148.17,653.91) node [anchor=north west][inner sep=0.75pt]    {$Q( 1,2,3,4)$};
\draw (273.99,655.02) node  [rotate=-359.9]  {$z$};
\draw (274.92,614.81) node  [font=\scriptsize,rotate=-359.9]  {$4\pi $};
\draw (247.95,631.35) node  [font=\scriptsize,rotate=-359.9]  {$3\pi ^{+}$};
\draw (114.95,630.57) node  [font=\scriptsize,rotate=-359.9]  {$\pi ^{+}$};
\draw (234.7,475.86) node  [font=\scriptsize,rotate=-359.9]  {$2\pi $};
\draw (112.42,502.09) node  [font=\scriptsize,rotate=-359.9]  {$\pi ^{+}$};
\draw (90.37,613.13) node  [font=\scriptsize,rotate=-359.9]  {$4\pi $};
\draw (126.21,476.99) node  [font=\scriptsize,rotate=-359.9]  {$4\pi $};
\draw (250.73,500.24) node  [font=\scriptsize,rotate=-359.9]  {$\pi ^{+}$};
\draw (276.24,515.71) node  [font=\scriptsize,rotate=-359.9]  {$2\pi $};
\draw (90.75,518.53) node  [font=\scriptsize,rotate=-359.9]  {$2\pi $};
\draw (87.11,570.49) node    {$3$};
\draw (180.78,569.65) node    {$1$};
\draw (278.78,571.82) node    {$2$};
\draw (178.61,468.65) node    {$4$};
\draw (466.17,655.02) node [anchor=north west][inner sep=0.75pt]    {$R( 1,2,3,4)$};
\draw (586.49,652.46) node  [rotate=-359.9]  {$z$};
\draw (587.42,612.26) node  [font=\scriptsize,rotate=-359.9]  {$4\pi $};
\draw (566.45,630.29) node  [font=\scriptsize,rotate=-359.9]  {$\pi ^{+}$};
\draw (439.95,633.51) node  [font=\scriptsize,rotate=-359.9]  {$3\pi ^{+}$};
\draw (558.7,490.81) node  [font=\scriptsize,rotate=-359.9]  {$\pi ^{+}$};
\draw (425.42,615.54) node  [font=\scriptsize,rotate=-359.9]  {$3\pi ^{+}$};
\draw (425.71,495.94) node  [font=\scriptsize,rotate=-359.9]  {$\pi ^{+}$};
\draw (568.23,507.69) node  [font=\scriptsize,rotate=-359.9]  {$\pi ^{+}$};
\draw (588.74,513.15) node  [font=\scriptsize,rotate=-359.9]  {$2\pi $};
\draw (536.28,568.6) node    {$1$};
\draw (591.28,569.26) node    {$2$};
\draw (547.2,472.86) node  [font=\scriptsize,rotate=-359.9]  {$2\pi $};
\draw (438.71,473.99) node  [font=\scriptsize,rotate=-359.9]  {$4\pi $};
\draw (491.11,465.65) node    {$4$};
\draw (441.61,567.99) node    {$3$};
\draw (273.66,883.69) node  [rotate=-359.9]  {$z$};
\draw (274.58,843.48) node  [font=\scriptsize,rotate=-359.9]  {$4\pi $};
\draw (253.11,861.51) node  [font=\scriptsize,rotate=-359.9]  {$\pi ^{+}$};
\draw (114.62,859.24) node  [font=\scriptsize,rotate=-359.9]  {$3\pi ^{+}$};
\draw (234.37,704.53) node  [font=\scriptsize,rotate=-359.9]  {$4\pi $};
\draw (112.09,730.76) node  [font=\scriptsize,rotate=-359.9]  {$\pi ^{+}$};
\draw (90.04,841.8) node  [font=\scriptsize,rotate=-359.9]  {$4\pi $};
\draw (125.88,705.66) node  [font=\scriptsize,rotate=-359.9]  {$2\pi $};
\draw (250.4,728.91) node  [font=\scriptsize,rotate=-359.9]  {$\pi ^{+}$};
\draw (275.91,744.38) node  [font=\scriptsize,rotate=-359.9]  {$2\pi $};
\draw (90.41,747.2) node  [font=\scriptsize,rotate=-359.9]  {$2\pi $};
\draw (86.78,799.15) node    {$3$};
\draw (180.44,798.32) node    {$4$};
\draw (278.44,800.49) node    {$2$};
\draw (178.28,697.32) node    {$1$};
\draw (146.17,890.02) node [anchor=north west][inner sep=0.75pt]    {$P( 1,2,3,4)$};

\end{tikzpicture}

%% file: diagrams/arcs.tex
\tikzset{every picture/.style={line width=0.75pt}} 

\begin{tikzpicture}[x=0.75pt,y=0.75pt,yscale=-1,xscale=1]

\draw   (266.58,40.31) -- (266.3,200.94) -- (100.13,200.77) -- (100.41,40.13) -- cycle ;
\draw    (100.41,40.13) .. controls (60.35,82.88) and (64.86,165.75) .. (100.13,200.77) ;
\draw    (266.58,40.31) -- (100.13,200.77) ;
\draw   (186.14,42.67) -- (181.17,40.15) -- (186.28,37.92) ;
\draw   (186.67,203.27) -- (181.66,200.82) -- (186.74,198.52) ;
\draw   (185.3,114.71) -- (190.84,113.43) -- (189,118.19) ;
\draw   (97.35,125.21) -- (100.07,120.79) -- (102.79,125.21) ;
\draw   (69.63,126.12) -- (71.9,121.45) -- (75.04,125.57) ;
\draw   (263.51,120.6) -- (266.28,116.21) -- (268.94,120.66) ;
\draw  [fill={rgb, 255:red, 0; green, 0; blue, 0 }  ,fill opacity=1 ] (100.13,198.56) .. controls (101.52,198.56) and (102.65,199.55) .. (102.65,200.77) .. controls (102.65,201.99) and (101.52,202.98) .. (100.12,202.98) .. controls (98.73,202.98) and (97.6,201.99) .. (97.6,200.77) .. controls (97.6,199.55) and (98.74,198.56) .. (100.13,198.56) -- cycle ;
\draw  [fill={rgb, 255:red, 0; green, 0; blue, 0 }  ,fill opacity=1 ] (266.3,198.74) .. controls (267.7,198.74) and (268.82,199.73) .. (268.82,200.95) .. controls (268.82,202.17) and (267.69,203.15) .. (266.29,203.15) .. controls (264.9,203.15) and (263.77,202.16) .. (263.77,200.94) .. controls (263.78,199.72) and (264.91,198.74) .. (266.3,198.74) -- cycle ;
\draw  [fill={rgb, 255:red, 0; green, 0; blue, 0 }  ,fill opacity=1 ] (266.59,38.1) .. controls (267.98,38.1) and (269.11,39.09) .. (269.11,40.31) .. controls (269.1,41.53) and (267.97,42.52) .. (266.58,42.52) .. controls (265.18,42.51) and (264.06,41.53) .. (264.06,40.31) .. controls (264.06,39.09) and (265.19,38.1) .. (266.59,38.1) -- cycle ;
\draw  [fill={rgb, 255:red, 0; green, 0; blue, 0 }  ,fill opacity=1 ] (100.41,37.92) .. controls (101.81,37.93) and (102.94,38.92) .. (102.93,40.13) .. controls (102.93,41.35) and (101.8,42.34) .. (100.41,42.34) .. controls (99.01,42.34) and (97.88,41.35) .. (97.89,40.13) .. controls (97.89,38.91) and (99.02,37.92) .. (100.41,37.92) -- cycle ;
\draw    (267.64,201.67) .. controls (306.26,159.08) and (302.61,75.59) .. (267.16,40.38) ;
\draw   (292.2,123.67) -- (294.71,118.68) -- (297.57,123.48) ;
\draw   (585.91,37.91) -- (585.63,198.54) -- (419.46,198.37) -- (419.74,37.73) -- cycle ;
\draw    (419.74,37.73) .. controls (379.68,80.48) and (384.19,163.35) .. (419.46,198.37) ;
\draw    (585.91,37.91) -- (419.46,198.37) ;
\draw   (500.5,35.44) -- (505.54,37.82) -- (500.49,40.19) ;
\draw   (500.99,196.08) -- (506.04,198.46) -- (500.99,200.83) ;
\draw   (512,113) -- (506.39,113.93) -- (508.52,109.3) ;
\draw   (416.68,122.8) -- (419.4,118.39) -- (422.12,122.8) ;
\draw   (388.96,123.71) -- (391.23,119.05) -- (394.37,123.17) ;
\draw   (582.84,118.2) -- (585.61,113.81) -- (588.27,118.25) ;
\draw  [fill={rgb, 255:red, 0; green, 0; blue, 0 }  ,fill opacity=1 ] (419.46,196.16) .. controls (420.85,196.16) and (421.98,197.15) .. (421.98,198.37) .. controls (421.98,199.59) and (420.85,200.58) .. (419.45,200.58) .. controls (418.06,200.57) and (416.93,199.58) .. (416.93,198.36) .. controls (416.94,197.15) and (418.07,196.16) .. (419.46,196.16) -- cycle ;
\draw  [fill={rgb, 255:red, 0; green, 0; blue, 0 }  ,fill opacity=1 ] (585.63,196.34) .. controls (587.03,196.34) and (588.15,197.33) .. (588.15,198.55) .. controls (588.15,199.77) and (587.02,200.75) .. (585.62,200.75) .. controls (584.23,200.75) and (583.1,199.76) .. (583.1,198.54) .. controls (583.11,197.32) and (584.24,196.34) .. (585.63,196.34) -- cycle ;
\draw  [fill={rgb, 255:red, 0; green, 0; blue, 0 }  ,fill opacity=1 ] (585.92,35.7) .. controls (587.31,35.7) and (588.44,36.69) .. (588.44,37.91) .. controls (588.43,39.13) and (587.3,40.12) .. (585.91,40.11) .. controls (584.52,40.11) and (583.39,39.12) .. (583.39,37.9) .. controls (583.39,36.69) and (584.52,35.7) .. (585.92,35.7) -- cycle ;
\draw  [fill={rgb, 255:red, 0; green, 0; blue, 0 }  ,fill opacity=1 ] (419.75,35.52) .. controls (421.14,35.52) and (422.27,36.51) .. (422.26,37.73) .. controls (422.26,38.95) and (421.13,39.94) .. (419.74,39.94) .. controls (418.34,39.94) and (417.22,38.95) .. (417.22,37.73) .. controls (417.22,36.51) and (418.35,35.52) .. (419.75,35.52) -- cycle ;
\draw    (586.97,199.27) .. controls (625.59,156.68) and (621.94,73.19) .. (586.49,37.98) ;
\draw   (611.53,121.27) -- (614.04,116.28) -- (616.9,121.08) ;
\draw   (265.41,256.19) -- (265.13,416.82) -- (98.96,416.65) -- (99.24,256.01) -- cycle ;
\draw    (99.24,256.01) .. controls (59.18,298.76) and (63.69,381.63) .. (98.96,416.65) ;
\draw    (265.41,256.19) -- (98.96,416.65) ;
\draw   (180,253.72) -- (185.04,256.1) -- (179.99,258.47) ;
\draw   (180.49,414.35) -- (185.54,416.74) -- (180.49,419.11) ;
\draw   (184.12,330.89) -- (189.55,329.17) -- (188.09,334.05) ;
\draw   (96.18,341.08) -- (98.9,336.67) -- (101.62,341.08) ;
\draw   (68.46,341.99) -- (70.73,337.33) -- (73.87,341.45) ;
\draw   (262.34,336.47) -- (265.11,332.09) -- (267.77,336.53) ;
\draw  [fill={rgb, 255:red, 0; green, 0; blue, 0 }  ,fill opacity=1 ] (98.96,414.44) .. controls (100.35,414.44) and (101.48,415.43) .. (101.48,416.65) .. controls (101.48,417.87) and (100.35,418.85) .. (98.95,418.85) .. controls (97.56,418.85) and (96.43,417.86) .. (96.43,416.64) .. controls (96.44,415.42) and (97.57,414.44) .. (98.96,414.44) -- cycle ;
\draw  [fill={rgb, 255:red, 0; green, 0; blue, 0 }  ,fill opacity=1 ] (265.13,414.61) .. controls (266.53,414.62) and (267.65,415.61) .. (267.65,416.82) .. controls (267.65,418.04) and (266.52,419.03) .. (265.12,419.03) .. controls (263.73,419.03) and (262.6,418.04) .. (262.6,416.82) .. controls (262.61,415.6) and (263.74,414.61) .. (265.13,414.61) -- cycle ;
\draw  [fill={rgb, 255:red, 0; green, 0; blue, 0 }  ,fill opacity=1 ] (265.42,253.98) .. controls (266.81,253.98) and (267.94,254.97) .. (267.94,256.19) .. controls (267.93,257.41) and (266.8,258.39) .. (265.41,258.39) .. controls (264.02,258.39) and (262.89,257.4) .. (262.89,256.18) .. controls (262.89,254.96) and (264.02,253.98) .. (265.42,253.98) -- cycle ;
\draw  [fill={rgb, 255:red, 0; green, 0; blue, 0 }  ,fill opacity=1 ] (99.25,253.8) .. controls (100.64,253.8) and (101.77,254.79) .. (101.76,256.01) .. controls (101.76,257.23) and (100.63,258.22) .. (99.24,258.22) .. controls (97.84,258.21) and (96.72,257.23) .. (96.72,256.01) .. controls (96.72,254.79) and (97.85,253.8) .. (99.25,253.8) -- cycle ;
\draw    (266.47,417.55) .. controls (305.09,374.96) and (301.44,291.47) .. (265.99,256.26) ;
\draw   (291.03,339.55) -- (293.54,334.55) -- (296.4,339.36) ;
\draw   (585.75,257.57) -- (585.47,418.21) -- (419.29,418.03) -- (419.57,257.4) -- cycle ;
\draw    (419.57,257.4) .. controls (379.51,300.15) and (384.02,383.02) .. (419.29,418.03) ;
\draw   (500.33,255.11) -- (505.38,257.49) -- (500.33,259.86) ;
\draw   (500.83,415.74) -- (505.87,418.13) -- (500.82,420.5) ;
\draw   (416.51,342.47) -- (419.24,338.06) -- (421.95,342.47) ;
\draw   (388.79,343.38) -- (391.06,338.72) -- (394.2,342.84) ;
\draw   (582.67,337.86) -- (585.44,333.48) -- (588.11,337.92) ;
\draw  [fill={rgb, 255:red, 0; green, 0; blue, 0 }  ,fill opacity=1 ] (419.29,415.83) .. controls (420.69,415.83) and (421.82,416.82) .. (421.81,418.04) .. controls (421.81,419.26) and (420.68,420.24) .. (419.29,420.24) .. controls (417.89,420.24) and (416.76,419.25) .. (416.77,418.03) .. controls (416.77,416.81) and (417.9,415.83) .. (419.29,415.83) -- cycle ;
\draw  [fill={rgb, 255:red, 0; green, 0; blue, 0 }  ,fill opacity=1 ] (585.47,416) .. controls (586.86,416) and (587.99,416.99) .. (587.99,418.21) .. controls (587.98,419.43) and (586.85,420.42) .. (585.46,420.42) .. controls (584.06,420.42) and (582.94,419.43) .. (582.94,418.21) .. controls (582.94,416.99) and (584.07,416) .. (585.47,416) -- cycle ;
\draw  [fill={rgb, 255:red, 0; green, 0; blue, 0 }  ,fill opacity=1 ] (585.75,255.37) .. controls (587.14,255.37) and (588.27,256.36) .. (588.27,257.58) .. controls (588.27,258.8) and (587.14,259.78) .. (585.74,259.78) .. controls (584.35,259.78) and (583.22,258.79) .. (583.22,257.57) .. controls (583.22,256.35) and (584.36,255.37) .. (585.75,255.37) -- cycle ;
\draw  [fill={rgb, 255:red, 0; green, 0; blue, 0 }  ,fill opacity=1 ] (419.58,255.19) .. controls (420.97,255.19) and (422.1,256.18) .. (422.1,257.4) .. controls (422.1,258.62) and (420.96,259.61) .. (419.57,259.6) .. controls (418.18,259.6) and (417.05,258.61) .. (417.05,257.39) .. controls (417.05,256.18) and (418.18,255.19) .. (419.58,255.19) -- cycle ;
\draw    (586.8,418.94) .. controls (625.42,376.35) and (621.77,292.85) .. (586.32,257.64) ;
\draw   (611.36,340.93) -- (613.88,335.94) -- (616.74,340.75) ;
\draw    (419.57,257.4) .. controls (456.5,217.78) and (541.5,213.78) .. (585.75,257.57) ;
\draw   (500.83,224.11) -- (505.88,226.49) -- (500.83,228.86) ;
\draw   (264.75,478.02) -- (264.47,638.65) -- (98.29,638.48) -- (98.57,477.84) -- cycle ;
\draw    (264.75,478.02) -- (98.29,638.48) ;
\draw   (184.45,480.23) -- (179.33,478.01) -- (184.3,475.48) ;
\draw   (179.83,636.19) -- (184.87,638.57) -- (179.82,640.94) ;
\draw   (183.53,551.91) -- (189.2,551.42) -- (186.71,555.86) ;
\draw   (95.51,562.91) -- (98.24,558.5) -- (100.95,562.91) ;
\draw   (261.67,558.31) -- (264.44,553.92) -- (267.11,558.37) ;
\draw  [fill={rgb, 255:red, 0; green, 0; blue, 0 }  ,fill opacity=1 ] (98.29,636.27) .. controls (99.69,636.27) and (100.82,637.26) .. (100.81,638.48) .. controls (100.81,639.7) and (99.68,640.69) .. (98.29,640.69) .. controls (96.89,640.68) and (95.76,639.7) .. (95.77,638.48) .. controls (95.77,637.26) and (96.9,636.27) .. (98.29,636.27) -- cycle ;
\draw  [fill={rgb, 255:red, 0; green, 0; blue, 0 }  ,fill opacity=1 ] (264.47,636.45) .. controls (265.86,636.45) and (266.99,637.44) .. (266.99,638.66) .. controls (266.98,639.88) and (265.85,640.86) .. (264.46,640.86) .. controls (263.06,640.86) and (261.94,639.87) .. (261.94,638.65) .. controls (261.94,637.43) and (263.07,636.45) .. (264.47,636.45) -- cycle ;
\draw  [fill={rgb, 255:red, 0; green, 0; blue, 0 }  ,fill opacity=1 ] (264.75,475.81) .. controls (266.14,475.81) and (267.27,476.8) .. (267.27,478.02) .. controls (267.27,479.24) and (266.14,480.23) .. (264.74,480.23) .. controls (263.35,480.22) and (262.22,479.24) .. (262.22,478.02) .. controls (262.22,476.8) and (263.36,475.81) .. (264.75,475.81) -- cycle ;
\draw  [fill={rgb, 255:red, 0; green, 0; blue, 0 }  ,fill opacity=1 ] (98.58,475.63) .. controls (99.97,475.64) and (101.1,476.63) .. (101.1,477.84) .. controls (101.1,479.06) and (99.96,480.05) .. (98.57,480.05) .. controls (97.18,480.05) and (96.05,479.06) .. (96.05,477.84) .. controls (96.05,476.62) and (97.18,475.63) .. (98.58,475.63) -- cycle ;
\draw    (265.8,639.38) .. controls (304.42,596.79) and (300.77,513.3) .. (265.32,478.09) ;
\draw   (290.36,561.38) -- (292.88,556.39) -- (295.74,561.19) ;
\draw    (98.57,477.4) .. controls (135.5,437.78) and (220.5,433.78) .. (264.75,477.57) ;
\draw   (184.02,448.7) -- (178.84,446.65) -- (183.72,443.95) ;
\draw   (585.75,478.5) -- (585.47,639.13) -- (419.29,638.96) -- (419.57,478.33) -- cycle ;
\draw    (419.57,478.32) .. controls (379.51,521.07) and (384.02,603.94) .. (419.29,638.96) ;
\draw   (500.33,476.03) -- (505.38,478.41) -- (500.33,480.79) ;
\draw   (500.83,636.67) -- (505.87,639.05) -- (500.82,641.42) ;
\draw   (416.51,563.4) -- (419.24,558.98) -- (421.95,563.4) ;
\draw   (388.79,564.31) -- (391.06,559.64) -- (394.2,563.76) ;
\draw   (582.67,558.79) -- (585.44,554.4) -- (588.11,558.85) ;
\draw  [fill={rgb, 255:red, 0; green, 0; blue, 0 }  ,fill opacity=1 ] (419.29,636.75) .. controls (420.69,636.75) and (421.82,637.74) .. (421.81,638.96) .. controls (421.81,640.18) and (420.68,641.17) .. (419.29,641.17) .. controls (417.89,641.17) and (416.76,640.18) .. (416.77,638.96) .. controls (416.77,637.74) and (417.9,636.75) .. (419.29,636.75) -- cycle ;
\draw  [fill={rgb, 255:red, 0; green, 0; blue, 0 }  ,fill opacity=1 ] (585.47,636.93) .. controls (586.86,636.93) and (587.99,637.92) .. (587.99,639.14) .. controls (587.98,640.36) and (586.85,641.35) .. (585.46,641.34) .. controls (584.06,641.34) and (582.94,640.35) .. (582.94,639.13) .. controls (582.94,637.92) and (584.07,636.93) .. (585.47,636.93) -- cycle ;
\draw  [fill={rgb, 255:red, 0; green, 0; blue, 0 }  ,fill opacity=1 ] (585.75,476.29) .. controls (587.14,476.29) and (588.27,477.28) .. (588.27,478.5) .. controls (588.27,479.72) and (587.14,480.71) .. (585.74,480.71) .. controls (584.35,480.71) and (583.22,479.72) .. (583.22,478.5) .. controls (583.22,477.28) and (584.36,476.29) .. (585.75,476.29) -- cycle ;
\draw  [fill={rgb, 255:red, 0; green, 0; blue, 0 }  ,fill opacity=1 ] (419.58,476.12) .. controls (420.97,476.12) and (422.1,477.11) .. (422.1,478.33) .. controls (422.1,479.55) and (420.96,480.53) .. (419.57,480.53) .. controls (418.18,480.53) and (417.05,479.54) .. (417.05,478.32) .. controls (417.05,477.1) and (418.18,476.11) .. (419.58,476.12) -- cycle ;
\draw    (586.8,639.87) .. controls (625.42,597.27) and (621.77,513.78) .. (586.32,478.57) ;
\draw   (611.36,561.86) -- (613.88,556.87) -- (616.74,561.67) ;
\draw    (419.57,478.33) .. controls (456.5,438.71) and (541.5,434.71) .. (585.75,478.5) ;
\draw   (500.83,445.03) -- (505.88,447.41) -- (500.83,449.79) ;

\draw (273.99,209.51) node  [rotate=-359.9]  {$z$};
\draw (274.92,169.3) node  [font=\scriptsize,rotate=-359.9]  {$4\pi $};
\draw (254.45,187.34) node  [font=\scriptsize,rotate=-359.9]  {$2\pi $};
\draw (127.46,190.56) node  [font=\scriptsize,rotate=-359.9]  {$\pi $};
\draw (248.21,47.85) node  [font=\scriptsize,rotate=-359.9]  {$\pi $};
\draw (112.42,172.58) node  [font=\scriptsize,rotate=-359.9]  {$2\pi $};
\draw (90.37,167.62) node  [font=\scriptsize,rotate=-359.9]  {$4\pi $};
\draw (275.25,71.3) node  [font=\scriptsize,rotate=-359.9]  {$2\pi $};
\draw (256.23,63.83) node  [font=\scriptsize,rotate=-359.9]  {$4\pi $};
\draw (113.22,54.08) node  [font=\scriptsize,rotate=-359.9]  {$4\pi $};
\draw (91.25,73.12) node  [font=\scriptsize,rotate=-359.9]  {$2\pi $};
\draw (86.11,122.46) node    {$3$};
\draw (223.28,123.13) node    {$1$};
\draw (278.28,123.8) node    {$2$};
\draw (136.61,123.13) node    {$4$};
\draw (188.55,215.5) node    {$DA( 1,2,3,4)$};
\draw (593.32,207.11) node  [rotate=-359.9]  {$z$};
\draw (594.25,166.9) node  [font=\scriptsize,rotate=-359.9]  {$4\pi $};
\draw (573.28,184.94) node  [font=\scriptsize,rotate=-359.9]  {$2\pi $};
\draw (446.79,188.16) node  [font=\scriptsize,rotate=-359.9]  {$2\pi $};
\draw (565.54,45.45) node  [font=\scriptsize,rotate=-359.9]  {$2\pi $};
\draw (432.25,170.18) node  [font=\scriptsize,rotate=-359.9]  {$\pi $};
\draw (409.7,165.22) node  [font=\scriptsize,rotate=-359.9]  {$4\pi $};
\draw (432.55,50.58) node  [font=\scriptsize,rotate=-359.9]  {$4\pi $};
\draw (575.07,62.33) node  [font=\scriptsize,rotate=-359.9]  {$3\pi $};
\draw (595.58,67.8) node  [font=\scriptsize,rotate=-359.9]  {$2\pi $};
\draw (410.08,70.62) node  [font=\scriptsize,rotate=-359.9]  {$2\pi $};
\draw (406.44,122.57) node    {$3$};
\draw (543.11,123.24) node    {$1$};
\draw (598.11,123.91) node    {$2$};
\draw (456.44,123.24) node    {$4$};
\draw (507.04,211.44) node    {$DB( 1,2,3,4)$};
\draw (272.82,425.39) node  [rotate=-359.9]  {$z$};
\draw (273.75,385.18) node  [font=\scriptsize,rotate=-359.9]  {$4\pi $};
\draw (252.78,403.21) node  [font=\scriptsize,rotate=-359.9]  {$\pi $};
\draw (126.29,406.43) node  [font=\scriptsize,rotate=-359.9]  {$2\pi $};
\draw (245.04,263.73) node  [font=\scriptsize,rotate=-359.9]  {$2\pi $};
\draw (111.75,388.46) node  [font=\scriptsize,rotate=-359.9]  {$2\pi $};
\draw (89.2,383.5) node  [font=\scriptsize,rotate=-359.9]  {$4\pi $};
\draw (112.05,268.86) node  [font=\scriptsize,rotate=-359.9]  {$3\pi $};
\draw (254.57,280.61) node  [font=\scriptsize,rotate=-359.9]  {$4\pi $};
\draw (275.08,286.07) node  [font=\scriptsize,rotate=-359.9]  {$2\pi $};
\draw (89.58,288.9) node  [font=\scriptsize,rotate=-359.9]  {$2\pi $};
\draw (85.94,340.85) node    {$3$};
\draw (222.61,341.52) node    {$1$};
\draw (277.61,342.19) node    {$2$};
\draw (135.94,341.52) node    {$4$};
\draw (186.64,430.72) node    {$DC( 1,2,3,4)$};
\draw (593.16,426.77) node  [rotate=-359.9]  {$z$};
\draw (594.08,386.57) node  [font=\scriptsize,rotate=-359.9]  {$4\pi $};
\draw (572.61,404.6) node  [font=\scriptsize,rotate=-359.9]  {$3\pi $};
\draw (434.12,402.32) node  [font=\scriptsize,rotate=-359.9]  {$2\pi $};
\draw (553.87,247.62) node  [font=\scriptsize,rotate=-359.9]  {$2\pi $};
\draw (431.59,273.85) node  [font=\scriptsize,rotate=-359.9]  {$\pi $};
\draw (409.54,384.89) node  [font=\scriptsize,rotate=-359.9]  {$4\pi $};
\draw (445.38,248.75) node  [font=\scriptsize,rotate=-359.9]  {$4\pi $};
\draw (569.9,272) node  [font=\scriptsize,rotate=-359.9]  {$2\pi $};
\draw (595.41,287.46) node  [font=\scriptsize,rotate=-359.9]  {$2\pi $};
\draw (409.91,290.29) node  [font=\scriptsize,rotate=-359.9]  {$2\pi $};
\draw (406.28,342.24) node    {$3$};
\draw (499.94,341.41) node    {$1$};
\draw (597.94,343.57) node    {$2$};
\draw (497.78,240.41) node    {$4$};
\draw (507.13,431.11) node    {$DE( 1,2,3,4)$};
\draw (272.16,647.22) node  [rotate=-359.9]  {$z$};
\draw (273.08,607.01) node  [font=\scriptsize,rotate=-359.9]  {$4\pi $};
\draw (252.11,625.05) node  [font=\scriptsize,rotate=-359.9]  {$2\pi $};
\draw (125.62,628.27) node  [font=\scriptsize,rotate=-359.9]  {$3\pi $};
\draw (244.37,485.56) node  [font=\scriptsize,rotate=-359.9]  {$\pi $};
\draw (111.09,610.29) node  [font=\scriptsize,rotate=-359.9]  {$4\pi $};
\draw (111.38,490.69) node  [font=\scriptsize,rotate=-359.9]  {$2\pi $};
\draw (253.9,502.44) node  [font=\scriptsize,rotate=-359.9]  {$2\pi $};
\draw (274.41,507.91) node  [font=\scriptsize,rotate=-359.9]  {$2\pi $};
\draw (221.94,563.35) node    {$1$};
\draw (276.94,564.02) node    {$2$};
\draw (232.87,467.62) node  [font=\scriptsize,rotate=-359.9]  {$2\pi $};
\draw (124.38,468.75) node  [font=\scriptsize,rotate=-359.9]  {$4\pi $};
\draw (176.78,460.41) node    {$4$};
\draw (127.28,562.74) node    {$3$};
\draw (150.17,645.11) node [anchor=north west][inner sep=0.75pt]    {$DF( 1,2,3,4)$};
\draw (593.16,647.7) node  [rotate=-359.9]  {$z$};
\draw (594.08,607.49) node  [font=\scriptsize,rotate=-359.9]  {$4\pi $};
\draw (572.61,625.53) node  [font=\scriptsize,rotate=-359.9]  {$2\pi $};
\draw (434.12,623.25) node  [font=\scriptsize,rotate=-359.9]  {$3\pi $};
\draw (553.87,468.54) node  [font=\scriptsize,rotate=-359.9]  {$4\pi $};
\draw (431.59,494.77) node  [font=\scriptsize,rotate=-359.9]  {$2\pi $};
\draw (409.54,605.81) node  [font=\scriptsize,rotate=-359.9]  {$4\pi $};
\draw (445.38,469.67) node  [font=\scriptsize,rotate=-359.9]  {$2\pi $};
\draw (569.9,492.93) node  [font=\scriptsize,rotate=-359.9]  {$\pi $};
\draw (595.41,508.39) node  [font=\scriptsize,rotate=-359.9]  {$2\pi $};
\draw (409.91,511.21) node  [font=\scriptsize,rotate=-359.9]  {$2\pi $};
\draw (406.28,563.17) node    {$3$};
\draw (499.94,562.33) node    {$4$};
\draw (597.94,564.5) node    {$2$};
\draw (497.78,461.33) node    {$1$};
\draw (465.67,646.04) node [anchor=north west][inner sep=0.75pt]    {$DG( 1,2,3,4)$};

\end{tikzpicture}

%% file: diagrams/arcs2.tex
\tikzset{every picture/.style={line width=0.75pt}} 

\begin{tikzpicture}[x=0.75pt,y=0.75pt,yscale=-1,xscale=1]

\draw   (265.25,54.83) -- (264.97,215.47) -- (98.79,215.29) -- (99.07,54.66) -- cycle ;
\draw    (99.07,54.66) .. controls (59.01,97.41) and (63.52,180.28) .. (98.79,215.29) ;
\draw   (179.83,52.37) -- (184.88,54.75) -- (179.83,57.12) ;
\draw   (180.33,213) -- (185.37,215.39) -- (180.32,217.76) ;
\draw   (96.01,139.73) -- (98.74,135.31) -- (101.45,139.73) ;
\draw   (68.29,140.64) -- (70.56,135.97) -- (73.7,140.1) ;
\draw   (262.17,135.12) -- (264.94,130.74) -- (267.61,135.18) ;
\draw  [fill={rgb, 255:red, 0; green, 0; blue, 0 }  ,fill opacity=1 ] (98.79,213.09) .. controls (100.19,213.09) and (101.32,214.08) .. (101.31,215.3) .. controls (101.31,216.52) and (100.18,217.5) .. (98.79,217.5) .. controls (97.39,217.5) and (96.26,216.51) .. (96.27,215.29) .. controls (96.27,214.07) and (97.4,213.08) .. (98.79,213.09) -- cycle ;
\draw  [fill={rgb, 255:red, 0; green, 0; blue, 0 }  ,fill opacity=1 ] (264.97,213.26) .. controls (266.36,213.26) and (267.49,214.25) .. (267.49,215.47) .. controls (267.48,216.69) and (266.35,217.68) .. (264.96,217.68) .. controls (263.56,217.68) and (262.44,216.69) .. (262.44,215.47) .. controls (262.44,214.25) and (263.57,213.26) .. (264.97,213.26) -- cycle ;
\draw  [fill={rgb, 255:red, 0; green, 0; blue, 0 }  ,fill opacity=1 ] (265.25,52.63) .. controls (266.64,52.63) and (267.77,53.62) .. (267.77,54.84) .. controls (267.77,56.06) and (266.64,57.04) .. (265.24,57.04) .. controls (263.85,57.04) and (262.72,56.05) .. (262.72,54.83) .. controls (262.72,53.61) and (263.86,52.62) .. (265.25,52.63) -- cycle ;
\draw  [fill={rgb, 255:red, 0; green, 0; blue, 0 }  ,fill opacity=1 ] (99.08,52.45) .. controls (100.47,52.45) and (101.6,53.44) .. (101.6,54.66) .. controls (101.6,55.88) and (100.46,56.87) .. (99.07,56.86) .. controls (97.68,56.86) and (96.55,55.87) .. (96.55,54.65) .. controls (96.55,53.43) and (97.68,52.45) .. (99.08,52.45) -- cycle ;
\draw    (99.07,54.66) .. controls (136,15.04) and (221,11.04) .. (265.25,54.83) ;
\draw   (180.33,21.37) -- (185.38,23.75) -- (180.33,26.12) ;
\draw    (264.46,215.48) -- (98.57,54.67) ;
\draw   (185.95,143.18) -- (184.64,137.64) -- (189.4,139.45) ;
\draw   (589.78,53.61) -- (589.5,214.25) -- (423.32,214.07) -- (423.6,53.44) -- cycle ;
\draw    (423.61,53.43) .. controls (383.54,96.18) and (388.05,179.05) .. (423.32,214.07) ;
\draw   (504.37,51.14) -- (509.41,53.53) -- (504.36,55.9) ;
\draw   (504.86,211.78) -- (509.9,214.16) -- (504.85,216.54) ;
\draw   (420.55,138.51) -- (423.27,134.09) -- (425.98,138.51) ;
\draw   (392.83,139.42) -- (395.09,134.75) -- (398.23,138.87) ;
\draw   (586.7,133.9) -- (589.47,129.51) -- (592.14,133.96) ;
\draw  [fill={rgb, 255:red, 0; green, 0; blue, 0 }  ,fill opacity=1 ] (423.33,211.86) .. controls (424.72,211.87) and (425.85,212.85) .. (425.85,214.07) .. controls (425.85,215.29) and (424.71,216.28) .. (423.32,216.28) .. controls (421.93,216.28) and (420.8,215.29) .. (420.8,214.07) .. controls (420.8,212.85) and (421.93,211.86) .. (423.33,211.86) -- cycle ;
\draw  [fill={rgb, 255:red, 0; green, 0; blue, 0 }  ,fill opacity=1 ] (589.5,212.04) .. controls (590.89,212.04) and (592.02,213.03) .. (592.02,214.25) .. controls (592.02,215.47) and (590.89,216.46) .. (589.49,216.46) .. controls (588.1,216.45) and (586.97,215.46) .. (586.97,214.25) .. controls (586.97,213.03) and (588.11,212.04) .. (589.5,212.04) -- cycle ;
\draw  [fill={rgb, 255:red, 0; green, 0; blue, 0 }  ,fill opacity=1 ] (589.78,51.4) .. controls (591.18,51.41) and (592.31,52.39) .. (592.3,53.61) .. controls (592.3,54.83) and (591.17,55.82) .. (589.78,55.82) .. controls (588.38,55.82) and (587.25,54.83) .. (587.26,53.61) .. controls (587.26,52.39) and (588.39,51.4) .. (589.78,51.4) -- cycle ;
\draw  [fill={rgb, 255:red, 0; green, 0; blue, 0 }  ,fill opacity=1 ] (423.61,51.23) .. controls (425.01,51.23) and (426.13,52.22) .. (426.13,53.44) .. controls (426.13,54.66) and (425,55.64) .. (423.6,55.64) .. controls (422.21,55.64) and (421.08,54.65) .. (421.08,53.43) .. controls (421.09,52.21) and (422.22,51.23) .. (423.61,51.23) -- cycle ;
\draw    (590.83,214.98) .. controls (629.45,172.39) and (625.81,88.89) .. (590.36,53.68) ;
\draw   (615.4,136.97) -- (617.91,131.98) -- (620.77,136.78) ;
\draw    (423.6,53.44) .. controls (460.53,13.82) and (545.53,9.82) .. (589.78,53.61) ;
\draw   (504.87,20.14) -- (509.91,22.53) -- (504.86,24.9) ;
\draw   (268.58,279.94) -- (268.3,440.58) -- (102.12,440.4) -- (102.4,279.77) -- cycle ;
\draw    (268.58,279.94) -- (102.12,440.4) ;
\draw   (188.28,282.15) -- (183.17,279.93) -- (188.14,277.4) ;
\draw   (183.66,438.11) -- (188.7,440.5) -- (183.65,442.87) ;
\draw   (187.36,353.83) -- (193.03,353.34) -- (190.54,357.79) ;
\draw   (99.35,364.84) -- (102.07,360.43) -- (104.78,364.84) ;
\draw   (265.5,360.23) -- (268.27,355.85) -- (270.94,360.29) ;
\draw  [fill={rgb, 255:red, 0; green, 0; blue, 0 }  ,fill opacity=1 ] (102.13,438.2) .. controls (103.52,438.2) and (104.65,439.19) .. (104.65,440.41) .. controls (104.65,441.63) and (103.51,442.61) .. (102.12,442.61) .. controls (100.73,442.61) and (99.6,441.62) .. (99.6,440.4) .. controls (99.6,439.18) and (100.73,438.2) .. (102.13,438.2) -- cycle ;
\draw  [fill={rgb, 255:red, 0; green, 0; blue, 0 }  ,fill opacity=1 ] (268.3,438.37) .. controls (269.69,438.38) and (270.82,439.36) .. (270.82,440.58) .. controls (270.82,441.8) and (269.69,442.79) .. (268.29,442.79) .. controls (266.9,442.79) and (265.77,441.8) .. (265.77,440.58) .. controls (265.77,439.36) and (266.91,438.37) .. (268.3,438.37) -- cycle ;
\draw  [fill={rgb, 255:red, 0; green, 0; blue, 0 }  ,fill opacity=1 ] (268.58,277.74) .. controls (269.98,277.74) and (271.11,278.73) .. (271.1,279.95) .. controls (271.1,281.17) and (269.97,282.15) .. (268.58,282.15) .. controls (267.18,282.15) and (266.05,281.16) .. (266.06,279.94) .. controls (266.06,278.72) and (267.19,277.74) .. (268.58,277.74) -- cycle ;
\draw  [fill={rgb, 255:red, 0; green, 0; blue, 0 }  ,fill opacity=1 ] (102.41,277.56) .. controls (103.81,277.56) and (104.93,278.55) .. (104.93,279.77) .. controls (104.93,280.99) and (103.8,281.98) .. (102.4,281.98) .. controls (101.01,281.97) and (99.88,280.98) .. (99.88,279.77) .. controls (99.89,278.55) and (101.02,277.56) .. (102.41,277.56) -- cycle ;
\draw    (269.63,441.31) .. controls (308.25,398.72) and (304.61,315.23) .. (269.16,280.01) ;
\draw   (294.2,363.3) -- (296.71,358.31) -- (299.57,363.12) ;
\draw    (102.4,279.33) .. controls (139.33,239.71) and (224.33,235.71) .. (268.58,279.5) ;
\draw   (187.86,250.62) -- (182.67,248.57) -- (187.56,245.88) ;
\draw   (589.95,283.5) -- (589.67,444.13) -- (423.49,443.96) -- (423.77,283.33) -- cycle ;
\draw    (423.77,283.32) .. controls (383.71,326.07) and (388.22,408.94) .. (423.49,443.96) ;
\draw   (504.53,281.03) -- (509.58,283.41) -- (504.53,285.79) ;
\draw   (505.03,441.67) -- (510.07,444.05) -- (505.02,446.42) ;
\draw   (420.71,368.4) -- (423.44,363.98) -- (426.15,368.4) ;
\draw   (392.99,369.31) -- (395.26,364.64) -- (398.4,368.76) ;
\draw   (586.87,363.79) -- (589.64,359.4) -- (592.31,363.85) ;
\draw  [fill={rgb, 255:red, 0; green, 0; blue, 0 }  ,fill opacity=1 ] (423.49,441.75) .. controls (424.89,441.75) and (426.02,442.74) .. (426.01,443.96) .. controls (426.01,445.18) and (424.88,446.17) .. (423.49,446.17) .. controls (422.09,446.17) and (420.96,445.18) .. (420.97,443.96) .. controls (420.97,442.74) and (422.1,441.75) .. (423.49,441.75) -- cycle ;
\draw  [fill={rgb, 255:red, 0; green, 0; blue, 0 }  ,fill opacity=1 ] (589.67,441.93) .. controls (591.06,441.93) and (592.19,442.92) .. (592.19,444.14) .. controls (592.18,445.36) and (591.05,446.35) .. (589.66,446.34) .. controls (588.26,446.34) and (587.14,445.35) .. (587.14,444.13) .. controls (587.14,442.91) and (588.27,441.93) .. (589.67,441.93) -- cycle ;
\draw  [fill={rgb, 255:red, 0; green, 0; blue, 0 }  ,fill opacity=1 ] (589.95,281.29) .. controls (591.34,281.29) and (592.47,282.28) .. (592.47,283.5) .. controls (592.47,284.72) and (591.34,285.71) .. (589.94,285.71) .. controls (588.55,285.71) and (587.42,284.72) .. (587.42,283.5) .. controls (587.42,282.28) and (588.56,281.29) .. (589.95,281.29) -- cycle ;
\draw  [fill={rgb, 255:red, 0; green, 0; blue, 0 }  ,fill opacity=1 ] (423.78,281.12) .. controls (425.17,281.12) and (426.3,282.11) .. (426.3,283.33) .. controls (426.3,284.55) and (425.16,285.53) .. (423.77,285.53) .. controls (422.38,285.53) and (421.25,284.54) .. (421.25,283.32) .. controls (421.25,282.1) and (422.38,281.11) .. (423.78,281.12) -- cycle ;
\draw    (591,444.87) .. controls (629.62,402.27) and (625.97,318.78) .. (590.52,283.57) ;
\draw   (615.56,366.86) -- (618.08,361.87) -- (620.94,366.67) ;
\draw    (423.77,283.33) .. controls (460.7,243.71) and (545.7,239.71) .. (589.95,283.5) ;
\draw   (505.03,250.03) -- (510.08,252.41) -- (505.03,254.79) ;

\draw (272.66,224.03) node  [rotate=-359.9]  {$z$};
\draw (236.61,204.86) node  [font=\scriptsize,rotate=-359.9]  {$3\pi $};
\draw (113.62,199.58) node  [font=\scriptsize,rotate=-359.9]  {$2\pi $};
\draw (233.37,44.87) node  [font=\scriptsize,rotate=-359.9]  {$4\pi $};
\draw (110.09,80.61) node  [font=\scriptsize,rotate=-359.9]  {$2\pi $};
\draw (89.04,182.15) node  [font=\scriptsize,rotate=-359.9]  {$4\pi $};
\draw (124.88,46.01) node  [font=\scriptsize,rotate=-359.9]  {$2\pi $};
\draw (127.9,68.26) node  [font=\scriptsize,rotate=-359.9]  {$\pi $};
\draw (89.41,87.55) node  [font=\scriptsize,rotate=-359.9]  {$2\pi $};
\draw (85.78,139.5) node    {$3$};
\draw (138.94,139.67) node    {$4$};
\draw (177.28,37.67) node    {$1$};
\draw (145.17,225.37) node [anchor=north west][inner sep=0.75pt]    {$DH( 1,2,3,4)$};
\draw (224.94,138.67) node    {$2$};
\draw (250.41,70.89) node  [font=\scriptsize,rotate=-359.9]  {$2\pi $};
\draw (254.58,184.99) node  [font=\scriptsize,rotate=-359.9]  {$4\pi $};
\draw (597.19,222.81) node  [rotate=-359.9]  {$z$};
\draw (598.12,182.6) node  [font=\scriptsize,rotate=-359.9]  {$4\pi $};
\draw (576.65,200.64) node  [font=\scriptsize,rotate=-359.9]  {$4\pi $};
\draw (438.15,198.36) node  [font=\scriptsize,rotate=-359.9]  {$\pi $};
\draw (557.9,43.65) node  [font=\scriptsize,rotate=-359.9]  {$2\pi $};
\draw (435.62,69.89) node  [font=\scriptsize,rotate=-359.9]  {$2\pi $};
\draw (413.57,180.92) node  [font=\scriptsize,rotate=-359.9]  {$4\pi $};
\draw (449.41,44.78) node  [font=\scriptsize,rotate=-359.9]  {$4\pi $};
\draw (573.93,68.04) node  [font=\scriptsize,rotate=-359.9]  {$\pi $};
\draw (599.44,83.5) node  [font=\scriptsize,rotate=-359.9]  {$2\pi $};
\draw (413.95,86.32) node  [font=\scriptsize,rotate=-359.9]  {$2\pi $};
\draw (410.31,138.28) node    {$3$};
\draw (503.98,137.44) node    {$1$};
\draw (601.98,139.61) node    {$2$};
\draw (501.81,36.44) node    {$4$};
\draw (469.7,224.15) node [anchor=north west][inner sep=0.75pt]    {$DI( 1,2,3,4)$};
\draw (275.99,449.15) node  [rotate=-359.9]  {$z$};
\draw (276.92,408.94) node  [font=\scriptsize,rotate=-359.9]  {$4\pi $};
\draw (255.95,426.97) node  [font=\scriptsize,rotate=-359.9]  {$\pi $};
\draw (129.45,430.19) node  [font=\scriptsize,rotate=-359.9]  {$4\pi $};
\draw (248.2,287.49) node  [font=\scriptsize,rotate=-359.9]  {$2\pi $};
\draw (114.92,412.22) node  [font=\scriptsize,rotate=-359.9]  {$4\pi $};
\draw (115.21,292.62) node  [font=\scriptsize,rotate=-359.9]  {$\pi $};
\draw (257.73,304.37) node  [font=\scriptsize,rotate=-359.9]  {$2\pi $};
\draw (278.24,309.83) node  [font=\scriptsize,rotate=-359.9]  {$2\pi $};
\draw (225.78,365.28) node    {$1$};
\draw (280.78,365.94) node    {$2$};
\draw (236.7,269.54) node  [font=\scriptsize,rotate=-359.9]  {$2\pi $};
\draw (128.21,270.67) node  [font=\scriptsize,rotate=-359.9]  {$4\pi $};
\draw (180.61,262.33) node    {$4$};
\draw (131.11,364.67) node    {$3$};
\draw (154,450.04) node [anchor=north west][inner sep=0.75pt]    {$DJ( 1,2,3,4)$};
\draw (469.9,452.34) node [anchor=north west][inner sep=0.75pt]    {$DK( 1,2,3,4)$};
\draw (597.36,452.7) node  [rotate=-359.9]  {$z$};
\draw (598.28,412.49) node  [font=\scriptsize,rotate=-359.9]  {$4\pi $};
\draw (576.81,430.53) node  [font=\scriptsize,rotate=-359.9]  {$\pi $};
\draw (438.32,428.25) node  [font=\scriptsize,rotate=-359.9]  {$4\pi $};
\draw (558.07,273.54) node  [font=\scriptsize,rotate=-359.9]  {$4\pi $};
\draw (435.79,299.77) node  [font=\scriptsize,rotate=-359.9]  {$\pi $};
\draw (413.74,410.81) node  [font=\scriptsize,rotate=-359.9]  {$4\pi $};
\draw (449.58,274.67) node  [font=\scriptsize,rotate=-359.9]  {$2\pi $};
\draw (574.1,297.93) node  [font=\scriptsize,rotate=-359.9]  {$2\pi $};
\draw (599.61,313.39) node  [font=\scriptsize,rotate=-359.9]  {$2\pi $};
\draw (414.11,316.21) node  [font=\scriptsize,rotate=-359.9]  {$2\pi $};
\draw (410.48,368.17) node    {$3$};
\draw (504.14,367.33) node    {$4$};
\draw (602.14,369.5) node    {$2$};
\draw (501.98,266.33) node    {$1$};

\end{tikzpicture}

%% file: diagrams/cell_graph.tex
\tikzset{every picture/.style={line width=0.75pt}} 

\begin{tikzpicture}[x=0.75pt,y=0.75pt,yscale=-1,xscale=1]

\draw  [fill={rgb, 255:red, 200; green, 200; blue, 200 }  ,fill opacity=1 ] (203.72,164.48) -- (166.22,142.82) -- (203.72,121.17) -- cycle ;
\draw  [fill={rgb, 255:red, 200; green, 200; blue, 200 }  ,fill opacity=1 ] (274.6,86.81) -- (237.1,65.16) -- (274.6,43.51) -- cycle ;
\draw  [fill={rgb, 255:red, 200; green, 200; blue, 200 }  ,fill opacity=1 ] (274.6,244.81) -- (237.1,223.16) -- (274.6,201.51) -- cycle ;
\draw  [fill={rgb, 255:red, 200; green, 200; blue, 200 }  ,fill opacity=1 ] (376.6,145.66) -- (339.1,167.31) -- (339.1,124.01) -- cycle ;
\draw  [fill={rgb, 255:red, 200; green, 200; blue, 200 }  ,fill opacity=1 ] (429.04,39.16) -- (452.52,62.65) -- (429.04,86.13) -- (405.56,62.65) -- cycle ;

\draw  [fill={rgb, 255:red, 200; green, 200; blue, 200 }  ,fill opacity=1 ] (523.1,165.81) -- (485.6,144.16) -- (523.1,122.51) -- cycle ;
\draw    (111.16,143) -- (166.22,142.82) ;
\draw    (237.1,65.16) -- (203.72,121.17) ;
\draw    (405.56,62.65) -- (274.6,43.51) ;
\draw  [fill={rgb, 255:red, 200; green, 200; blue, 200 }  ,fill opacity=1 ] (429.81,193.32) -- (453.29,216.81) -- (429.81,240.29) -- (406.32,216.81) -- cycle ;

\draw    (274.6,86.81) -- (339.1,124.01) ;
\draw    (203.72,164.48) -- (237.1,223.16) ;
\draw    (274.6,201.51) -- (339.1,167.31) ;
\draw    (376.6,145.66) -- (485.6,144.16) ;
\draw    (452.52,62.65) -- (523.1,122.51) ;
\draw    (453.29,216.81) -- (523.1,165.81) ;
\draw    (274.6,244.81) -- (406.32,216.81) ;

\draw  [fill={rgb, 255:red, 200; green, 200; blue, 200 }  ,fill opacity=1 ]  (81.6,134.16) -- (110.6,134.16) -- (110.6,152.16) -- (81.6,152.16) -- cycle  ;
\draw (96.1,143.16) node   [align=left] {$\displaystyle Cyl$};
\draw (191.22,142.82) node    {$U$};
\draw (263.25,63.71) node    {$T$};
\draw (429.14,61.21) node    {$Q$};
\draw (430.6,216.71) node    {$P$};
\draw (262.91,221.71) node    {$S$};
\draw (352.92,144.21) node    {$R$};
\draw (511.92,142.71) node    {$N$};
\draw (128.04,118.2) node [anchor=north west][inner sep=0.75pt]    {$DA$};
\draw (188.84,84.2) node [anchor=north west][inner sep=0.75pt]    {$DB$};
\draw (190.84,189) node [anchor=north west][inner sep=0.75pt]    {$DC$};
\draw (336.04,33.8) node [anchor=north west][inner sep=0.75pt]    {$DE$};
\draw (303.64,84.6) node [anchor=north west][inner sep=0.75pt]    {$DF$};
\draw (328.84,211) node [anchor=north west][inner sep=0.75pt]    {$DG$};
\draw (284.04,168.2) node [anchor=north west][inner sep=0.75pt]    {$DH$};
\draw (480.44,69) node [anchor=north west][inner sep=0.75pt]    {$DI$};
\draw (421.64,125.4) node [anchor=north west][inner sep=0.75pt]    {$DJ$};
\draw (466.84,173.8) node [anchor=north west][inner sep=0.75pt]    {$DK$};
\draw (240.04,124) node [anchor=north west][inner sep=0.75pt]    {$\begin{pmatrix}
1 & 2 & 3 & 4\\
4 & 3 & 2 & 1
\end{pmatrix}$};
\draw (368.84,89.2) node [anchor=north west][inner sep=0.75pt]    {$\begin{pmatrix}
1 & 2 & 3 & 4\\
1 & 3 & 4 & 2
\end{pmatrix}$};
\draw (365.64,160.8) node [anchor=north west][inner sep=0.75pt]    {$\begin{pmatrix}
1 & 2 & 3 & 4\\
1 & 3 & 4 & 2
\end{pmatrix}$};
\draw (218.04,131.6) node [anchor=north west][inner sep=0.75pt]    {$\circlearrowleft $};
\draw (346.84,96.4) node [anchor=north west][inner sep=0.75pt]    {$\circlearrowleft $};
\draw (346.44,170.4) node [anchor=north west][inner sep=0.75pt]    {$\circlearrowleft $};

\end{tikzpicture}